\documentclass[microtype]{gtpart} 
\usepackage{amssymb,epsfig,amsmath,latexsym,amsthm}
\usepackage{graphicx}
\usepackage[mathscr]{euscript}
\usepackage{overpic}
\usepackage{enumerate}
\usepackage[left=30mm,right=25mm,top=30mm,bottom=40mm]{geometry}
\usepackage[colorinlistoftodos]{todonotes}
\usepackage{hyperref}
\usepackage{color}

\DeclareMathOperator{\Isom}{Isom}

\DeclareMathOperator{\vol}{vol}

\DeclareMathOperator{\area}{area}

\DeclareMathOperator{\im}{im}

\DeclareMathOperator{\Diff}{Diff}
\DeclareMathOperator{\PSL}{PSL}
\DeclareMathOperator{\SL}{SL}
\DeclareMathOperator{\inte}{int}
\DeclareMathOperator{\fram}{frame}
\DeclareMathOperator{\shaved}{shaved}
\DeclareMathOperator{\core}{core}
\DeclareMathOperator{\imp}{im}
\DeclareMathOperator{\rep}{re}

\newcommand{\ov}[1]{\overline{#1}}
\newcommand{\closure}[1]{\overline{#1}}

\newcommand{\bC}{\mathbb C}

\newcommand{\bH}{\mathbb H}

\newcommand{\bZ}{\mathbb Z}

\newcommand{\bE}{\mathbb E}

\newcommand{\bR}{\mathbb R}

\newcommand{\cA}{\mathcal A}
\newcommand{\cF}{\mathcal F}

\newcommand{\cK}{\mathcal K}

\newcommand{\cS}{\mathcal S}
\newcommand{\cT}{\mathcal T}
\newcommand{\cU}{\mathcal U}

\newcommand{\sA}{\mathscr A}
\newcommand{\sB}{\mathscr B}
\newcommand{\sC}{\mathscr C}

\newcommand{\sG}{\mathscr G}

\newcommand{\sK}{\mathscr K}

\newcommand{\sP}{\mathscr P}
\newcommand{\sT}{\mathscr T}

\newcommand{\mF}{\mathcal{F}}
\newcommand{\mH}{\mathcal{H}}
\newcommand{\mO}{\mathcal{O}}
\newcommand{\mT}{\mathcal{T}}
\newcommand{\fb}{\mathfrak{b}}
\newcommand{\fa}{\mathfrak{a}}
\newcommand{\fc}{\mathfrak{c}}

\newcommand{\partialinfty}{\partial_\infty}
\newcommand{\Sinfty}{\partial_\infty \bH^3}
\newcommand{\iso}{\cong}

\newcommand{\sep}{:}
\newcommand{\gdspace}{\mathfrak D}

\newcommand{\la}{\langle}
\newcommand{\ra}{\rangle}

\newcommand{\al}{\alpha}

\newcommand{\eps}{\epsilon}

\newcommand{\Regina}{\textsc{Regina}}
\newcommand{\NE}{\mbox{\textsc{NE}}}
\newcommand{\connectsum}{\ \#\ }
\newcommand{\homeo}{\approx}
\newcommand{\freeproduct}{\ast}
\newcommand{\isomorphic}{\simeq}

\newcommand{\zahlen}{\mathbb{Z}}
\newcommand{\reals}{\mathbb{R}}
\newcommand{\complex}{\mathbb{C}}
\newcommand{\SFS}{Seifert-fibered space}
\newcommand{\hinfp}{H_{\mathfrak p}}
\newcommand{\hwinfp}{H_w({\mathfrak p})}
\newcommand{\pp}{\mathfrak p}
\newcommand{\ww}{\mathfrak w}

\newcommand{\tr}{\mathrm tr}

\newcommand{\volumeBound}{2.62}

\newcommand\rsmraise[1]{%
  \ifx#1\displaystyle .8\else
    \ifx#1\textstyle .8\else
      \ifx#1\scriptstyle .6\else
        .45%
      \fi
    \fi
  \fi}
  
\newcommand\rsetminusaux[2]{\mspace{-4mu}
  \raisebox{\rsmraise{#1}\depth}{\rotatebox[origin=c]{-20}{$#1\smallsetminus$}}
 \mspace{-4mu}
}

\newcommand\rsetminus{\mathbin{\mathpalette\rsetminusaux\relax}}

\newcommand{\dirsep}{/}
\newcommand{\dirprefix}{}

\newcommand{\filename}[1]{\texttt{#1}}

\usepackage{thmtools, thm-restate}

\newtheorem{theorem}{Theorem}[section]
\newtheorem{proposition}[theorem]{Proposition}
\newtheorem{conjecture}[theorem]{Conjecture}
\newtheorem{lemma}[theorem]{Lemma}
\newtheorem{corollary}[theorem]{Corollary}
\newtheorem{definition}[theorem]{Definition}

\theoremstyle{remark}
\newtheorem{remark}[theorem]{Remark}
\newtheorem{remarks}[theorem]{Remarks}

\newtheorem{notation}[theorem]{Notation}

\newtheorem{problem}[theorem]{Problem}

\newtheorem*{claim}{Claim:}
\newtheorem*{theorem*}{Theorem}

\numberwithin{equation}{section}

\let\emph\relax 
\DeclareTextFontCommand{\emph}{\bfseries\em}

\usepackage[T1]{fontenc}
\usepackage{inconsolata}
\usepackage{listings}
\usepackage{algorithm}
\usepackage[noend]{algpseudocode}

\newcommand{\bcusp}{\mathfrak B}
\newcommand{\fcusp}{\mathfrak F}
\newcommand{\pspace}{\mathscr P}
\newcommand{\dspace}{\mathscr D}
\newcommand{\cN}{\mathcal N}
\newcommand{\cV}{\mathcal V}

\title{Hyperbolic 3-manifolds of low cusp volume}

\author[Gabai]{David Gabai}
\address{Department of Mathematics \\ Fine Hall \\ Princeton University \\ Princeton, NJ 08544}
\email{gabai@math.princeton.edu}
\urladdr{https://www.math.princeton.edu/people/david-gabai}
\author[Haraway]{Robert Haraway}
\email{bobbycyiii@fastmail.com}
\urladdr{https://bobbycyiii.github.io}
\author[Meyerhoff]{Robert Meyerhoff}
\address{Department of Mathematics \\ Maloney Hall \\ Boston College \\ Chestnut Hill, MA 02467}
\email{robert.meyerhoff@bc.edu}
\urladdr{https://sites.google.com/bc.edu/robert-meyerhoff}
\author[Thurston]{Nathaniel Thurston}
\author[Yarmola]{Andrew Yarmola}
\address{Department of Mathematics \\ Fine Hall \\ Princeton University \\ Princeton, NJ 08544}
\email{yarmola@math.princeton.edu}
\urladdr{https://web.math.princeton.edu/~yarmola/}

\begin{document}
\begin{abstract}
We classify the complete hyperbolic 3-manifolds admitting a maximal cusp of volume at most 2.62.
We use this to show that the figure-8 knot complement is the unique 1-cusped hyperbolic 3-manifold with nine or more non-hyperbolic fillings;
to show that the figure-8 knot complement and its sister are the unique hyperbolic 3-manifolds with minimal volume maximal cusps; and
to extend results on determining low volume closed and cusped hyperbolic 3-manifolds.
\end{abstract}
\maketitle

\tableofcontents


\section{Introduction}

A central goal in low dimensional topology is to relate various topological, combinatorial, geometric and algebraic structures associated to manifolds.
For hyperbolic 3-manifolds, there is the long-standing Thurston, Hodgson-Weeks, Matveev-Fomenko Hyperbolic Complexity Conjecture: \textit{The complete low-volume hyperbolic 3 -manifolds can be obtained by filling cusped hyperbolic 3 -manifolds of small topological complexity.}
One of the challenges of this open-ended conjecture is to identify the right notion of complexity, especially when restricting this conjecture to a particular class of manifolds. 
This paper addresses the version of that conjecture for complete hyperbolic 3-manifolds with maximal cusps of low volume. 
All manifolds in this paper are orientable.

In the context of closed hyperbolic 3-manifolds, the existence of sufficiently thick tubes about short geodesics is crucial to many results in the subject. 
Examples include, the Smale conjecture for hyperbolic 3-manifolds \cite{Gabai:2001} and the important inequality (slightly updated to reflect \cite{GabaiTrnkova:2015}), that if $X \neq \textrm{Vol3}$ is a closed hyperbolic 3-manifold, then it is obtained by filling a 1-cusped hyperbolic 3-manifold $Y$ such that  $\vol(Y)<3.0177 \vol(X)$ \cite{AgolCullerShalen:2006}.
This inequality is based on earlier works of Agol \cite{Agol:2002} and Agol-Dunfield \cite{AgolStormThurstonDunfield:2007} and uses Perelman's work on Ricci flow \cite{Perelman:2002}, \cite{Perelman:2003}.
It is needed to show that the Weeks manifold is the unique closed hyperbolic 3-manifold of least volume, \cite{GMM:2009}.
These results rely on the $\log(3)/2${\it -theorem} of \cite{GMT:2003}.
The sharp form of that result \cite{GabaiTrnkova:2015} asserts that with the exception of the manifold known as Vol3, any closed hyperbolic 3-manifold $X$ has a $\log(3)/2$ tube about some geodesic; furthermore, with five exceptions this geodesic can be taken to be any shortest one.

Recall that in the 1960's Gregory Margulis proved that each end of a complete finite-volume hyperbolic 3-manifold $Y$ is homeomorphic to $T^2\times (1, \infty)$.
These ends are called \emph{cusps}.
Note, the term cusp in this paper always refers to a \emph{rank-two} cusp, even if the manifold in question is not finite-volume.
Each cusp contains a properly embedded \emph{horocusp}, which is a region isometric to a quotient of $\inte(H_\infty)$, where $H_\infty = \{(x,y,t) \sep t \geq 1\} \subset \bH^3$ in the upper half-space model and we quotient by a group of translations of $(x,y)$ isomorphic to $\bZ\oplus \bZ$, see \cite{Thurston:1997}.
Further, after isometrically changing coordinates, the end can be maximally enlarged to the \emph{maximal horocusp} $\kappa$. 
Here, $\inte(\kappa)$ is embedded in $Y$ and $\partial \kappa$ has finitely many self-tangencies.
One can in addition assume that a $\pi_1(Y)$-translate of $H_\infty$ is a horoball $H_0$ tangent to $H_\infty$ in the upper halfspace model at $(0,0,1) \in \bH^3$. 
A basic fact is that $\vol(\kappa)=\area(\partial \kappa)/2$, called the \textit{cusp volume}.

The works of J{\o}rgenson and Thurston from the 1970's demonstrate the close relation between thick tubes about geodesics and maximal horocusps of complete manifolds \cite{Thurston:1978}.
Indeed, J{\o}rgenson showed that after passing to a subsequence, any infinite sequence $Y_i$ of distinct complete hyperbolic 3-manifolds of uniformly bounded volume limits geometrically to a complete manifold $Y_\infty$, where thicker and thicker tubes of the $Y_i$'s limit to cusps of $Y_\infty$. 
Conversely, Thurston showed that given a cusp of a complete finite-volume 3-manifold $Y$ and $\epsilon>0$, then all but finitely many fillings on that cusp produce hyperbolic manifolds geometrically $\epsilon$-close to $Y$.
The papers, \cite{NeumannZagier:1985}, \cite{HodgsonKerckhoff:2005}, \cite{HodgsonKerckhoff:2008},\cite{FuterPurcellSchleimer:2019} have made this connection more explicit and quantitative.

The following is the main technical result of this paper.

\begin{restatable}{Thm}{main}\label{thm:main} Let $Y$ be a complete hyperbolic 3-manifold with cusps\footnote{Cusps are always rank-two in this paper. Note, additional boundary types are allowed in the hypothesis.}.  Either each cusp of $Y$ has an embedded horocusp of volume $>2.62$ or $Y$ is obtained by filling one of the 16 manifolds listed in Table \ref{table:ancestral}.

\begin{table}[h]
    \begin{center}
    \begin{tabular}{|c|c|c|c|c|c|c|c|}\hline
      s596  & s647  & s774  & \textbf{s776} & s780  & s782  & s785  & v2124 \\\hline
      \underline{v2355} & v2533 & v2644 & v2731   & v3108 & v3127 & v3211 & v3376 \\ \hline
    \end{tabular}
    \end{center}
    \caption{An ancestral set for all hyperbolic 3-manifolds with $\vol(\kappa) \leq 2.62$.}\label{table:ancestral}
 \end{table}
\end{restatable}

\begin{remarks}\hfill
\begin{enumerate}[(i)]
\item With the exception of s776, the so called 3-cusped \emph{magic manifold}, all the other manifolds in Table \ref{table:ancestral} have two cusps. 
\item In the language of \cite{GMM:2011}, s596, s647, s774, s780, s776, s785 are Mom-3 manifolds, while s782, v2124, v2533, v2644, v2731, v3108, v3127, v3211, v3376 are Mom-4 manifolds; see \cite{Haraway:2020}.  Only v2355 is neither, but it is a Mom-5 manifold.
\end{enumerate}
\end{remarks}

The proof, outlined below, combines new geometric and combinatorial methods with rigorous computer assistance.  This result and byproducts of its proof have a number of applications to the theory of hyperbolic 3-manifolds, knot theory and 3-manifold topology.  To start with we answer the long-standing problem of identifying the complete hyperbolic 3-manifolds with minimal cusp volume.

\begin{restatable}{Thm}{figeight}\label{thm:fig8}The figure-8 knot complement and its sister are the complete hyperbolic 3-manifolds with a maximal horocusp of minimal volume. This volume is $\sqrt{3}$.\end{restatable}

\begin{remarks} \hfill
\begin{enumerate}[(i)]
\item These manifolds minimize cusp volume among all complete hyperbolic 3-manifold with torus cusps,  even among multi-cusped manifolds and complete manifolds of infinite volume.  

\item Results on this question go back to the 1980's. A lower bound of $\sqrt{3}/4$ for the volume of a maximal cusp was obtained in \cite{Meyerhoff:1985}. Adams improved this bound by a factor of 2 in \cite{Adams:1987}. There were no further improvements until Cao and Meyerhoff \cite{CaoMeyerhoff:2001}. Although they didn’t explicitly state it, their proof method is easily modified to obtain a bound of $3.35/2$ for the volume of a maximal cusp in a hyperbolic 3-manifold, which is within 0.0571 of the actual value of $\sqrt{3}$.
\end{enumerate}
\end{remarks}

In the 1970's, Bill Thurston showed that volume strictly decreases under filling \cite[Theorem 6.5.6]{Thurston:1978}. On the other hand we have,

\begin{restatable}{Thm}{change}\label{thm:change} Depending on the filling,  cusp volume can increase or decrease when filling one cusp of  m295 (aka L9n14 or $9^2_{50}$). However, the cusp volume always decreases after filling one cusp of {\normalfont m129}, the Whitehead link complement. Further, the cusp shapes for fillings of {\normalfont m129} are never rectangular.
\end{restatable}

\begin{remark} Cusp shape might not change under filling, see \cite{NeumannReid:1993}, \cite{Calegari:2001}.  See \cite{Purcell:2008} for results on the effect of cone deformations on cusp geometry.  \end{remark}

Let $\kappa$ be a maximal horocusp of the complete hyperbolic 3-manifold $Y = \bH^3/\Gamma$ and normalize so that $H_\infty$ and $H_0$ both cover $\kappa$ and are conjugate under $\Gamma$. Let $\bcusp=\la m,n,g \ra \leq \Gamma$, where $m,n$ are generators of $\{ \gamma \in \Gamma \sep \gamma \cdot H_\infty = H_\infty\}$ and $g$ is the conjugating element $g(H_\infty)=H_0$.  Adopting terminology from \cite{Agol:2000}, we call $\bcusp$ a (geometric) \emph{bicuspid} group. 
A remarkable result of Agol states:

\begin{theorem*}[\cite{Agol:2010}]\label{agol}
 If $\kappa$ is a maximal horocusp of the complete hyperbolic 3-manifold $Y$ and $\vol(\kappa)<\pi$, then any bicuspid group $\bcusp$ associated to $\kappa$ is finite index in $\pi_1(Y)$ and $\vol(Y) < \infty$.  In particular, there exists a relation $w(m,n,g)$ in $\bcusp$.\end{theorem*}

In contrast, without requiring Theorem \ref{agol}, we have the following corollary of Theorem \ref{thm:7neckl}, the complete statement of which is given later in the outline of the paper.
 
\begin{restatable}{Thm}{bicusp}\label{thm:bicusp}   If $\kappa$ is a maximal horocusp of the complete hyperbolic 3-manifold $Y$ and $\vol(\kappa) \le 2.62$,  then any bicuspid group $\bcusp$ associated to $\kappa$ is equal to $\pi_1(Y)$ and $\vol(Y) < \infty$.  Furthermore, there exists a relation $w(m,n,g)$ whose $g$-length is at most 7 and at least 4.\end{restatable}

Here, \textit{g-length} means the sum of the absolute values of the exponents of $g$ in $w$.

\begin{remark} Experimentally, the 1-cusped manifolds m135 and m136 have index-2 bicuspid groups. 
In both cases, the volume of their maximal cusps is $2\sqrt{2} \approx 2.828427...$\footnote{Though we are unable to find a reference for this number, it is computable by a straightforward analysis of the gluing equations of m135 and m136.}.\end{remark}

By a packing argument (see  \cite{Meyerhoff:1986}), if $Y$ has a maximal horocusp of volume $V$, then $\vol(Y)\ge (2 v_3 / \sqrt{3}) V \approx (1.17195...)V$, where $v_3$ is the volume of the regular ideal tetrahedron in $\bH^3$. 
It follows that either $\vol(Y)\ge 2.62\cdot(1.17195...)\approx 3.0705...$  or $Y$ is obtained by filling one of the manifolds of Table \ref{table:ancestral}. Thus, through an analysis of the fillings of the manifolds in Table \ref{table:ancestral}, we obtain:

\begin{restatable}{Thm}{onecusp}\label{thm:onecusp}  The 14 complete non-compact hyperbolic 3-manifolds with volume $\le 3.07$ are those listed in Table \ref{table:onecusp307}.

\begin{table}[H]
  \begin{center}
    \begin{tabular}{|c|c|c|c|c|c|c|}\hline
      {m003} & {m004} & {m006} & {m007} & {m009} & {m010} & {m011}\\\hline
      {m015} & {m016} & {m017} & {m019} & {m022} & {m023} & {m026}\\\hline
    \end{tabular}
  \end{center}
  \caption{One-cusped hyperbolic 3-manifolds with volume less than 3.07.}
  \label{table:onecusp307}
\end{table}
\end{restatable}

\begin{remarks} \hfill
\begin{enumerate}[(i)]
\item \cite{CaoMeyerhoff:2001} first identified the two non-compact manifolds of least volume.  
\item \cite{GMM:2011} identified the first 10 manifolds on this list as all those with volume $\leq 2.848$.
\end{enumerate}
\end{remarks}

A further analysis of Table \ref{table:ancestral} together with the above mentioned theorem of \cite{AgolCullerShalen:2006} yields, 

\begin{restatable}{Thm}{volume}\label{thm:volume}  The closed orientable hyperbolic three-manifolds of volume at most $1.01749$
  are 
  \begin{enumerate}[(i)]
 \item the Weeks-Matveev-Fomenko manifold, a.k.a. \texttt{m003(2,1)}, with volume $0.94271 \pm \epsilon,$
 \item the Meyerhoff manifold, a.k.a. \texttt{m004(5,1)}, named in \cite{meyername},  with volume $0.98137 \pm  \epsilon,$
 \item Vol3, a.k.a. \texttt{m007(3,1)}, named in \cite{GMT:2003}, with volume $1.01499 \pm \epsilon$,
 \end{enumerate}
  where $\epsilon = 10^{-5}.$
\end{restatable}


\begin{remark} The Weeks manifold was first shown to be minimal volume in \cite{GMM:2011}. 
There, it was shown that the minimal volume manifold was obtained by filling a Mom-$\le$ 3 manifold.  
An analysis of fillings of such manifolds was carried out in \cite{Milley:2009}.

We also note that this further supports a conjecture of Martelli \cite{Martelli:2006} that the closed orientable hyperbolic 3-manifolds of smallest Matveev complexity are those of smallest volume. There are four such manifolds and we have identified three of them.
\end{remark} 

In his revolutionary work in the 1970's on the geometry of the figure-8 knot complement, Bill Thurston showed that it has exactly 10 non-hyperbolic Dehn fillings.  
Thurston's result spawned a tremendous amount of work towards understanding hyperbolic and non-hyperbolic fillings of hyperbolic 3-manifolds, e.g. see the surveys in \cite{Gordon:1995}, \cite{Boyer:2002}, \cite{LackenbyMeyerhoff:2013}.  
In particular, there are infinitely many 1-cusped hyperbolic manifolds with six non-hyperbolic fillings.  
By 1998, there were (resp. eight, two, zero, one)) known 1-cusped manifolds with (resp. seven, eight, nine, ten) exceptional surgeries.  
This led Cameron Gordon to conjecture in \cite{Gordon:1995} that if $M$ be a hyperbolic 3-manifold with boundary a torus, then $M$ has at most $8$ non-hyperbolic fillings unless $M$ is the figure eight knot exterior.

 
The Gromov-Thurston $2\pi$-theorem showed that if a manifold $Y'$ is obtained by filling the cusped manifold $Y$ along a curve $\gamma \subset \partial \kappa$ of length > $2\pi$ then $Y'$ has a complete metric of negative sectional curvature (hence a hyperbolic structure by Perelman).  
The shortest curve on $\partial \kappa$ has length $\ge 1$ and Thurston used this fact to show that the number of non-hyperbolic fillings is bounded by 48.  
But coupling a length lower bound with a maximal horocusp area lower bound can lead to an improvement, and this was carried out by Bleiler and Hodgson ([BH92]) — using Adams’s area lower bound — to get a bound of 24 on the number of non-hyperbolic fillings. 
Had the Cao-Meyerhoff area bound been available at that time, the non-hyperbolic-filling bound would have been reduced to 14. 

In 2000, Agol and Lackenby independently improved $2\pi$ to 6 with the the weaker conclusion that $Y'$ has word hyperbolic fundamental group (again Perelman implies that $Y$ is hyperbolic).  
Coupling the 6-theorem with the Cao-Meyerhoff bound led to a non-hyperbolic filling bound of at most 12 — tantalizingly close to the figure-eight case bound of 10. As the 6-theorem is sharp, attention was focused on area bound improvements.  
However, a new area bound of 3.7 in the ``non-Mom'' case in \cite{GMM:2009} did not improve upon the non-hyperbolic filling bound of 12. 
The \cite{GMM:2009} area result seemed quite strong at the time, so potential for progress on the non-hyperbolic filling bound seemed bleak.  
But Lackenby and Meyerhoff were able to prove in \cite{LackenbyMeyerhoff:2013} the 10 bound without using the area bound improvements by generalizing the 6-theorem and exploiting Mom technology.

 
In \cite{Agol:2000}, \cite{Agol:2010} Ian Agol proved  that If $Y$ is a 1-cusped hyperbolic 3-manifold with 9 or more non-hyperbolic fillings, then the maximal cusp of $Y$ has volume $>18/7 = 2.57...  $. 
Since $18/7 < 2.62$ all the 1-cusped hyperbolic 3-manifolds with a maximal cusp of volume $\le 18/7$ arise from filling the manifolds in Table \ref{table:ancestral}. 
The work of \cite{MartelliPetronio:2006} shows that the figure-8 knot complement is the unique 1-cusped manifold obtained by filling $s776$ with $9$ or more non-hyperbolic fillings.  
An analysis of the fillings of the 2-cusped manifolds in Table \ref{table:ancestral}, making use of the Agol-Lackenby 6-theorem and normal surface theory, also yields the same conclusion, thereby proving Gordon's conjecture.  We note that another version of this analysis was carried out by Crawford in his thesis \cite{Crawford} based on a preliminary version of our Theorem \ref{thm:main}.

\begin{restatable}{Thm}{gordon}\label{thm:gordon} The figure-8 knot exterior is the unique 1-cusped hyperbolic 3-manifold with nine or more non-hyperbolic fillings.\end{restatable}

\begin{remark} For manifolds of two or more cusps it was known by Gordon \cite{Gordon:1998} that the distance between non-hyperbolic fillings on a fixed cusp, with the other cusp(s) unfilled, is at most 5, hence by \cite{Agol:2000}, there are at most 8 non-hyperbolic fillings on that cusp with the other cusp(s) unfilled.\end{remark}

We now outline the proof of Theorem \ref{thm:main}. 
Let $\kappa$ be a maximal horocusp of the complete hyperbolic 3-manifold $Y$ with $\vol(\kappa)\le 2.62$.  
Let $\bcusp = \la m,n,g \ra$ be a bicuspid group corresponding to some self-tangency of $\kappa$.  
Note, $\bH^3/\bcusp$ has a maximal cusp isometric to $\kappa$. 
The proof of our main result breaks into three steps.

\begin{restatable}{Thm}{params}\label{thm:param}
If $\bcusp$ is a geometric bicuspid group whose preferred maximal horocusp $\kappa$ has volume $\le 2.62$, then $\bcusp$ has a relation $w(m,n,g)$ with $g$-length at most $7$ and at least $4$.\end{restatable}

The proof of Theorem \ref{thm:param} uses a rigorous computer assisted validation that analyzes the compact parameter space of elements $m,n,g \in \mathrm{PSL}(2,\bC)$ where $m,n$ correspond to parabolics fixing $\infty \in \hat\bC$, whose fundamental domain restricted to $\partial H_\infty$  has area at most $5.24$ and $g$ is a transformation taking $H_0$ to $H_\infty$.  

The next step transforms this group theoretic statement into a topological one.  

\begin{definition}  A compact 3-manifold $W$ is a necklace-$k$ manifold if it is built from a handle structure on $T^2\times [0,1]$ with a single 1-handle attached to $T^2\times 1$ and a single 2-handle running over the 1-handle $k$ times.  
We say that $W$ is a {\it full} necklace-$k$ manifold if, in addition, the attaching curve for the 2-handle cuts $T^2 \times 1 \setminus D_1 \cup D_2$ into disks, where $D_1 \cup D_2$ is the attaching locus of the $1$-handle. 
\end{definition}

Notice that a necklace-$k$ manifold has a preferred boundary torus corresponding to $T^2 \times 0$. 


\begin{restatable}{Thm}{existneckl}\label{thm:7neckl}
If $\kappa$ is a maximal horocusp for a complete hyperbolic 3-manifold $Y$ and $\bcusp$ is an associated bicuspid group a with a relation of $g$-length $ k \le 7$, then there is an embedded full necklace-$k'$ manifold $W \subset \ov{Y}$ with $k' \leq k$ and the preferred cusp of $W$ is the boundary torus corresponding to $\kappa$. 
Furthermore, each other component of $\partial W$ either cuts off a solid torus or a cusp end in $Y$.
\end{restatable}

In other words, $Y$ is obtained from $W$ by filling at most two boundary components with solid tori and then passing to the interior. 
Note that $Y$ having finite volume is not part of the hypothesis, though follows from the conclusion.
Notice that Theorem \ref{thm:bicusp} is a direct corollary of Theorem \ref{thm:7neckl}.

Theorems \ref{thm:param} and \ref{thm:7neckl} imply:

\begin{restatable}{Thm}{hcc}\label{hcc} (Hyperbolic complexity conjecture for low volume maximal cusps)\\If $Y$ is a complete finite volume hyperbolic 3-manifold with a maximal horocusp of volume $\le 2.62$, then $Y$ is obtained by filling a full necklace-$k$ manifold for $k\le 7$.\end{restatable}

Theorem \ref{thm:param} implies that in $\bH^3$, seen as the universal cover of $Y$, we see \emph{necklaces} of distinct horoballs $B_0, B_1, \ldots, B_k$ with $B_k = B_0$ and $k \leq 7$, where $B_i$ is tangent to $B_{i+1}$ and there is an element of $\bcusp$ that transforms $\{B_i, B_{i+1}\}$ to $\{H_0, H_\infty\}$ set-wise, indices taken mod $k$.
Now, if such a necklace of horoballs bounds a 2-disc $D$ whose $\pi_1(Y)$-orbit consists of pairwise disjoint discs, then we can construct a handle decomposition where the tangency between $H_0$ and $H_\infty$ corresponds to the 1-handle and $D$ corresponds to the 2-handle.  
Technical issues that need to be addressed include showing that the necklaces are unknotted, unblocked and unlinked.  
Repeated use is made of the condition $k\le 7$.  
Indeed, bicuspid subgroups to both m135 and m136 demonstrate, at least experimentally, the necessity of this hypothesis, for the minimal necklaces for those manifolds have eight beads and the disc spanning this necklace is blocked by a translate of the pair $(H_0, H_\infty)$, see Figure \ref{fig:m135}.  
The passage from the relation $w$ to the disc $D$ may require finding a new disc $D'$ corresponding to a word $w'$ whose $g$-length is at most the $g$-length of $w$.  
The theory of necklaces is quite beautiful in its own right and is of independent interest.

\begin{figure}[h]
  \begin{center}\begin{overpic}[scale=.22]{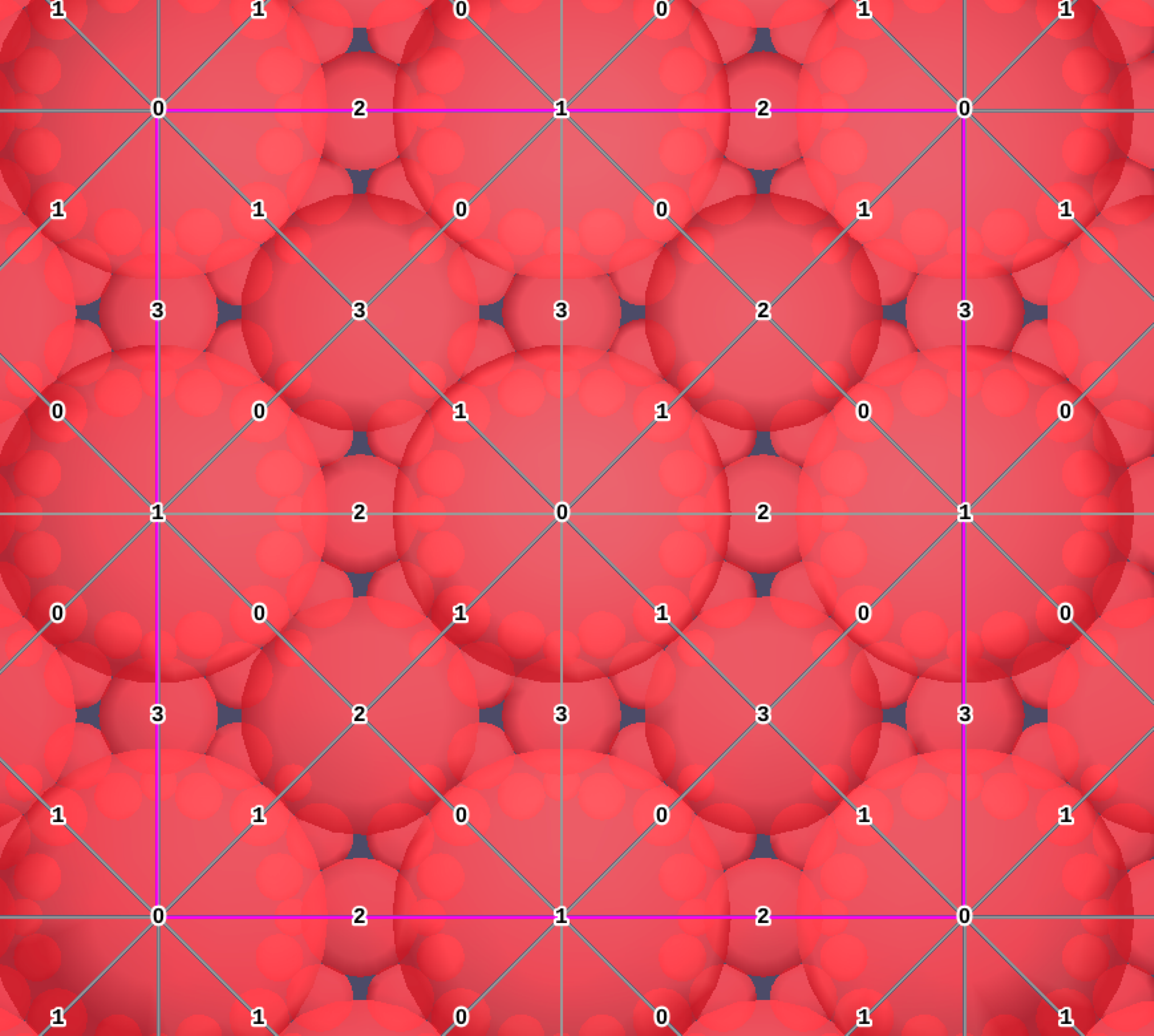}
    \end{overpic}
  \end{center}
  \caption{Blocked necklace in {m135} corresponding to the $0$-edge orbit. }
  \label{fig:m135}
\end{figure}


We now indicate why a bicuspid group is index one in $\pi_1(Y)$, when Y has a maximal cusp of volume $\le 2.62$.  
By construction, the $W$ of Theorem \ref{thm:7neckl} has the property that $m, n$ are the generators of $\pi_1(T)$, where $T$ is the preferred torus component, and $g$ is an element of $\pi_1(W)$ coming from the 1-handle.  
Since $Y$ deformation retracts to a possibly empty filling of $W$, it follows that  $\la m,n, g \ra = \bcusp =  \pi_1(Y)$. 

The \textit{tunnel number} of a compact 3-manifold $Y$ with non-empty boundary is one less than the Heegaard genus of $Y$ as a union of compression bodies.  
This definition naturally extends to complete finite-volume hyperbolic 3-manifolds.
When $\partial Y>1$, (resp. has multiple cusps), this may depend on the partition of the boundary components (resp. cusps) into two sets.
In what follows, the partition consists of a designated cusp in one set and all the remaining cusps in the other.  

\begin{corollary} If $Y$ is a cusped hyperbolic 3-manifold with a maximal horocusp of volume $\le 2.62$, then $Y$ is tunnel number one.  \end{corollary}


Note that {m135} and {m136}, the two examples where the best $g$-length relator appears to be $8$ and the bicuspid group appears to be index 2, have tunnel number $2$ as their homology is rank 3 and Berge's program \texttt{heegaard} \cite{heegaard} finds a genus 3 splitting.

We now explain how we identify the manifolds in Table \ref{table:ancestral}. 
The fact that $W$ in Theorem \ref{thm:7neckl} carries a full necklace $k$-structure implies that it has a natural combinatorial description as a {\it dipyramid} gluing, see Section \ref{subsec:bead}. The study of such structures should be of independent interest. Since $k \leq 7$, we are able to enumerate all such manifolds and recognize which admit nonelementary embeddings into hyperbolic manifolds.


\begin{restatable}{Thm}{neckl}\label{thm:neckl}
Let $W$ be a full necklace-$k$ manifold that has a nonelementary embedding into a hyperbolic manifold $Y$ such that the preferred cusp of $W$ maps to a cusp of $Y$.
If $k \leq 7$, then $Y$ is a Dehn filling of one of the manifolds in Table \ref{tbl:initAncestral}.
  \begin{table}[H]
    \begin{center}
      \begin{tabular}{|c|c|c|c|c|c|c|}\hline
        \textbf{m125} & \textbf{m129} & \textbf{m202} & \textbf{m203} & \textbf{m292} & \textbf{m295} & \textbf{m328} \\ \hline
        \textbf{m359} & \textbf{m367} & \textbf{s441} & \textbf{s443} & s596 & s647 & s774 \\ \hline
        \textbf{s776} & s780 & s782 & s785 & \textbf{v1060} & v2124 & v2355 \\ \hline
        v2533 & v2644 & v2731 & v3108 & v3127 & v3211 & v3376 \\ \hline
      \end{tabular}
      \caption{Initial ancestral set with Dehn fillings of s776 are in \textbf{bold}.}\label{tbl:initAncestral}
    \end{center}
  \end{table}
\end{restatable}

The proof of Theorem \ref{thm:neckl} uses an observation that any nonelementary embedding of $W$ in $Y$ can be re-embedded to make $Y$ a Dehn filling of $W$, see Lemma \ref{lem:hypDehn} and Appendix  \ref{proof:hypDehn}.
Then, using normal surface enumeration in Regina \cite{Regina}, we show that the collection of all hyperbolic Dehn fillings of full necklace-$k$ manifolds with $k \leq 7$ must arise as fillings of the above 28 manifolds.

Finally, Theorem \ref{thm:main} follows from Theorem \ref{thm:neckl} and Theorem \ref{hcc} after noting many manifolds in Table \ref{tbl:initAncestral} are fillings of s776.

\subsection{Acknowledgments}

This work was partially supported by a grant from the Simons Foundation (\#228084 to Robert Meyerhoff) and 
by a grant from the National Science Foundation (DMS-1308642 to Robert Meyerhoff). The third author was partially supported by these same grants. The first author was partially supported by NSF grants DMS-1006553, DMS-1607374 and DMS-2003892. The fifth author was partially supported as a Visiting Student Research Collaborator with DMS-1006553 and Postdoctoral Researcher with DMS-1607374. The authors also acknowledge the support and resources of the Polar Computing Cluster at the Mathematics Department and the Program in Applied and Computational Mathematics at Princeton University.

\renewcommand{\i}{\mathfrak{i}}
\renewcommand{\j}{\mathfrak{j}}
\renewcommand{\k}{\mathfrak{k}}

\section{The parameter space}\label{sec:param}

In this section, we explain the computer-assisted proof of Theorem \ref{thm:param}. 
We will leave the computational details for Appendix \ref{sec:verify} and focus on the mathematical setup. 
The goal is to analyze low-cusp-volume hyperbolic 3-manifolds by studying the marked 3-generator bicuspid subgroups associated to their cusps. The space of all possible such marked groups with $\vol(\kappa) \leq 2.62$ will lie in, up to isometry, a compact parameter space $\mathscr{P}$ of three complex dimensions. Rigorous computer assisted analysis of this parameter space will show that all such groups admit a relators of $g$-length at most $7$.

\subsection{The upper half-space model}
We will use the upper-half-space model $\bH^3 = \{ (z,t) \in \complex \times \reals \sep t > 0\}$. Consider $q \in \bH^3$ in quaternion notation $q = x + \i y + \j t$, where $\i ^2 = \j^2 = \k^2 = \i \j \k = -1$. 
Identifying  $\Isom^+(\bH^3) \iso \PSL(2, \complex)$, the action on $q = x + \i y + \j t$ is given by
\[ \pm \begin{pmatrix} a & b\\ c & d\end{pmatrix} \cdot q = (a q + b)(c q + d)^{-1}.\]
Notice that for $z \in \complex$, we have that $z \, \j = \j \, \ov{z}$, so one computes
\begin{equation}\label{eq:pure_action}
  \pm\begin{pmatrix} a & b\\ c & d\end{pmatrix} \cdot t \, \j = (a \, t \,  \j + b)(c \, t \, \j  + d)^{-1} 
    = \frac{t^2 \, a \, \ov{c} + b \, \ov{d}}{t^2 |c|^2 + |d|^2} + \frac{t \, \j}{t^2 |c|^2 + |d|^2}.
\end{equation}
In particular, $\j$ is mapped to a point of Euclidean height $1 / (|c|^2 + |d|^2)$, which we will use later.

\subsection{The compact parameter space}
Let $Y$ be an oriented cusped hyperbolic 3-manifold.
We identify $\Gamma = \pi_1(Y)$ with a discrete torsion-free group of orientation-preserving M\"obius transformations acting on $\bH^3$.
Fix a maximal parabolic subgroup $\Pi \iso \bZ^2 < \Gamma$ and recall that the conjugacy class of $\Pi$ corresponds to a cusp of $Y$.
Fix a maximal horoball $H$ in $\bH^3$ with the property that $\Pi = \{ \gamma \in \Gamma \sep int(\gamma \cdot H) \cap int(H) \neq \emptyset\}$. We say $\kappa = H/\Pi$ is the {\it maximal horocusp} of the chosen cusp in $Y$.  
It follows that there is an element $G \in \Gamma \smallsetminus \Pi $ such that $G \cdot H  \cap H = p$ is a point.
We call the subgroup $\bcusp = \langle \Pi, G \rangle$ a {\it bicuspid subgroup} of $\Gamma$. 


\begin{definition}
  A \emph{bicuspid marking} is the tuple $(\Pi, H, G)$, where $\Pi$ is a rank-two parabolic subgroup of $PSL(2, \complex)$, $H$ is a horoball preserved by $\Pi$, and $G \cdot H$ is a horoball tangent to $H$.
  The associated bicuspid group is $\bcusp = \langle \Pi, G \rangle$.
  The area of a bicuspid marking is the area of the (Euclidean) torus $\partial \kappa$.
  Two bicuspid markings  $(\Pi_1, H_1, G_1)$ and $(\Pi_2, H_2, G_2)$ are isomorphic if there is a group isomorphism $\psi : \bcusp_1 \to \bcusp_2$ with $\psi(\Pi_1) = \Pi_2$, $\psi(G_1) \cdot H_2$ tangent to $H_2$, and $\psi^{-1}(G_2) \cdot H_1$ tangent to $H_1$.
  
A \emph{geometric bicuspid marking} is one that arises from an oriented cusped hyperbolic 3-manifold. In particular, $\bcusp$ is discrete, torsion-free, and $\inte(H/\Pi)$ embeds and is maximal in $\bH^3/\bcusp$.
\end{definition}

\begin{remarks} \hfill
\begin{enumerate}[(i)]
\item While $\Pi$ is always discrete and isomorphic to $\bZ \oplus \bZ$, we don't assume that $\bcusp$ is discrete in a general bicuspid marking.
\item An isomorphism between geometric bicuspid markings gives rise to an isometry between $\bH^3/\bcusp_1$ and $\bH^3/\bcusp_2$ and vice versa.
\end{enumerate}
\end{remarks}

Suppose we have a bicuspid marking $(\Pi, H, G)$. For computational reasons, we use a slightly different normalization than described in the introduction.
Working in the upper-half-space model of $\bH^3$, there is some isometry $\phi$ taking $H$ to a horoball at infinity.
Replacing $(\Pi, H, G)$ with $(\phi\Pi \phi^{-1}, \phi \cdot H, \phi G \phi^{-1})$, we may assume that every element of $\Pi$ is of the form $z \mapsto z + c$ for some $c \in \complex$, identifying $\Pi$ as an additive subgroup of $\complex$.
We can thus assign the length $|c|$ to elements $z \mapsto z + c$ of $\Pi$.
Let $M(z) = z + \lambda$ be a shortest nontrivial element of $\Pi$.
Conjugating by $z \mapsto \lambda^{-1} z$, we may assume the shortest-length nontrivial translation $M$ in $\Pi$ is $z \mapsto z + 1$. We may choose $N(z) = z + L$ to be a shortest element of $\Pi$ linearly independent of $M$ (i.e. {\it next-shortest}).
With $g$ the inverse of $G$, we have that $g \cdot H$ is also a horoball tangent to $H$, centered at some point $c \in \complex$.
Conjugating by $z \mapsto z - c$, we may assume $g \cdot H$ is centered at 0.
The horoball $G \cdot H$, called the {\it Adams horoball}, has center $P = G(\infty) \in \complex$.
These normalizations imply that $G(z) = P + 1/(S^2 z)$ for some complex number $S$, which controls the height and rotation of the Adams horoball.

\begin{definition}\label{def:bicuspid_marking}
  A \emph{bicuspid triple} is a triple $\mathfrak{p} = (P,S,L) \in \complex^3$ with $S  \neq 0$. It gives rise to M\"{o}bius transformations $M_\mathfrak{p}(z) = z + 1$, $N_\mathfrak{p}(z) = z + L$, and $G_\mathfrak{p}(z) = P + 1/(S^2 z)$. We name these parameters:
  $P$ is the parabolic parameter; $S$ is the (loxodromic) square-root parameter; and $L$ is the lattice parameter.
\end{definition}
 We represent the associated transformations in $\PSL(2,\complex)$  by matrices up to sign as follows:
\begin{equation*}
    G_\mathfrak{p} = \pm\begin{pmatrix} P S \, \i & \i /S \\ S\, \i  & 0 \end{pmatrix} \qquad
    g_\mathfrak{p} = \pm\begin{pmatrix} 0 & -\i /S \\ - S\, \i  & P S \, \i  \end{pmatrix}
 \end{equation*}
\begin{equation*}
    M_\mathfrak{p} = \pm\begin{pmatrix} 1& 1 \\ 0  & 1 \end{pmatrix} \qquad
    N_\mathfrak{p} = \pm\begin{pmatrix} 1 & L\\ 0  & 1  \end{pmatrix}.
\end{equation*}

To associate a bicuspid marking to $\mathfrak{p}$, we need to choose a horoball $\hinfp$ centered at $\infty$. Since we want $G_\mathfrak{p} \cdot \hinfp$ to be tangent to $\hinfp$, by equation (\ref{eq:pure_action}), we must have
\[\hinfp = \{ (z,t) \in \bH^3 \sep t > 1/|S|\}.\]
In our parametrization, we have $\area(\partial \kappa_\mathfrak{p}) = 2\vol(\kappa_\mathfrak{p}) = |S^2 \, \imp(L)|$. To $\mathfrak{p}$, we associate the bicuspid marking $(\Pi_\mathfrak{p}, \hinfp, G_\mathfrak{p})$, where $\Pi_\mathfrak{p} = \la M_\mathfrak{p}, N_\mathfrak{p}\ra$. Note, we often drop the subscript when it is clear from context.

Given a geometric bicuspid marking $(\Pi, H, G)$ of area $\leq 5.24$, we now show that there is a compact set $\mathscr{P} \subset \bC^3$ such that $(\Pi, H, G)$ is isometric to $(\Pi_\mathfrak{p}, \hinfp, G_\mathfrak{p})$ for some $\mathfrak{p} \in \mathscr{P}$.

 \newcommand{\Abound}{5.24}
\begin{definition}
Let $\mathscr{P}$ be the compact subset of $\complex^3$ defined by the following conditions:
\begin{enumerate}
\setcounter{enumi}{-1}
\item $|S| \geq1$
\item $\imp(S) \geq 0$, $\imp(L) \geq 0$, $\imp(P) \geq 0$, $\rep(P) \geq 0$
\item $-1/2 \leq \rep(L) \leq 1/2$
\item $|L| \geq 1$
\item $\imp(P) \leq \imp(L)/2$
\item $\rep(P) \leq 1/2$
\item $|S^2 \, \imp(L)| \leq \Abound$
\end{enumerate}
\end{definition}

\begin{proposition}\label{prop:compact}
 Suppose $(\Pi, H, G)$ is a geometric bicuspid marking of area at most $5.24$. Then, there is $\mathfrak{p} = (P,S,L) \in \mathscr{P}$ such that $(\Pi, H, G)$ is isometric to $(\Pi_\mathfrak{p}, \hinfp, G_\mathfrak{p})$.
\end{proposition}

\begin{proof}
 As above, we may assume $H$ is centered at $\infty$ and $\Pi = \la M, N \ra$, where $M(z) = z+1$ is a shortest-length generator and $N(z) = z + L$ is the next-shortest.
  Since $N$ is next-shortest, we have $|L| \geq 1$ and $-1/2 \leq \rep(L) \leq 1/2$.
  If $\imp(L) < 0$, then we replace $N$ with $z \mapsto z - L$, which is also next-shortest and has $\imp(-L) > 0$.

  Next, let $P \in \bC$ be the center $G(\infty)$ of the {\it Adams horoball} $G \cdot H$.
  By post-composing $G$ with elements of $\Pi$, we may choose $G$ such that $-\imp(L)/2 \leq \imp(P) \leq \imp(L)/2$ and  $-1/2 \leq \rep(P) \leq 1/2$.
  Reflecting across the $x$-axis creates isomorphic groups with isometric quotients.
  Thus we may assume $0 \leq  \imp(P) \leq \imp(L)/2$.
  Reflecting across the $y$-axis allows us to assume likewise that $0 \leq \rep(P) \leq 1/2$.
  Finally, changing $S$ to $-S$ leaves $G$ invariant, so we may assume $\imp(S) \geq 0$.
  This determines $P$, $S$, and $L$ and a bicuspid marking isomorphism. Properties  (1), (2), (4), and (5) clearly hold and (3) holds since $N$ is next-shortest. Further, since $\kappa_\mathfrak{p}$ is a maximal horocusp in $\bH^3 / \bcusp_\mathfrak{p}$, we know that $\hinfp$ has Euclidean height $\leq 1$. Thus, $1/|S| \leq 1$ and $(0)$ holds.
  
Finally, the area of $\partial \kappa_\mathfrak{p}$ is $|S^2 \im(L)|$, and this is at most $\Abound$, so (6) holds. Since all the inequalities are satisfied, $(P,S,L) \in \mathscr{P}$ as desired.

Note that the positive lower bounds for $|S|$ and $|\imp(L)|$ impose upper bounds for $|S|$ and $|\imp(L)|$ by using $|S^2 \, \imp(L)| \leq \Abound$.
Thus $\mathscr{P}$ is compact.
\end{proof}



\subsection{Discrete points}

Naturally, we are interested in points $\mathfrak{p} \in \complex^3$, where $\bcusp_\mathfrak{p}$ is discrete, torsion-free, and geometric.
To investigate these points it is convenient to introduce the following notions:

\begin{definition}
  The \emph{free bicuspid group} is $\fcusp = \langle M,N,G\ \mid \ [M,N]\rangle \simeq (\zahlen\oplus\zahlen)\ast\zahlen.$
  For all $\mathfrak{p} = (P,S,L) \in \complex^3$ with $S \neq 0$, the associated \emph{bicuspid representation} $\rho_\mathfrak{p}: \fcusp \to \PSL(2,\complex)$ is given by Definition \ref{def:bicuspid_marking} and has image $\bcusp_\mathfrak{p}$. 
 Let \[\dspace = \{ \mathfrak{p} \in \pspace \sep im(\rho_\mathfrak{p}) \text{ is discrete and  torsion-free}\}.\]
 Note that $\bcusp_\mathfrak{p}$ is automatically nonelementary for $\mathfrak{p} \in \dspace$.
 Let \[\gdspace = \{ \mathfrak p \in \dspace \sep \mathfrak{p} \text{ corresponds to a geometric bicuspid marking}\}\]
\end{definition}

The delicate distinction between $\gdspace$ and $\dspace$ is that, in general, even if $\bcusp_\mathfrak{p}$ is discrete and torsion-free, there could be an element $h \in \bcusp_\mathfrak{p}$ such that $h \cdot \hinfp \cap \hinfp$ has non-empty interior. Recall that $\hinfp$ is explicitly defined using the $S$ parameter. This means that $\mathfrak{p}$ does not arise in the image of the map in Proposition \ref{prop:compact}. We say that $\mathfrak{p} \in \dspace \setminus \gdspace$ are \emph{incorrectly marked} (i.e. non-geometric since $\hinfp$ is larger than the true maximal horoball of the group $\bcusp_\mathfrak{p}$ ).

\begin{definition}
For a word $w \in \fcusp$ and $\mathfrak{p} \in \pspace$, let \[w(\mathfrak{p}) = \rho_{\mathfrak{p}}(w) = \pm \begin{pmatrix} a_w(\mathfrak{p})& b_w(\mathfrak{p})\\ c_w(\mathfrak{p}) & d_w(\mathfrak{p})\end{pmatrix} \in \PSL(2,\bC).\]
Define \[\cU_w  = \{ \mathfrak{p} = (P,S,L) \in \pspace \sep |c_w(\mathfrak{p}) /S|  < 1\}\]
\end{definition}

\begin{lemma}[Large horoball]\label{lem:relator} If $\mathfrak{p} = (P,S,L) \in \gdspace \cap \cU_w$ for some $w \in \fcusp$ of nonzero $g$-length, then $\bcusp_\mathfrak{p}$ contains a relator of $g$-length at most the $g$-length of $w$.\end{lemma}
\begin{proof}  Let $\hwinfp$ be the image of $\hinfp$ under $w(\mathfrak{p})$. Recall that $\hinfp$ has Euclidean height $1/|S|$, so by \ref{eq:pure_action} $\hwinfp$ has Euclidean height $|S|/|c_w(\mathfrak{p})|^2$. Thus, $\hinfp$ and $\hwinfp$ intersect precisely when $|S|/|c_w(\mathfrak{p})|^2 > 1/|S|$, or equivalently $|c_w(\mathfrak{p}) /S|  < 1$.

Assume $\mathfrak{p} \in \gdspace \cap \cU_w$, then $\bcusp_\pp$ is geometric and $\hwinfp$ intersects $\hinfp$. But $\hinfp$ is maximal in the geometric $\bcusp_\pp$, so we must have $\hinfp = \hwinfp$ and $w(\mathfrak{p}) \in \Pi_\pp$. In particular, $m_\pp^p n_\pp^q w(\mathfrak{p})$ is the identity for some $p,q \in \bZ$ and has the same $g$-length as $w$.
\end{proof}

\begin{definition} Let \[ \cK_w =  \{ \mathfrak{p} \in \cU_w \sep |c(\mathfrak{p})| > 0 \text{ or } a(\mathfrak{p}) \neq \pm 1 \text{ or } d(\mathfrak{p}) \neq \pm 1\} \subset \cU_w\]
\end{definition}

\begin{corollary}[Killer word]\label{cor:killer} $\gdspace \cap \cK_w = \emptyset$ for all $ w\in \fcusp$.
\end{corollary}

\begin{proof} If $\pp \in \cK_w \subset \cU_w$, then $w(\pp)$ is not a parabolic fixing $\infty$. In particular, $w(\pp) \notin \Pi_\pp$, so  by the proof of Lemma \ref{lem:relator} we know that $\pp \notin \gdspace$.
\end{proof}

\begin{remark} It is possible that $\cK_w \cap \dspace$ is non-empty. However, all such points would be incorrectly marked and can be ignored because every geometric bicuspid group of interest will appear in $\gdspace$ by Proposition \ref{prop:compact}. 
\end{remark}

The task of proving Theorem \ref{thm:param} boils down to proving that there is a finite cover of $\gdspace \subset \cup_{i = 1}^n \cU_{w_i}$, where each $w_i$ is a word of $g$-length at most $7$. We describe the computational proof of this in the next section.

Before we move on, we need to define a few more notions that will be necessary for the proof of Theorem \ref{thm:fig8}. Since Theorem \ref{thm:fig8} uniquely identifies two manifolds, we need tools to find exact relators in bicuspid groups.

\begin{definition}
Define the variety of $w$ by $\cV_w  = \{ \mathfrak{p} \in \pspace \sep w(\mathfrak{p}) =\pm I\}$ 
  and consider the neighborhood $\cN_w$ of $\cV_w$ by \[\cN_w  = \{ \mathfrak{p} \in \pspace \sep  |c_w(\mathfrak{p})| < 1 \text{ and } |b_w(\mathfrak{p})| < 1\}.\]
\end{definition}

\begin{lemma}[Variety neighborhood]\label{lem:var_nbd}
For all nontrivial $w \in \fcusp$, if $\mathfrak{p} \in \dspace \cap \cN_w$ then $\mathfrak{p} \in \cV_w$. In particular, $w$ is a relator in $\bcusp_\mathfrak{p}$.
\end{lemma}

\begin{proof}  The Shimizu-Leutbecher Theorem states that if \[ \begin{pmatrix} 1 & 1\\ 0 & 1\end{pmatrix} \quad \text{and} \quad \begin{pmatrix} a & b\\ c & d\end{pmatrix}\] generate a discrete subgroup of $\PSL(2,\bC)$ and  $|c| < 1$, then we must have $c = 0$.
        For $\mathfrak{p} \in \dspace$, $\bcusp_\mathfrak{p}$ is discrete and torsion-free, in particular, $M_\mathfrak{p}$ and $w(\mathfrak{p})$ generate a discrete subgroup of $\PSL(2,\complex)$. 
Suppose $\mathfrak{p} \in \dspace \cap \cN_w$.
Dy definition $|c_w(\mathfrak{p})| < 1$, but discreteness implies $c_w(\mathfrak p) = 0$, so $w(\mathfrak{p}) \in \Pi_{\mathfrak{p}}$. Further, since $|b_w(\mathfrak{p})| < 1$ and $M_{\mathfrak{p}}$ is a shortest-length generator of $\Pi_\mathfrak{p}$, we must have $b_w(\mathfrak p) = 0$. Thus, $\mathfrak{p} \in \cV_w$. 
\end{proof}

Lastly, we remark that nontrivial $w \in \fcusp$ of $g$-length $\leq 3$ are never the identity in geometric bicuspid groups.

\begin{lemma}[$g$-length 3]\label{glen3} If $w \in \fcusp$ has nonzero $g$-length $\leq 3$ then $\gdspace \cap \cU_w = \emptyset.$
\end{lemma}
\begin{proof} It is clear that any word of $g$-length $1$ must move $H_\pp$ off of itself, so $\cU_w = \emptyset$. For $w$ of $g$-length $2$, it is easy to see that $\gdspace \cap \cU_w = \emptyset$ since $\bcusp_\pp$ is torsion-free for all $\pp \in \gdspace$. Finally, if $w$ has $g$-length $3$, then for any $\pp \in \gdspace \cap \cU_w$ we get a $3$-necklace where each tangency is in the same orthoclass. This was shown to be impossible in \cite{GMM:2009}.
\end{proof}

\subsection{Computational setup}

To prove Theorem \ref{thm:param}, we take the same approach as \cite{GMT:2003}. 
In fact, all of our arithmetic code, as described in Appendix \ref{sec:verify}, is the {\it exact same code} used in \cite{GMT:2003}. This arithmetic relies on using 1-jets approximations with error and round-off error for computations, which often outperforms interval arithmetic. We briefly describe this arithmetic in Appendix \ref{ssec:affine_jets}.

To start, we place the compact parameters space $\sP$ into a large box
\[\sB = \{ (x_0, x_1, x_2, x_3, x_4, x_5) \in \bR^6 \sep |x_i| \leq 2^{(19-i)/6}\},\]
where each subsequent side of the box is $1/\sqrt[6]{2}$ times the size of the previous. 
The embedding is given by $L = x_3 + \i x_0$, $S = x_4 + \i x_1$, and $P = x_5 + \i x_2$. 
Since $|x_1| \leq 8$ and $|x_6| \leq 2^{13/6} \approx 4.49$, we see that $\sB$ is large enough to contain all of $\sP$.

The constant side ratio of $1/\sqrt[6]{2}$ for $\sB$ is chosen specifically such that if we cut along the $1^\text{st}$ dimension, then the two resulting boxes have the same constant side ratio as $\sB$. 
Cutting those along the $2^\text{nd}$ dimension yields 4 boxes that still have the same constant side ratio. 
This is reminiscent of the $1/\sqrt{2}$ side ratio of A-series printer paper, except for $6$-dimensional boxes instead of rectangles in the plane.

This behavior yields two advantages. 
First, the (sub-)boxes of $\sB$ obtained by cutting in this manner stay relatively ``round,'' making computational corrections for rounding error less dramatic. 
Second, this subdivision allows us to encode these special (sub-)boxes in binary as \emph{boxcodes}. 
For example, a boxcode $0$ corresponds to the box $\sB_0$ obtained by cutting $\sB$ in half along the $1^\text{st}$ dimension and taking the resulting box on the ``left.''
We fix some preferred orientation on $\bR^6$ to define ``right'' and ``left.''
The box $\sB_{01}$ corresponds to cutting $\sB_0$ in half along the $2^\text{nd}$ dimension and taking the piece on the ``right.''
For deeper boxes, we keep cutting along the next dimension and after cutting along the $6^\text{th}$ dimension, we start again with the $1^\text{st}$. 
For a boxcode $\fb$, we let $\sB_\fb$ denote the corresponding box.

To analyze all of $\sB$, we build a binary tree $\sT$ corresponding to the boxcodes. 
Each terminal node $\fb$ of this tree will correspond to a box with an associated killer word or a necklace word (see the next paragraph). 
The terminal nodes of the tree give a subdivision of $\sB$ into boxes of different sizes. 
The reason for this approach is that the number of boxes needed is quite large\footnote{1,394,524,064 terminal boxes to be precise.}, if we were to divide $\sB$ into boxes all of the same size, our computation and verification time would exponentially increase.

A killer or necklace word at a terminal node $\fb$ defines a collection of inequalities that must be checked to hold at every point of $\sB_\fb$. 
These inequalities are only composed of basic algebraic operations $\pm, \times, /, \sqrt{\,}$, and absolute values. 
They could in theory be checked by hand given an unreasonably large amount of time. 
A word $w$ is called a \emph{killer word} for a box $\sB_\fb$ if the encoded inequalities prove that  $\sB_\fb \subset \cK_w$. By Corollary \ref{cor:killer}, this condition corresponds either to indiscreetness or an incorrect marking of the bicuspid subgroup.
 
A \emph{necklace word} $w$ at a terminal box $\sB_\fb$, corresponds to a word of $g$-length $\leq 7$ and inequalities proving that $\sB_\fb \subset \cU_w$. In particular, if $\sB_\fb \subset \cU_w$, then any geometric bicuspid triples in $\sB_\fb$ must have a relator of the same $g$-length by Lemma \ref{lem:relator}. In summary,

\begin{proposition}\label{thm:gLength7}
 An evaluation of the inequalities encoded in the binary tree $\sT$ proves that there is a finite collection of terminal boxes $\sB_{\fb_i}$ with associated necklace words $w_i$, such that $\gdspace \subset \bigcup_i \sB_{\fb_i} \subset \bigcup_i \cU_{w_i}$. Further, each $w_i$ has $g$-length at most $7$.
\end{proposition}

\begin{proof} The proof is contained in the data and code available at \cite{verify-code} and can be checked using the \texttt{verify} program. The tree encodes a cover $\sB = \bigcup_i \sB_{\fb_i} \cup \bigcup_j  \sB_{\fa_j}  \cup \bigcup_k  \sB_{\fc_k}$, where to each boxcode $b_i$, we associate a necklace word $w_i$, to each boxcode $\fa_j$ a killer word $w_j'$, and to each boxcode $\fc_k$ a failed boundary condition. A failed boundary condition verification shows that one of the inequalities of Proposition \ref{prop:compact} fails over the \emph{entire} box $\sB_{\fc_k}$, showing that $\sB_{\fc_k} \cap \sP = \emptyset$ for each $ k$. The program also verifies killer word inequalities for $w_j'$ over $ \sB_{\fa_j}$ for each $j$, which proves that $\gdspace \subset \bigcup_i \sB_{\fb_i}$. Finally, we also check necklace word inequalities for $w_i$ over $\sB_{\fb_i}$ and the $g$-length of $w_i$ for each $i$, which shows that $\sB_{\fb_i} \subset \bigcup_i \cU_{w_i}$ where $g$-length of $w_i \leq 7$ (and is nonzero) for all $i$. 

The conditions are checked at all terminal nodes by traversing the binary tree from the root node (in depth-first order). A successful traversal of the tree guarantees that we have obtained a cover of $\sB$, which corresponds to the root node. See the README of \cite{verify-code} on how to use \texttt{verify}. \end{proof}

The results in this section can be restated as

\params*

\renewcommand{\i}{\mathfrak{i}}
\renewcommand{\j}{\mathfrak{j}}
\renewcommand{\k}{\mathfrak{k}}

\section{Parameter space applications}\label{sec:paramapp}

In this section, we prove and discuss the results

\change*

\figeight*

The first result only relies on the structure of the parameter space itself, while the second result relies on the ability of the parameter space decomposition to rigorously find relators in Kleinian groups and being able to use those relators to identify specific hyperbolic 3-manifolds. The latter is done by building an isomorphism to the canonical SnapPy group presentation of the manifold in question via a character variety argument.

\subsection{Cusp volume change under filling}

By studying specific examples of $g$-length $\leq 7$ relators, we can find some interesting patterns. In particular, if a cusped manifold has a relator $w$ in its fundamental group then so do all of its Dehn filings. In the case of m129, the Whitehead link complement, such a relator defines a variety where the cusp volume is forced to be maximized at the discrete point corresponding to m129 and is strictly smaller for all discrete points on that variety.

First, we look at a special variety in our parameter space, which we will then show contains a geometric bicuspid group corresponding to a cusp of m129.

\begin{proposition}\label{whvar} If $gMGGMgN$ is a relator in a bicuspid group $\bcusp_\pp$ where $\pp = (P,S,L) \in \bC^3$ with $S \neq 0$, then $P = L/2$, $ S^2= -4/L$ and, if $\bcusp_\pp$ is geometric, the maximal cusp area of $\bcusp_\pp$ is $4 | \imp(L)/L|$.\end{proposition}

\begin{proof} Using the parametrization \begin{equation*}
    G = \pm\begin{pmatrix} P S \, \i & \i /S \\ S\, \i  & 0 \end{pmatrix} \qquad
    g = \pm\begin{pmatrix} 0 & -\i /S \\ - S\, \i  & P S \, \i  \end{pmatrix}
 \end{equation*}
\begin{equation*}
    M= \pm\begin{pmatrix} 1& 1 \\ 0  & 1 \end{pmatrix} \qquad
    N= \pm\begin{pmatrix} 1 & L\\ 0  & 1  \end{pmatrix}.
\end{equation*}

We get that \[gMGGMgN = \pm\begin{pmatrix} 
 P S^2+1 & P S^2 (L-P)+L \\
 S^2 \left(P S^2+2\right) & P S^4 (L-P)+S^2 (2 L-P)+1 \\
\end{pmatrix} = \begin{pmatrix}  1 & 0 \\ 0 & 1 \end{pmatrix}\]

Recall that the cusp area is given by $|S^2 \, \imp(L)|$ if $\bcusp_\pp$ is geometric. Since $S \neq 0$, the lower left entry tells us that $PS^2 + 2 = 0$. The top right entry then gives $P = L/2$. Since the top left entry is now $-1$, we must have \[-1 = P S^4 (L-P)+S^2 (2 L-P)+1 = 1 + (L S^2)/2.\]
So $S^2= -4/L$ and the area is $|S^2 \, \imp(L)| = 4 | \imp(L)/L| \leq 4$ when $\bcusp_\pp$ is geometric.
\end{proof}

Turning to $m129$, we prove the following:

\begin{lemma}\label{whparam} The triple $\ww = (P,S,L) = (\i, 1 + \i, 2 \i) \in \sP$ corresponds to a geometric bicuspid marking of m129. Further, $gMGGMgN = \pm I$ at $\ww$ and $\bH^3/\bcusp_\ww$ is isometric to m129. Lastly, the other cusp of m129 corresponds to the conjugacy class of $\la gm, ggmGmg\ra$. \end{lemma}

\begin{proof} In \cite{GuillouxWill:2019} and \cite{Wielenberg:1978}, the authors describe how the gluing of the regular ideal octahedron with vertices $0, -1, \i,  -1+\i, (-1+\i)/2,$ and $\infty$ gives m129. The face pairings generate the fundamental group \[ \Gamma  = \la u, t_2, w_1 \ra  \text{ where }  u =   \pm\begin{pmatrix} 1& \i \\ 0  & 1 \end{pmatrix},   t_2 = \pm\begin{pmatrix} 1& 2 \\ 0  & 1 \end{pmatrix},   w_1 =  \pm\begin{pmatrix}  1 & 0 \\ -1-\i &1  \end{pmatrix}.\]
From the gluing of the regular ideal octahedron, as shown in \cite{GuillouxWill:2019}, it is easy to see that the maximal cusp has a self-tangency corresponding $w_1u^{-1}$ or, equivalently, the conjugate $u^{-1}w_1$. In particular, under conjugation by $\psi(z)=\i z + (1-\i)/2$, we get $M \leftrightarrow u$, $n \leftrightarrow t_2$ and $G \leftrightarrow u^{-1} w_1$:
 \begin{equation*}
 M_\ww = \pm\begin{pmatrix} 1& 1 \\ 0  & 1 \end{pmatrix} = \psi^{-1}u\psi \quad  N_\ww = \pm\begin{pmatrix} 1& 2\i \\ 0  & 1 \end{pmatrix} =  \psi^{-1}t_2^{-1}\psi,  \end{equation*}
  \begin{equation*}G_\ww = \pm\begin{pmatrix}
 -1-\i & (1+\i)/2 \\
 -1+\i & 0 \\
\end{pmatrix} = \psi^{-1} u^{-1}w_1 \psi.
 \end{equation*}
This shows that $\bcusp_\ww$ is a geometric bicuspid marking for $m129$ and $\bH^3/\bcusp_\ww$ is isometric to m129. It is easy to check that $gMGGMgN$ becomes the identity at $\ww$.

Lastly, we need the peripheral elements of the other cusp of m129. In  \cite{GuillouxWill:2019}, they give the isomorphism between $\Gamma$ and the SnapPy presentation $\pi_1 =  \la a, b \mid aaaBBabAAAbbAB \ra$ where $a \leftrightarrow u^{-1} w_1 \leftrightarrow G$ and $b \leftrightarrow u^{-2}w_1 \leftrightarrow mG$. SnapPy reports the peripheral curves of the other cusp as $AAb$ and $AAAbbA$. Under our identifications, these become $ggmG$ and $gggmGm$. Conjugating by $g$ gives the peripheral subgroup $\la gm, ggmGmg\ra$. \end{proof}

Combining Proposition \ref{whvar} and Lemma \ref{whparam}, we see that every Dehn filling of m129 appears on the variety of $gMGGMgN$. We can now exploit the fact that the peripheral subgroup of the other cusp becomes loxodromic under filling to prove:

\begin{proposition}\label{whdehn} Any hyperbolic Dehn filling of one cusp of  {\normalfont m129} has cusp area $< 4$ and the cusp shape is never rectangular.\end{proposition}
\begin{proof} Since m129 has a symmetry interchanging the cusps, we can work with one of the cusps. By Proposition \ref{whvar}, the variety $V = V_{gMGGMgN}$ forces $P = L/2$, $S^2 = -4/L$ and, at a geometric marking, maximal cusp area $4|\im(L)/L|$. By Lemma \ref{whparam}, any Dehn filling of the other cusp of m129 appears as $\bcusp_\pp$ for some $\pp \in V$. Note, even though near the triple $\ww$ from Lemma \ref{whparam} one can guarantee that $\bcusp_\pp$ is geometric, it might not be true for arbitrary fillings. However, since $G_\pp \in \bcusp_\pp$, we always have that $H_\pp$ contains the true maximal cusp. Thus, the maximal cusp area for $\bH^3/\bcusp_\pp$ is at most $4 | \imp(L)/L| \leq 4$ for any discrete $\pp \in V$. 

Since $4 | \imp(L)/L| = 4$ if and only if the cusp is rectangular, it remains to show that if $\la gm, ggmGmg\ra$ is loxodromic at $\pp \in V$ corresponding to a Dehn filling then we cannot have a cusp where $L = \i t$ for some $t \in \bR_{> 0}$. Assume otherwise and let $L = t\i$. If $\bcusp_\pp$ is a $(p,q)$-Dehn filling, then $x^py^q = \pm I$ where $x = gm$ and $y = ggmGmg$. One checks that \[\pm \tr(y) = \tr(x^2) - 4 \quad \text{and} \quad \tr(x) = (\sqrt{t/2} +\sqrt{2/t})+i (\sqrt{t/2}-\sqrt{2/t})\] for all $t \in \bR_{> 0}$. As $x$ and $y$ are commuting loxodromics, they are simultaneously diagonalizable with eigenvalues $\lambda_x^{\pm1}$ and $\lambda_y^{\pm 1}$. Computing the eigenvalues from the traces one sees that $\log|\lambda_y| = \pm 2 \log|\lambda_x|$. It follows that if $x^py^q = \pm I$ for relatively prime $(p,q)$, then $(p,q) = (\pm 2, 1)$, which are non-hyperbolic fillings. Thus, for any hyperbolic Dehn filling, we have $L \neq \i t$ for any $t \in \bR_{>0}$ and the maximal cusp area is less than $4$.
\end{proof}

Using SnapPy, we can now find an example where this behavior does not occur.

{\bf Proof of Theorem \ref{thm:change}.} SnapPy is able to rigorously estimate cusp shapes and volumes. The maximal cusp volume of m295 is around $2.516534$, the maximal cusp volume of m295(1,9) is around 2.54523, and that of m295(2,1) (aka m004) is $\sqrt{3} \approx 1.7321$. Along with Proposition \ref{whdehn}, the proof is complete.
\qed

\begin{remark} Experimentally, it appears that m125, m202, m203 behave similarly to m129 with maximal cusp volume always decreasing. The behavior of m295 appears to be more common, with m292 (a volume sibling of m295) and many others having fillings where maximal cusp volume increases for many of the fillings.
\end{remark}

\subsection{Smallest cusp volume manifolds}

From the parameter space perspective, the key ingredient in the proof of Theorem \ref{thm:fig8} is the variety neighborhood lemma.
 Recall that Lemma \ref{lem:var_nbd} allows one to rigorously prove that if $\sB_\fb \subset \cN_w$ and $\mathfrak{p} \in \sB_\fb \cap \dspace$, then $w$ is the identity in $\bcusp_\mathfrak{p}$.
  In what follows, we will show that the parameter space of marked bicuspid groups with cusp area $\leq 3.65$ decomposes into boxes that are killed and boxes that lie in \emph{two} variety neighborhoods simultaneously.
   In particular, this will give us two relators at each discrete point, which will allow us to determine the presentation of these manifolds on the nose.

\begin{proposition}\label{prop:identify} Let
\[\mathfrak D_{3.65} = \{ \mathfrak p \in \gdspace \mid \bcusp_\mathfrak{p} \text{ is a geometric bicuspid group of cusp area } \leq  3.65\}\]
then for every $\mathfrak{p} \in \mathfrak D_{3.65}$, the group $\bcusp_\mathfrak{p}$ admits at least one of the pairs of the relators in Tables \ref{table:relm003}, \ref{table:relm004}, or \ref{table:rel3}.

\begin{table}[h]
    \begin{center}
    \begin{tabular}{|lll|}\hline
GNgMgNG & and & mGnGmGmnGmnG  \\\hline
MnGmGMngg & and & MgggMgNg  \\\hline
mGGmgMNg & and & mmnGGmGmGG \\\hline
mgNgmGG & and & mGmGGmnGG \\\hline
mnGGmngMg & and & NGmGGGmG \\\hline
mnGNggNG & and & GGmnGmGmnG \\\hline
nGNmggNmG & and & NgMgNggg \\\hline
ngMgnGG & and & mNGmGGGmG \\\hline
    \end{tabular}
        \end{center}
      \caption{Possible relator pairs that correspond to m003.}
    \label{table:relm003}
 \end{table}
 \begin{table}[h]
    \begin{center}
       \begin{tabular}{|lll|}\hline
MNgmGMGmg & and & MgMGmgNgmG \\\hline
MnGmgMgmG & and & gmGMgMGmg \\\hline
MnGmgMgmG & and & mgmGMgNgMG \\\hline
mGMnGmgMg & and & MGMgmGnGmg \\\hline
mnGMgmgMG & and & MgmGGmgMG \\\hline
    \end{tabular} 
     \end{center}

     \caption{Possible relator pairs that correspond to m004.}
    \label{table:relm004}
 \end{table}
 \begin{table}[h]
    \begin{center}
        \begin{tabular}{|lll|}\hline
GmGGmG & and & nGmnGnGmnG\\\hline
gMGmnGMg & and & mGmGmGmG\\\hline
mGGmGG & and & mgMGmgMG\\\hline
mggmGmnG & and & nGmGnGmG\\\hline
MgMgMgMg & and & GmgMNgmG\\\hline
MGMgMGMg & and & GnGGnG\\\hline
mgMGmgMG & and & mGGmGG\\\hline
MGnGMgMg & and & mnGGmnGG\\\hline
MgNgMGmG & and & MNgMNgMNgMNg\\\hline
MGnGMGnG & and & mnGGmnGG\\\hline
mnGGmnGG & and & mGmgNgmG\\\hline
MnGGMnGG & and & nGmGnGmG\\\hline
nGmgMgmG & and & GGGG\\\hline
    \end{tabular}
    \end{center}
    \caption{Relator pairs that cannot arise in a manifold.}
    \label{table:rel3}
 \end{table}
\end{proposition}

\begin{proof} The proof is contained in the data and code available at \cite{verify-code} and can be checked using the \texttt{identify} program. See the README of \cite{verify-code} on how to use \texttt{identify}. The terminal box of this smaller tree falls into four categories: killer word boxes, $g$-length $\leq 3$ boxes, out-of-bounds boxes, and variety intersection boxes. A box $\sB_\fb$ falls into one of first two categories if $\sB_\fb \subset \cK_w$ or $\sB_\fb \subset \cU_{w}$ and $g$-length of $w \leq 3$. By Corollary \ref{cor:killer} and Lemma \ref{glen3} such boxes do not intersect $\mathfrak D_{3.65}$. The last category verifies that $\sB_\fb \subset \cN_{r_1} \cap \cN_{r_2}$ for words $r_1$ and $r_2$, which by Lemma \ref{lem:var_nbd} proves that they are relators at discrete points.
\end{proof}

Our next goal is to analyze each pair of relators. The idea will be to look at the presentations $\Gamma_{r_1,r_2} = \la m,n, g \mid mnMN, r_1, r_2\ra$. We will either prove that the presentation (or any quotient of it) cannot correspond to a hyperbolic 3-manifold group or give an isomorphism between $\Gamma_{r_1,r_2}$ or  $\pi_1(m003)$ or $\pi_1(m004)$. For $\mathfrak{p} \in \mathfrak D_{3.65}$, the bicuspid representation with image $\bcusp_\mathfrak{p}$ will have to factor through $\Gamma_{r_1,r_2}$ for the associated $r_1, r_2$. Thus, the isomorphisms will give epimorphisms from  $\pi_1(m003)$ or $\pi_1(m004)$ onto $\bcusp_\mathfrak{p}$. The last step will be to promote these epimorphisms to isomorphisms using the following Lemma.

\begin{lemma}\label{epi} Let $\Gamma \leq \PSL(2,\bC)$ be such that $\bH^3/\Gamma$ is a cusped finite-volume hyperbolic 3-manifold. Let $\pi_1$ be the fundamental group of m003 or m004. If $\phi: \pi_1 \to \Gamma$ is an epimorphism, then it is an isomorphism.
\end{lemma}

\begin{proof} We adapt and closely follow an argument of \cite{Soma:2002}. Since $\bH^3/\Gamma$ is a cusped finite-volume hyperbolic 3-manifold, we can lift $\Gamma$ to $\SL(2,\bC)$ as $\tilde\Gamma$.  Let $\mathfrak U$ be the $\SL(2,\bC)$-character variety of $\tilde\Gamma$. Note that $\phi$ induces an algebraic map $\phi^* : \mathfrak U \to  \mathfrak X$, where $\mathfrak X = {\mathfrak X}_{m003}$ or ${\mathfrak X}_{m004}$ is the corresponding $\SL(2,\bC)$-character variety for $\pi_1$. Since $\phi$ is an epimorphism, this map is an injective regular map. It is known that both ${\mathfrak X}_{m003}$ and ${\mathfrak X}_{m004}$ only have one component that contains irreducible representations, which must therefore be the canonical (or Dehn surgery) component, see for example \cite{Tillmann:2000}. Thus, $\phi^*$ must map the canonical component ${\mathfrak U}_0$ of ${\mathfrak U}$ to the canonical component ${\mathfrak X}_0$ of $ {\mathfrak X}$. Since $\Gamma$ is cusped,  ${\mathfrak U}_0$ has dimension at least $1$. Further, since $\phi^*$ is injective and regular, it follows that the $\phi^*$ is onto ${\mathfrak X}_0$. Let $\chi_\eta \in {\mathfrak X}_0$ be a character corresponding a lift $\eta$ of a faithful discrete representation of $\pi_1$ and  let $\chi_\rho \in {\mathfrak U}_0$ be such that $\phi^*(\chi_\rho) = \chi_\eta$. Let $\rho$ be a representation realizing $\chi_\rho$. Then, we have $\eta = \rho \circ \phi$ is a faithful representation of $\pi_1$, which implies that $\phi$ is monic and therefore an isomorphism. 
\end{proof}

{\bf Proof of Theorem \ref{thm:fig8}.} For each pair of relators in Table \ref{table:rel3}, it is clear that they (or their quotients) cannot correspond to a hyperbolic 3-manifolds group because at least one of the relators is a power of a word of $g$-length at most 3. Since a word of $g$-length $\leq 3$ cannot be the identity by Lemma \ref{glen3}, it follows that these relators must give rise to elliptics, which is impossible.

Notice that any discrete geometric bicuspid group $\bcusp_\pp$ in our parameter space automatically satisfies the conditions of Lemma \ref{epi} by either Theorem \ref{thm:bicusp} or \ref{agol}. Thus, it remains to list isomorphisms from $\pi_1(m003)$ and $\pi_1(m004)$ onto the corresponding presentations. These are found in Lemma \ref{isos}.
\qed

\newcommand{\orthogonal}{\bot}
\section{Necklaces}\label{sec:necklaces}

Motivated by Theorem \ref{thm:param}, we now demonstrate that  the existence of a necklace with at most $7$ beads in the universal cover of a complete hyperbolic 3-manifold $Y$ implies that $Y$ admits a non-elementary embedding of a full $k$-necklace-manifold with $k \leq 7$. In particular, we prove

\existneckl*

This section is quite long and deals with a careful analysis of possible topological and geometric aspects on horoball necklaces. In the next subsection, we try to motivate the relevant terminology and difficulties.

\subsection{Introduction to necklaces and horoball systems}

Theorem \ref{thm:param} gives a nontrivial word that is the identity in the fundamental group of a complete hyperbolic 3-manifold with a cusp of low volume.
The next lemma shows that this algebraic condition corresponds geometrically to the existence of a \textit{necklace}.

\begin{definition}  
  A \emph{(k-)necklace} $\eta = (N_1,\ldots, N_k)$ is a cyclically ordered set of horoballs such that $N_i$ is has disjoint interior from and is tangent to $N_{i+1}$ for $1 \leq i \leq k$.
  Further, any two horoballs have disjoint interiors or coincide.
  We also require that there are at least 3 distinct horoballs in $\eta$.
  In what follows, the indices for a $k$-necklace are always taken cyclically with $N_{k+1} = N_1.$
  The $N_i$'s are called the \emph{beads}, and $k$ is called the \emph{necklace} or \emph{bead number} of $\eta.$
  We often abuse notation by denoting $\bigcup N_i$ by $\eta.$
\end{definition}

\begin{lemma}\label{lem:words_to_necklaces}
  Suppose $Y$ is a complete hyperbolic 3-manifold $\bH^3/\Gamma$ with a maximal cusp $\kappa.$
  Fix a bicuspid subgroup $\la m,n,g \ra \leq \Gamma$ corresponding to $\kappa$.
  Suppose $w$ is a word in $m,$ $n,$ and $g$ that is the identity in $\Gamma$ and has $g$-length $k.$
  Then there is a $k$-necklace in $\bH^3$ consisting only of lifts of $\kappa.$
\end{lemma}

\begin{proof}
  See the proof of \cite[Lemma 5.3]{Agol:2010}. 
  Writing $w = \prod_{i =1}^{\deg_g(w)} \lambda_i g^{\pm}$ where $\lambda_i \in \langle m,n \rangle$ and letting $w_k =  \prod_{i =1}^k \lambda_i g^{\pm},$ the desired necklace is $\{w_i(\mH) : 1 \leq i \leq k\}$. Note that in such a necklace a horoball could be visited multiple times, i.e. we could have $N_i = N_{i+j}$ for some $j$ with $|j| > 1$.
\end{proof}

\begin{definition}  
  A \emph{horoball system} $(\mH, \mT)$ is a collection of horoballs $\{C_i\} \subset \bH^3$  with pairwise disjoint interiors and $\mT$ is a subset of the tangency points between these horoballs.
  We require that no triple $(C_i, C_j, C_k)$ of horoballs is pairwise tangent \emph{via} tangencies in $\mT.$ 
  We call $\mT$ the (true) \emph{tangency set}.
  We often abuse notation by denoting $\bigcup C_i$ by $\mH$ and refer to a horoball system as $\mH$ when $\mT$ is clear from context.  
  A tangency $s$ between horoballs $C_i, C_j$ in $\mH$ is called \emph{false} if $s\notin \mT.$  
  Let $\mF$ denote the set of \emph{false tangencies}.  

  We say that $\eta$ is a \emph{necklace in} $(\mH,\mT)$ if each $N_i \in \mH$ and $N_i \cap N_{i+1}$ lies in $\mT.$
  A necklace $\eta$ is \emph{minimal} if $N_i\cap N_j\in \mT$ implies $|i-j|=1.$  
  It is \emph{globally minimal} if it minimizes bead number among all necklaces in $(\mH, \mT).$  

 Since $\mH$ is locally finite, removing balls of radius $\epsilon > 0$ from $\mH$ around all \emph{false} tangencies produces isotopic sets for all $\eps$ small enough.
  We call any set so obtained $shaved(\mH)$, since the removal of the balls ``shaves'' the false tangencies away from $\mH.$
  
For necklaces in $(\mH,\mT)$, $\shaved$ is inherited from the system. For a necklace $\eta = (N_1, \ldots, N_k)$ \emph{not in} an ambient horoball system, $\shaved(\eta)$ or $\shaved(N_i)$ denotes a shaving-back \emph{non-sequential} tangencies. 
\end{definition}

\begin{lemma}\label{minimal} 
  If $(\mH, \mT)$ is a horoball system and $\eta=(N_1,\ldots, N_k)$ is a non-minimal necklace in $(\mH,\mT),$
  then a globally minimal necklace has length at most $(k+2)/2.$
\end{lemma}

\begin{proof}
  If $\eta$ is non-minimal, then $N_i \cap N_j \in \mT$ for some $i < j$ with $|i - j| > 1.$
  Then $(N_i, N_{i+1}, \ldots, N_j)$ and $(N_j, \ldots, N_k, N_1, \ldots, N_i)$ are necklaces, and one of them has bead number at most $(k+2)/2.$
\end{proof}

\begin{figure}[h]
  \begin{center}\begin{overpic}[scale=.5]{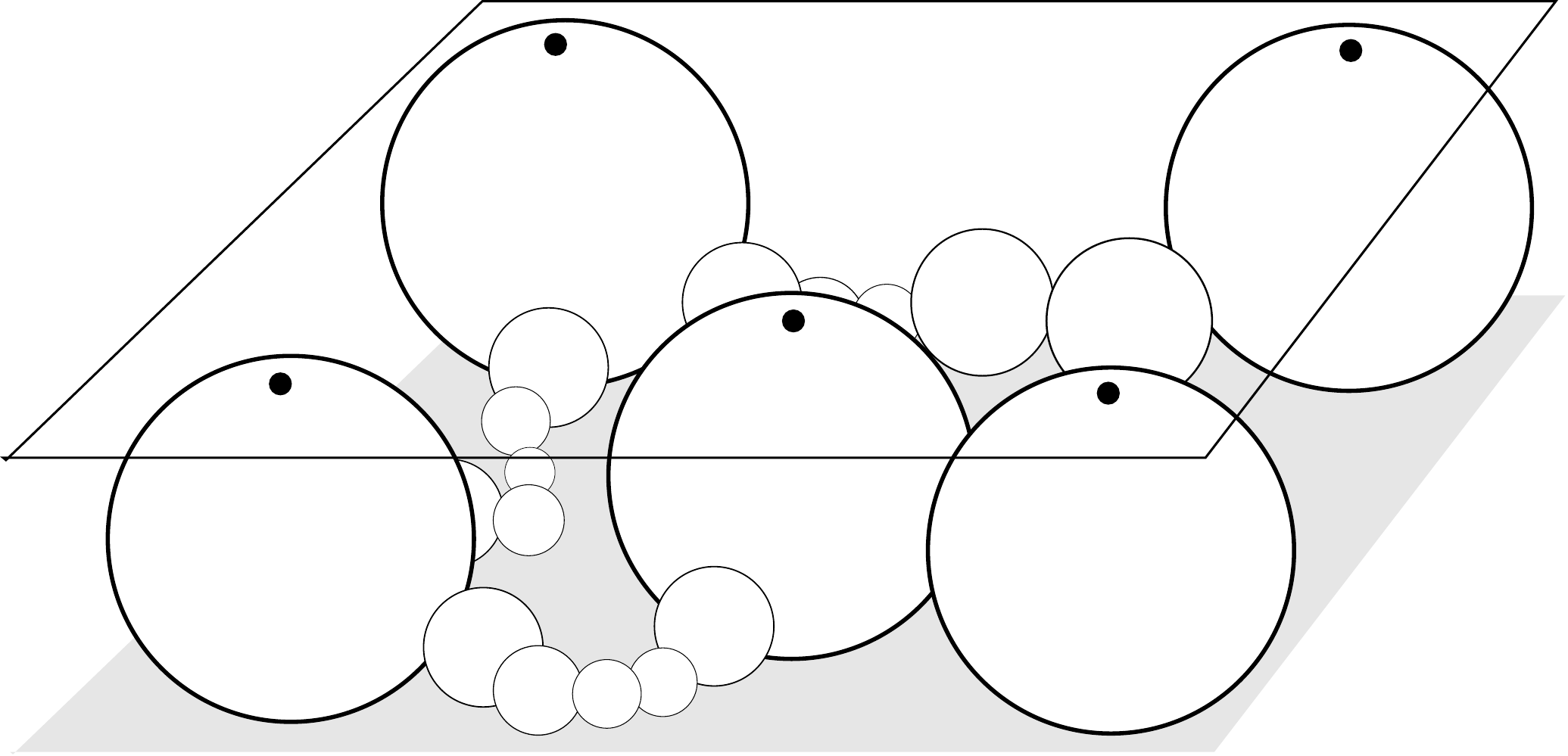}
    \end{overpic}
  \end{center}
  \caption{A local picture of a horoball system}
  \label{fig:horo_sys}
\end{figure}

\begin{remark}\label{topology} 
  A horoball system can arise naturally as the preimage in $\bH^3$ of a maximal horocusp $\kappa$ in a cusped hyperbolic 3-manifold $Y.$
  We will refer to geometric objects, constructions, and operations in $\bH^3$ as being {\it upstairs}, thinking of $\bH^3$ as the universal cover sitting ``above'' some hyperbolic 3-manifold.
  Here $\mT$ is naturally in bijection with elements of $\mO(1),$ the first orthoclass of $\kappa,$ which is the $\Gamma$-orbit of a fixed pair of tangent horoballs.
  Elements of $\mO(1)$  are in 1-1 correspondence with elements of $\mT$ by passing to the point of tangency.
  See \cite{GMM:2009} for more details. 
  False tangencies arise when multiple orthoclasses have 0 orthodistance. 
  As shown in \cite{CaoMeyerhoff:2001} or \cite{GMM:2009}, there are never three distinct pairwise tangent horoballs with all pairs in $\mO(1).$ 
  That is, there are no $(1,1,1)$ triples in the language of \cite{GMM:2009}. 
  For this reason we disallow triples of pairwise tangent horoballs with all tangencies in $\mT$ in general horoball systems.

  As in \cite{GMM:2009}, we will be constructing handle structures on submanifolds of a cusped hyperbolic 3-manifold $Y.$ 
  Here, our handle structure $W_1$ will have base $T^2\times I$ and have a single 1-handle $\sigma$ attached.
  The $T^2\times I$ will be a shaved maximal horocusp (to avoid creating extra topology) and the 1-handle will be a neighborhood of the preferred tangency. 
  We will then prove that we can attach a single 2-handle embedded in $Y \rsetminus \inte(W_1)$ to obtain an embedded manifold $W$.
  Upstairs, the shaving away of horoballs at false tangencies is done because the 2-handle may need to pass through them. 
  Going downstairs, $W_1$ is understood to be appropriately shaved back.
  We further prove that we can choose this 2-handle to be ``good,'' meaning $W$ is a full necklace manifold.

  \begin{figure}[h]
    \begin{center}\begin{overpic}[scale=.7]{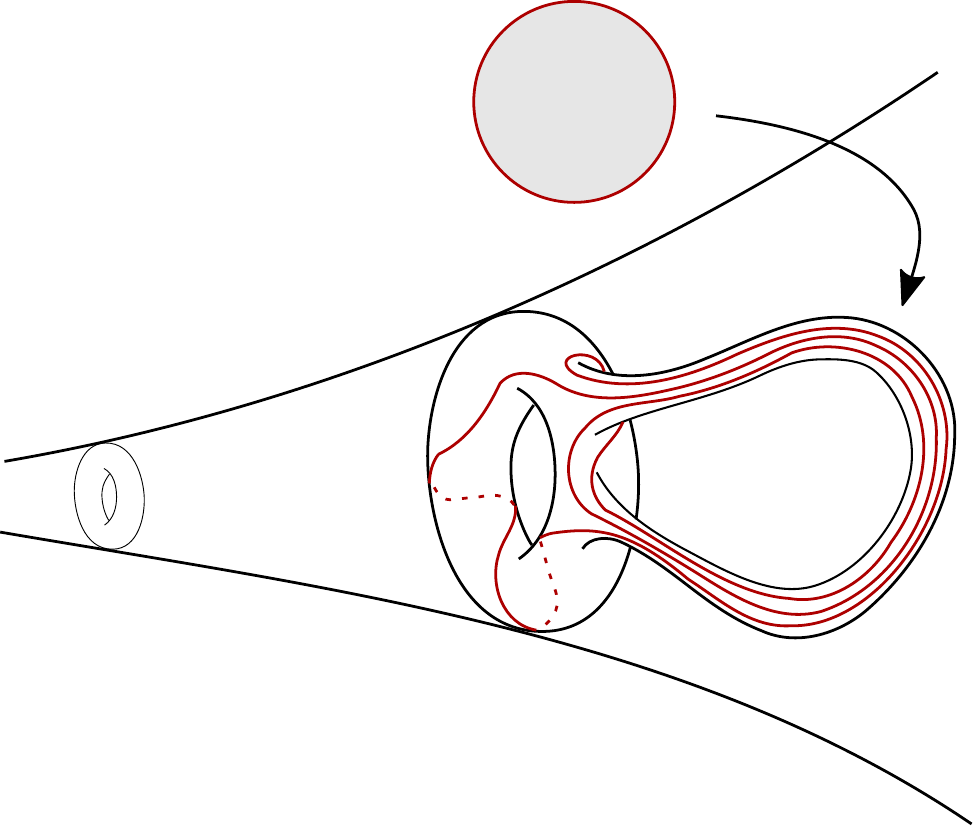}
        \put(28,32){$W_1$}
      \end{overpic}
    \end{center}
    \caption{A picture showing $W_1$ and a curve on $\partial W_1$ along which we hope to attach an embedded 2-handle in $Y \smallsetminus \inte(W_1)$ to obtain the necklace manifold $W$.}
    \label{fig:handles}
  \end{figure}

  Notice that the subgroup of $\Gamma$ generated by $\pi_1(X)$ is our bicuspid group  $\bcusp$. 
  Replacing the tangency set $\mT$ with a $\Gamma$ orbit of a false tangency gives rise to a different bicuspid group. 
  A $k$-necklace corresponds to an element in the bicuspid group that is trivial in $\Gamma$ and represented by a word with $g$-length equal to $k.$ 
\end{remark}

\begin{definition}
  If the horoball system $(\mH, \mT)$ arises from a bicuspid group of a cusped hyperbolic 3-manifold, then we say that $\mH$ arises \emph{geometrically}.
\end{definition}

\begin{notation}
  If $S \subset \bH^3$, then $\partial_\infty S$ denotes its limit points in $\Sinfty$. 
  We let $\hat{S} = S \cup \partialinfty S.$ 
  Conversely, given a set $\hat S\subset\bH^3\cup\Sinfty,$ $S$ will denote $\hat S\cap\bH^3.$

  \end{notation}
  
  \begin{definition} Let $\eta=(N_1,\ldots, N_k)$ be a necklace is a horoball system $(\mH,\mT)$.
  For horoballs $C_i$ and $C_j$, let $\gamma(C_i,C_j)$ denote the geodesic from $\partialinfty C_i$ to $\partialinfty C_j$. 
  If $t\in \mT,$ then $\gamma_t = \gamma(C_i, C_j)$ where $C_i\cap C_j=t$.
  For necklaces, we use the convention that  $\gamma_i$  denotes the geodesic $\gamma(N_i, N_{i+1})$.
  Each $\hat\gamma_t$ is called a \emph{tie}.
  Notice that distinct ties can only intersect at endpoints.
  We call $\bigcup_{i=1}^k\hat\gamma_i$ the \emph{frame} of $\eta$ and denoted by $\fram(\eta).$ 
  Define the \emph{open frame} of $\eta$ to be $\fram(\eta)\cap\bH^3.$
  When $\eta$ is minimal, we build a simple closed curve $C(\eta) \subset \bH^3,$ unique up to isotopy, called the \emph{core} of $\eta$ obtained by taking a  union of arcs $\delta_i,$ where $\delta_i$ is an unknotted arc in $N_i$ connecting $N_{i-1}\cap N_i $ to $N_i\cap N_{i+1}.$
  \end{definition} 

\begin{figure}[h]
  \begin{center}\begin{overpic}[scale=.7]{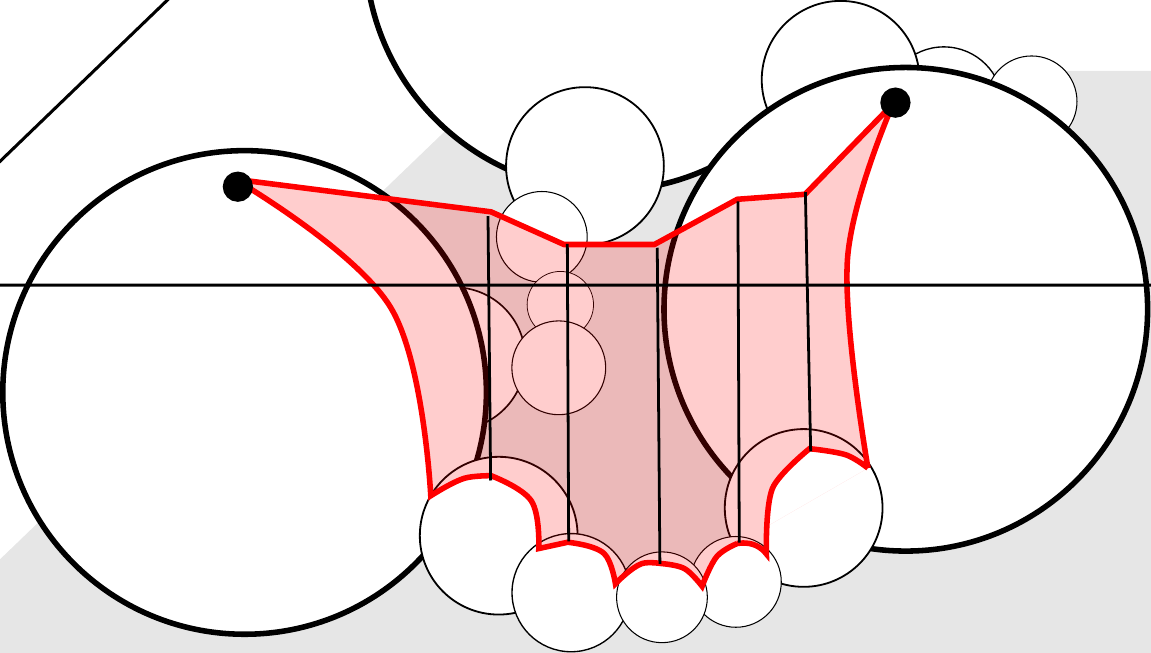}
    \end{overpic}
  \end{center}
  \caption{A picture showing a compressing disc for an unknotted necklace.}
  \label{fig:handles}
\end{figure}

\begin{definition}\label{def:nekl} Let $K_\eta = \bH^3\rsetminus \inte(\shaved(\eta)).$
  The necklace $\eta$ is \emph{unknotted} if there exists a properly embedded 2-disc $(D, \partial D) \to (K_\eta , \partial K_\eta )$ whose boundary meets each $N_i\cap N_{i+1}$ exactly once. 
    $D$ is called a \emph{spanning, unknotting} or \emph{compressing disc}. 
  We say that an unknotted $\eta$ is \emph{blocked} in a horoball system $\mH$ if there exists $t\in \mT$ such that $\gamma_t$ transversely intersects $\inte(D)$ for some spanning disc $D$ for $\eta$ exactly once up to proper isotopy in $(K_\eta, \partial K_\eta).$ 
  Otherwise, we say $\eta$ is \emph{unblocked} in $\mH.$
  We say that $\eta$ is \emph{unlinked} in $\mH$ if there exists a spanning disc $D$ for $\eta$ with $\inte(D)$ inside $\bH^3\rsetminus \shaved(\mH).$ 
   If $(\mH,\mT)$ arises from a maximal cusp of the hyperbolic 3-manifold $Y,$ then an unlinked necklace $\eta$ is \emph{simple} if it has a compressing disc $D$ such that $\phi(D)\cap D=\emptyset$ for all nontrivial $\phi\in \Gamma.$ 
\end{definition}

\begin{remark} \label{looptheorem}
  For a geometric horoball system $(\mH,\mT)$ as in Remark \ref{topology}, there exists a simple unlinked necklace in $(\mH,\mT)$ if and only if $\partial W$ is compressible in $Y\rsetminus \inte(W).$ 
  This follows by applying the loop theorem to any compressing disk.
\end{remark}

\begin{lemma}\label{lem:looptheorem}
  Let $(\mH,\mT)$ be a geometric horoball system and let $\eta$ be an unlinked $k$-necklace in $(\mH,\mT).$ 
  Then either $\eta$ is simple or there exists a simple $k'$-necklace $\eta'$ with $k'\le k.$ 
  Further, if $k$ is odd, one can take $k' < k.$
\end{lemma}

\begin{proof}
  Let $D$ be a spanning disk for $\eta$ and assume $\eta$ is not simple. 
  Since the action of $\Gamma$ is properly discontinuous and $D$ is compact, $D$ meets only finitely many $\Gamma$-translates of itself.
  By a standard cut-and-paste argument, we can assume no intersections are entirely in $\inte(D)$, so every intersection with a translate arises from shared horoballs.
  Cutting and pasting this finite collection of discs gives spanning discs of simple necklaces. 
  Since all of these necklaces are made from pieces of $\Gamma$-translates of $\eta,$ one of these is a $k'$-necklace with $k' \leq k$. Further, if $k$ is odd we can choose $k' < k$.
\end{proof}

\begin{remarks} \label{facts}
  We now record some basic facts and examples.   
  \begin{enumerate}
  \item The core of a minimal necklace is unique up to isotopy in $\bH^3.$  
  \item If $\eta$ is minimal and unknotted, then $\core(\eta)$ is the unknot in $\bR^3$.
  \item Every link in $\bR^3$ is realized as the core of a disjoint union of minimal necklaces.
    Figure \ref{fig:tref} shows an 18-bead trefoil.  
  \item There exist knotted necklaces with unknotted cores.
    Figure \ref{fig:core} gives a 14-bead example and Figure \ref{fig:eightneckl} shows a 2-component link with 3 and 8 beads with unlinked cores.
  \item There exist unknotted and unblocked necklaces in horoball systems that are linked.
    Figure \ref{fig:borrom} shows a Borromean necklace.
  \item There exist horoball systems $(\mH, \mT_1)$ and $(\mH, \mT_2)$ with $\mT_2 \subset \mT_1$ such that $\eta$ is a necklace in both system but is unlinked in $\mT_2$ and linked in $\mT_1.$
  \end{enumerate}
\end{remarks}

\begin{figure}[h]
  \begin{center}\begin{overpic}[scale=1.2]{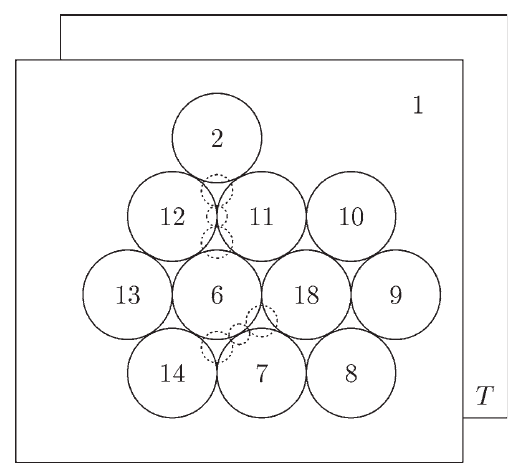}
    \end{overpic} 
  \end{center}
  \caption{A necklace whose core is a trefoil. Picture from \cite{Maehara:2007}.}
  \label{fig:tref}
\end{figure}
  
  \begin{figure}[h]
  \centering
  \begin{minipage}{0.45\textwidth}
\begin{overpic}[scale=.67]{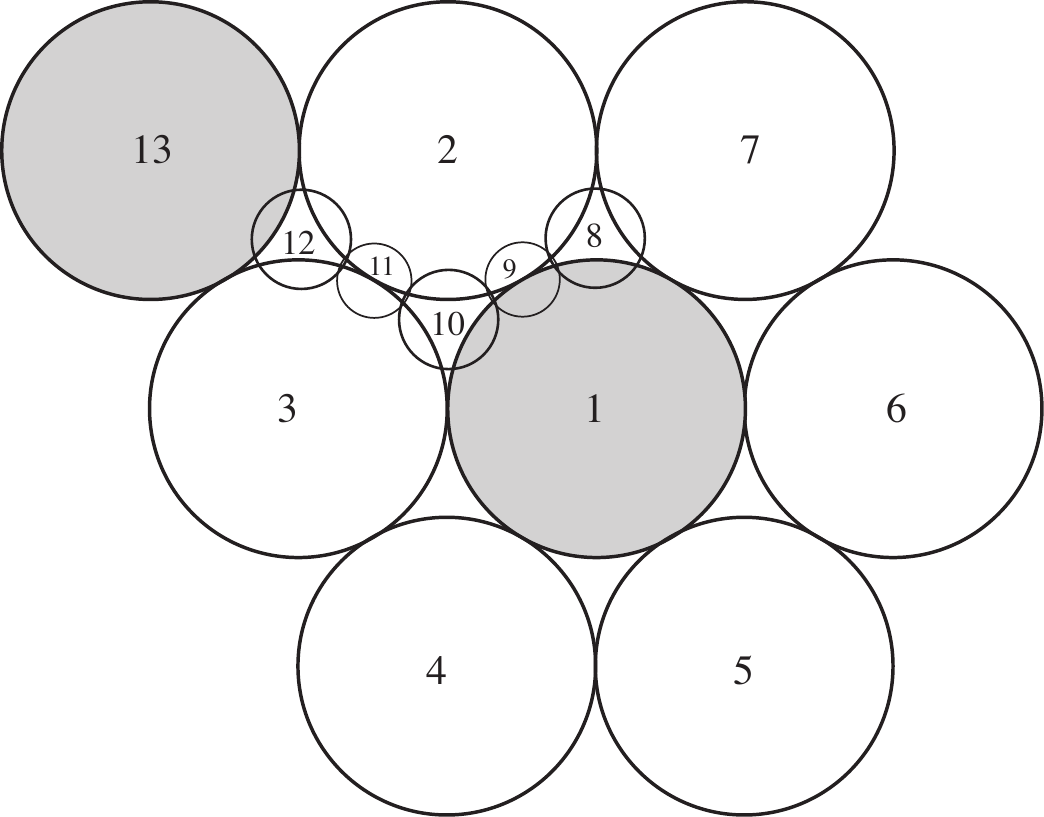}
    \end{overpic}
  \caption{A knotted necklace with unknotted core. The $14^\text{th}$ horoball is at infinity.}
    \label{fig:core}
  \end{minipage}
  \hspace*{4ex}
  \begin{minipage}{0.45\textwidth}
\begin{overpic}[scale=.09]{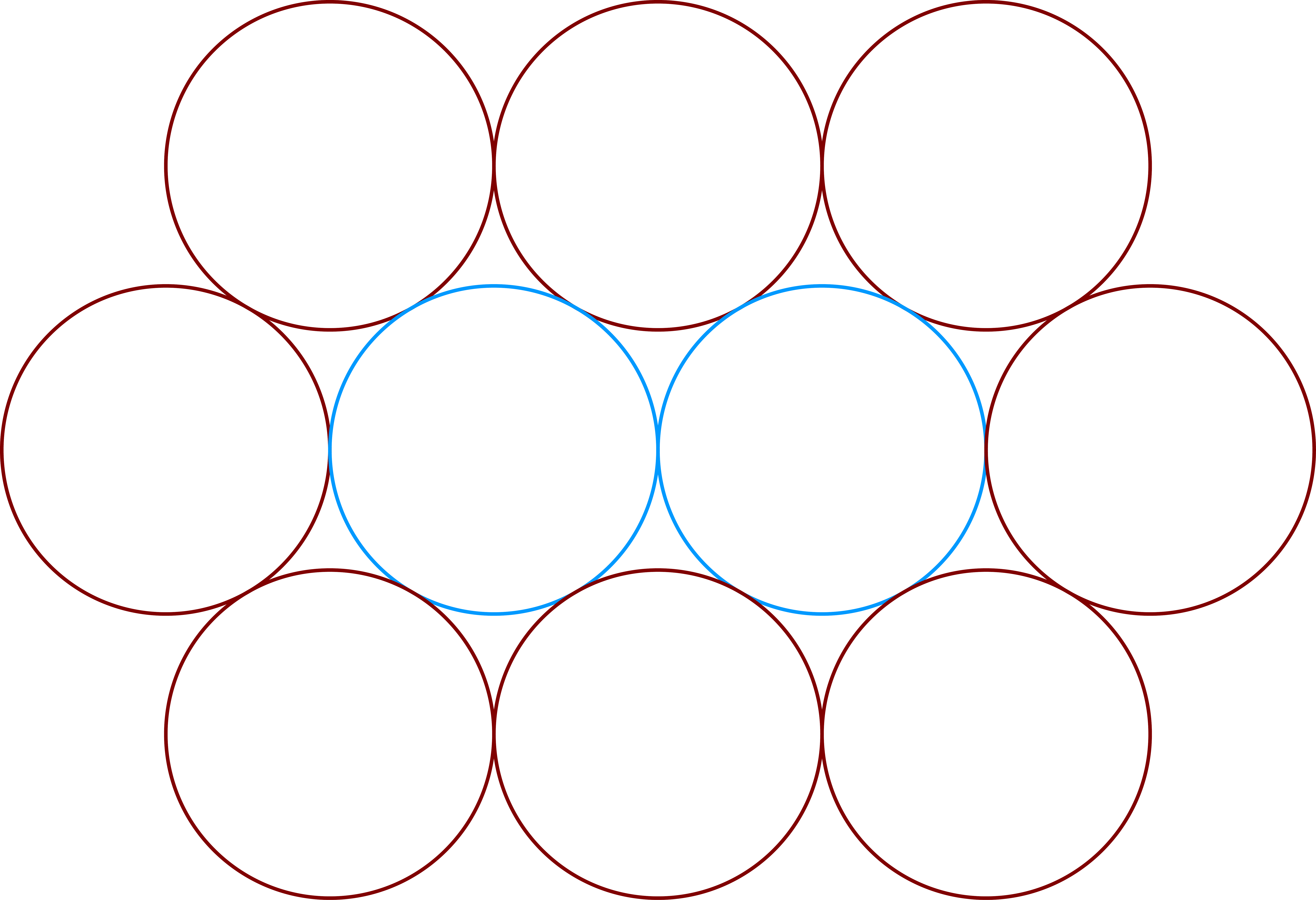}
    \end{overpic}    \caption{A knotted link with unknotted core. The third blue horoball is at infinity.}
    \label{fig:eightneckl}
  \end{minipage}
\end{figure}

\begin{figure}[h]
  \begin{center}\begin{overpic}[scale=.7]{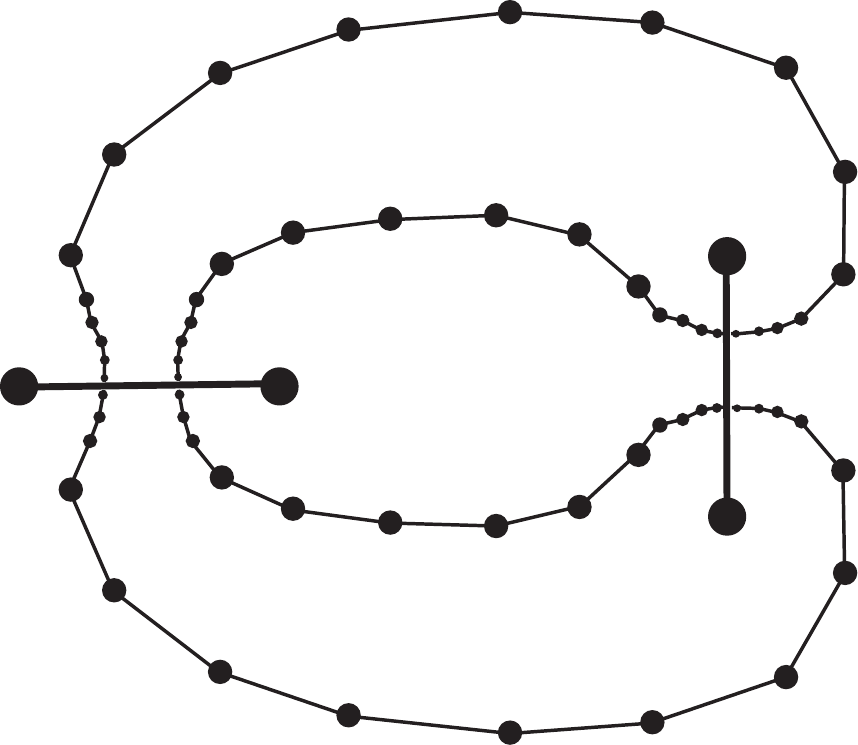}
    \end{overpic}
  \end{center}
  \caption{A Borromean necklace. It is unknotted and unblocked, but is linked. The arcs are ties in a horoball system.}
  \label{fig:borrom}
\end{figure}

{\bf Outline of the proof of Theorem \ref{thm:7neckl}.} The proof breaks down into two main steps. First, using the terminology of this section, we prove that every geometric horoball system that contains some $\leq 7$-necklace must contain a (possibly different) globally minimal simple $\leq 7$-necklace $\eta$. Then, we show that the handle structure $W$ obtained by attaching the $2$-handle corresponding to a compressing disc for $\eta$ downstairs can be replaced by some (possibly different) {\it full} necklace-$\leq 7$ structure.

 To attack step one, it is enough to find an unlinked $\leq 7$-necklace by Lemma \ref{lem:looptheorem}. In the following subsections, we first show that minimal $\leq 8$-necklaces are unknotted,  then demonstrate that minimal unblocked $\leq 7$-necklaces are unlinked, and finally prove that globally minimal $\leq 7$-necklaces in geometric systems must be unblocked. We remark that this is a rather technical task as unblocked-ness is false when working with non-orientable manifolds. There are non-orientable examples that contain blocked $6$-necklaces as observed in \cite{AdamsKnudson:2013}. Lastly, to get a full structure, we use tools similar to those in \cite{GMM:2011} which involve some controlled topological surgery arguments.

\subsection{Unknotting criteria}

\begin{definition}
  A \emph{spanning disc} for a geodesic $\gamma\in \bH^3$ is an embedded smooth closed disc $\hat D\subset \bH^3\cup \Sinfty$  with $\partial D=\gamma.$ 
  A \emph{spanning} or \emph{unknotting disc} for $\fram(\eta)$ is an embedded smooth closed disc $\hat D$ in $\bH^3\cup \Sinfty$ with $\partial \hat D=\fram(\eta).$ 
  If such a $\hat D$ exists then we say that $\fram(\eta)$ is unknotted. 
\end{definition}

\begin{lemma}  \label{deformation}
  If $\eta$ is a necklace, then $\shaved(\hat \eta)$  deformation retracts to $\fram(\eta).$
\end{lemma}

\begin{proof}
  In a shaved necklace, each shaved horoball is only tangent to two others. 
  It therefore defamation retracts to the piece of the frame it contains.
\end{proof}

\begin{theorem}[Unknotting Criterion] \label{criterion}
  A minimal necklace $\eta= (N_1, \cdots, N_k)$ is unknotted if and only if $\fram(\eta)$ is unknotted if and only if each tie $\hat \gamma_i$ has a spanning disc $\hat D_i$ with $\hat D_i\cap\fram(\eta)=\hat\gamma_i.$
\end{theorem}

\begin{proof}
  The first ``if and only if''  follows from Lemma \ref{deformation}. 
  The backward direction of the second implication follows from the fact that if each $\hat \gamma$ has a spanning disc with interior disjoint from $\fram(\eta),$ then, since ties can only meet at endpoints, the usual cut and paste argument produces a collection of such discs $\hat D_i$ that intersect only along $\partialinfty\fram(\eta).$ 
  Since $\cup_{i=1}^k \partialinfty D_i$ is a simple closed curve in $\Sinfty$ it follows that $\fram(\eta)$ is unknotted. 
  The forward direction follows from the next topological lemma.
\end{proof}     

\begin{lemma}
  Let $U_1, \ldots, U_k$ denote smooth properly embedded arcs in the closed upper half-space of $\bR^3$  each with a single local maximum and otherwise transverse to the planes $z$ = constant. 
  Let $u'_i$ and $u''_i$ denote the endpoints of $U_i$ and assume that $u''_i=u'_{i+1}$ and that there are no other points of intersection among the $U_i$'s. 
  Let $F=\bigcup U_i.$ 
  Then $F$ bounds a  disc $D$ if and only if each $U_i$ bounds a spanning disc $D_i$ with $D_i\cap F=U_i.$
\end{lemma}

\begin{proof}
  By a standard cut and paste argument we can assume that the $D_i$'s are pairwise disjoint, so $\cup_i D_i \cap \{z = 0\}$ is a simple curve and the backward direction is immediate.

  Now let $D$ be a spanning disc for $F.$ 
  We can assume that D is smooth and standard near $\{z = 0\}.$ 
  The forward direction entails properly isotoping both $D$ and $F,$ until $F$ becomes a union of closed geodesics in a closed hyperbolic plane and $D$ is the region bounded by them. 
  This isotopy follows standard geometric arguments as in \cite{Rolfsen:1990}.
  Here is a brief outline. 
  First isotope $D$ to eliminate all local minima. 
  Then by isotoping both $F$ and $D$  eliminates all local maxima disjoint from $F.$ 
  At this point $D$ has only saddle tangencies so  $D$ and $F$ can be properly isotoped to the standard pair. 
\end{proof}

\begin{theorem} \label{thm:8unknot}
  Minimal $\le 8$-necklaces are unknotted. 
\end{theorem}

\begin{proof}
  We will consider the case of length 8, as the other cases are simpler. 
  We show that $\fram(\eta)$ satisfies the unknotting criterion, where $\eta=(N_1, \cdots, N_8).$ 
  It suffices to show that the tie $\hat \gamma_{8}$ has a spanning disc. 
  Conjugate $\bH^3$ so that $N_8 =\{(x,y,z)|z\ge 1\}$ and $N_1$ is full-sized and centered at $(0,0).$ 
  The total visual angle of $N_2, \cdots, N_7$ from $\gamma_{8}$ is $\leq 2\pi$ with equality if and only if each is full-sized and tangent to $N_1.$ 
  In that case, the $N_7, N_2$ tangency $t$ is false and can let $\ell$ denote the ray from $(0,0)$ in $\{z = 0\}$ through $\pi(t),$ where $\pi$ is the projection map to $\{z = 0\}.$ 
  In the other case, let $\ell$ be the ray from $(0,0)$ disjoint from $\pi(N_2)\cup\cdots\cup \pi(N_7).$ 
  Then, $\ell \times [0,\infty)\cup\infty$ is a spanning disc for $\gamma_8.$ 
\end{proof}


We expect that $8$ is far from optimal. However, Remark \ref{facts} (iv) shows that 13 is an upper bound.

\begin{conjecture} Minimal necklaces of length $\le$ 13 are unknotted. \end{conjecture}

A related result of Maehara states that necklaces of length $\leq 11$ have unknotted cores \cite[Theorem 10]{Maehara:2007}. Given Maehara's example of an 18-necklace trefoil in Figure \ref{fig:tref}, it is natural to ask if this is optimal. 

\begin{conjecture} Minimal necklaces of length $<$ 18 have unknotted cores. \end{conjecture}

\subsection{Minimal unblocked $\leq 7$-necklaces are unlinked}

Following Definition \ref{def:nekl}, we want to find short simple unlinked necklaces in geometric horoball system.
By Lemma \ref{looptheorem}, we know that is it enough to find an unlinked necklace.
This section is devoted to proving the following proposition.

\begin{proposition}\label{prop:unblocked_to_unlinked}
  If $\eta$ is minimal, unblocked, and $k\le 7,$ then $\eta$ is unlinked.
\end{proposition}

To motivate what follows, we preview the structure of the proof of Proposition \ref{prop:unblocked_to_unlinked}.
To show $\eta$ is unlinked, we must find an unknotting disc inside $\bH^3\rsetminus \inte(\shaved(\mH)).$
We already know by Theorem \ref{thm:8unknot} that $\eta$ admits an unknotting disc $D$ in $\bH^3$.
Our goal will be to surger this disc so that its interior if disjoint from $\shaved(\mH)$.
This ultimately comes down to accounting for tangencies involving non-necklace horoballs.

\begin{definition}\label{non_necklace_tangencies}
  Let $\eta=(N_1,\cdots, N_k)$ be a necklace in a horoball system $(\mH, \mT).$ 
  Let $\mT'\subset \mT$ denote the tangencies of the form $C_i\cap C_j,$ where $C_i\in \eta$ and $C_j\notin\eta$ and let $\mT''\subset\mT$ denote tangencies with both $ C_i ,$ $C_j$  $\notin\eta.$ 
  If $t\in \mT',$ then let $N_t $ denote the horoball of $\eta$ that contains $t$ and $B_t$ the non-$\eta$ horoball.
\end{definition}

We need ways of separating non-necklace tangencies from our disc $D.$
We do this through the use of a \emph{transverse hull} and \emph{escape planes}. 
The former controls tangencies between two non-necklace horoballs.
The latter control tangencies between a necklace horoball and a non-necklace horoball.

\subsubsection{Transverse hull}
Our first goal is to control the tangencies in $\mT''$. For this subsection $\mH$ is a horoball system containing a minimal, unknotted, unblocked, $k \leq 7$-necklace $\eta$. For $\delta > 0,$ let
\[ \sK_\delta = \{p \in \bH^3\ |\ d(p, hull(\fram(\eta)) \leq \delta\}. \]

\begin{lemma}[Universal Hull Radius]\label{lem:universal_hull_radius}
 There is a universal $\eps > 0$ such that for all $0 < \delta < \eps$ one has $\sK_\delta \cap \gamma_t = \emptyset$ for all $t \in \cT''$.
\end{lemma}

\begin{proof}
  Suppose $t \in \cT''$ and $t = B_1 \cap B_2.$
  Conjugate $\bH^3$ so that $B_2$ is $H_\infty$ and $B_1$ is $H_0$.
  Let $\lambda=\pi(\fram(\eta)),$ viewed as an oriented piecewise linear periodic loop $\lambda: \bR \to \bR^2\rsetminus (0,0).$ 
  Here, the vertices of $\lambda$ are the centers $\partialinfty N_i,$ and the edges are the projections of the ties. 
  Each edge has visual angle at most $\pi/3$ at $(0,0).$
  Since $\gamma(B_1, B_2)$ does not block $\eta,$ the winding count of $\eta$ around $\gamma(B_1, B_2)$ is zero. 
  Define $\theta(t)$ to be the cumulative signed angle between $\lambda(0)$ and $\lambda(t)$ as measured from $(0,0).$
  Then, since the winding count of $\eta$ is zero, $\theta$ attains a maximum $\theta_{max}$ and minimum $\theta_{min}.$
  Now, the visual angle of each segment is at most $\pi/3,$ so immediately $\theta_{max} - \theta_{min} \leq 7\pi/3.$
  But in fact, since the winding count of $\eta$ is zero, every deviation from $\lambda(0)$ must be compensated eventually by an equal and opposite deviation.
  Hence $\theta_{max} - \theta_{min} \leq 7\pi/6.$
  So the visual angle $\phi = \theta_{max} - \theta_{min}$ of $\lambda$ at $(0,0)$ satisfies $\phi \leq 7\pi/6 < 2\pi.$
  Consequently, we may define $W_\phi$ to be the minimal closed wedge bounded by two rays based at $(0,0)$ such that $\lambda \subset W_\phi,$ so that $W_\phi$ has angle $\phi.$

  We claim that in fact $\phi < \pi$.
  Let $R_1$ and $R_2$ be the bounding rays of $W_\phi.$ 
  Each contains vertices of $\lambda.$ 
  Permuting the indices, we can assume $\partialinfty N_1\subset R_1$ and $\partialinfty N_j\subset R_2.$ 
  These vertices divide $\lambda$ into two arcs $\lambda_1$ and $\lambda_2,$ with $\lambda_1$ determined by $N_1, N_2, \ldots, N_j$ and the other by $N_j, N_{j+1}, \ldots, N_1.$
  The visual angles of $\lambda_1$ and $\lambda_2$ are both just $\phi.$
  We can assume that $\lambda_1$ has no more segments than $\lambda_2,$ and hence that $\lambda_1$ has at most 3 segments. 
  If $\lambda_1$ has fewer than three segments, then its visual angle $\phi \leq 2\pi/3<\pi$. 
  If $\lambda_1$ has three segments, then $k\ge 6$ and $\phi \leq \pi.$
  Assume for a contradiction that $\phi = \pi.$
  The four vertices of $\lambda_1$ are the centers of the beads $N_1, \ldots, N_4$ of $\eta.$
  Since $\phi = \pi,$ these beads are full-sized and tangent to $B_1.$
  Let $P$ be the hyperbolic plane tangent to $N_4$ and $N_3$ at $N_4\cap N_3.$ 
  Then $L = \partial_\infty P$ is a Euclidean line through the points $(0,0)$ and $\pi(N_4\cap N_3).$
  $L$ makes an angle of $\pi/6$ with $R_2.$ 
  Since $N_5$ is tangent to $N_4$ and has interior disjoint from $N_3,$ we see, by conjugating $N_3$ to $H_\infty,$ that $\partial_\infty N_5$ must lie in the closed half space of $\bR^2$ with boundary $L$ that contains $\partial_\infty N_4.$ 
  It follows that, going back to the perspective with $\partialinfty B_2$ at infinity, the visual angle $\psi$ of the $\partialinfty N_4, \partial_\infty N_5$ edge at $(0,0)$ is at most $\pi/6.$ 
  A similar argument holds for $\partialinfty N_k, \partialinfty N_1.$ 
  Therefore, the union of the remaining edges of $\lambda_2$ must have visual angle at least $2\pi/3.$
  When $k = 6$, this is impossible. Then $k = 7$, 
  it follows that each of $N_5, N_6, N_7$ are also full-sized and tangent to $B_1.$ 
  However, since $\psi \leq \pi/6$, this contradicts disjoints of $N_4$ and $N_5$.
  So $\phi < \pi$ as claimed.

We now argue by compactness to show that $\phi \leq \pi - \eps_1$ for some universal $\eps_1 > 0$.
Let $P$ be the space of all minimal, unknotted, unblocked, $k \leq 7$-necklaces for which $5 \pi / 6 \leq \phi \leq \pi$.
We claim that it is compact and therefore the supremum of $\phi$ is attained and is $\pi - \eps_1$ for some $\eps_1 > 0$. 
For every $\eta'$ in $P$, $\partial_\infty \eta'$ must lie in a closed annulus around $(0,0)$.
This is because the projection of each tie is length at most $1$ and therefore the necklace cannot be too far from $(0,0)$ as this would contradict $\phi \geq 5 \pi /6$.
Also note that $\partial_\infty \eta$ cannot be too close to $(0,0)$, as an isometry swapping $B_1$ and $B_2$ does not change $\phi$ and we can replay the previous argument.
Since the disjointness conditions for the horoballs are all closed, it follows that $P$ is compact and we are done.

Finally, since $\phi \leq \pi - \eps_1$, we can bound the distance from $\gamma(B_1, B_2)$ to the convex hull of $\fram(\eta)$ as follows.
By construction, $hull(\fram(\eta)) \subset hull(W_\phi)$.
Notice that the distance between $\gamma(B_1, B_2)$ and $hull(W_\phi)$ is entirely controlled by $\phi$ as $hull(W_\phi)$ lies outside a tubular neighborhood of $\gamma(B_1, B_2)$ whenever $\phi < \pi$.
Thus, there is a lower bound for this distance in terms of $\eps_1$, giving a universal bound for $d(\gamma(B_1, B_2),  hull(\fram(\eta)))$.

\end{proof}  

\begin{definition}[Transverse hull]
  By a suitable choice of $0 < \delta < \epsilon$, we may ensure
  that $\sK = \sK_\delta$ is transverse to all horoballs and shaved horoballs,
  and that $\partial \sK \cap \gamma_t = \emptyset$ for all $t \in \cT'$.
  We call such a $\sK$ a \emph{transverse hull}.
\end{definition}

\subsubsection{Escape planes}

\begin{definition}
  For $t \in \mT',$ an \emph{escape plane} $\hat P_t\subset (\bH^3\cup \Sinfty)\rsetminus(\inte(\shaved(\eta))\cup\inte(B_t))$  is a smooth properly embedded disc with $t\cup\partialinfty B_t\cup \partialinfty N_t\subset \hat P_t$ and $\partial \hat P_t\subset \Sinfty \cup \partial \hat B_t\cup \partial \hat N_t.$ 
  $\hat P_t$ is said to be \emph{generic} if $\hat P_t\cap (\mT\cup\mF) = t,$ $\hat P_t\cap\partialinfty \mH = \partialinfty B_t\cup \partialinfty N_t.$
  Note that $P_t$ is orthogonal to both $N_t$ and $B_t.$
  See Figure \ref{fig:escape}.
\end{definition}

\begin{figure}[h]
  \begin{center}\begin{overpic}[scale=0.7]{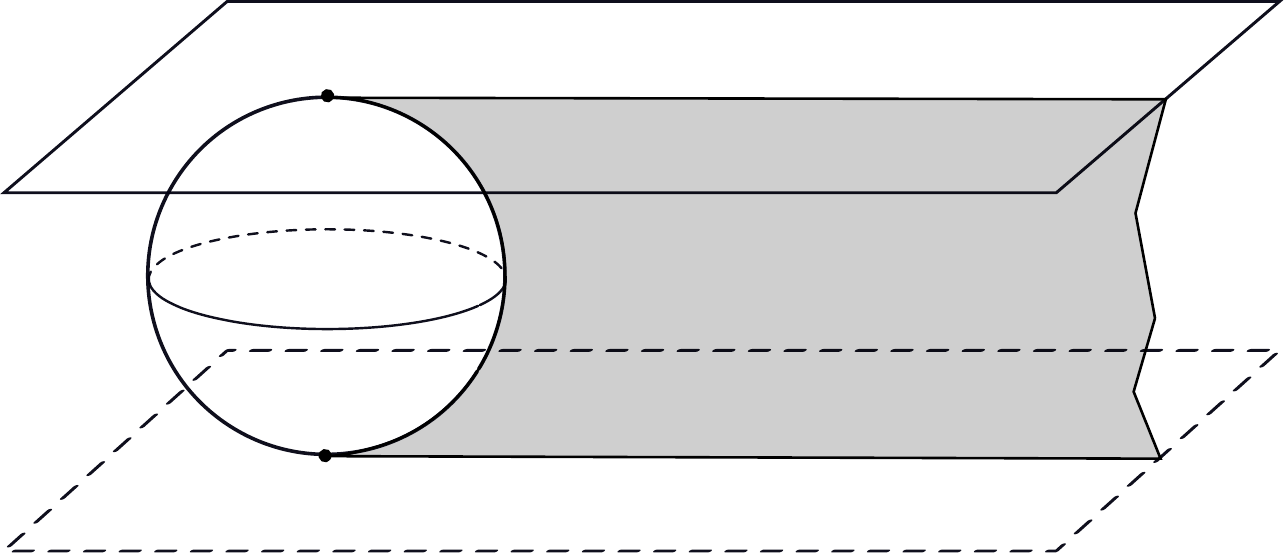}
      \put(24,38){$N_t$}
      \put(24,19){$B_t$}
      \put(60,20){$P_t$}
    \end{overpic} 
  \end{center}
  \caption{An escape plane with $N_t$ at infinity and $B_t$ the full sized horoball.}
  \label{fig:escape}
\end{figure}

\begin{notation} Let $\pi : \bH^3 \to \bR^2$ be the vertical projection in the upper half-space model. \end{notation}

\begin{lemma}[Escape Planes] \label{lem:escape-plane}
  Suppose given a $\le 7$-necklace $\eta =(N_1,\cdots, N_k)$ in $(\mH,\mT).$
  There exists a set $\{\hat P_t\}_{t\in \mT'}$ of generic, totally geodesic escape planes transverse to $\sK$ such that
  \begin{enumerate}
  \item the intersection with $\bH^3$ is a locally finite set, and\label{it:loc_fin}
  \item for distinct $s,t\in \mT',$ each component $\al$ of $\hat P_t\cap \hat P_s$ is an arc with an endpoint on $\Sinfty,$ $\partial N_t$ or $\partial N_s.$\label{it:intersection_arc}
  \end{enumerate}
\end{lemma} 

\begin{proof}
  Let $t\in \mT'$ be a tangency between $B_t\notin \eta$ and $N\in \eta.$
  By reindexing $\eta$ we can assume that $N=N_k.$ 
  Conjugate $\bH^3$ so that $N_k$ is the horoball $z\ge 1$ and $B_t$ is the full-sized horoball centered at $(0,0).$ 
  Let $\tau = \fram(\eta).$
  Let $\rho: \mathbb{R}^2 \setminus \{0\} \to S^1$ be the projection given by $\vec{v} \mapsto \vec{v}/\|\vec{v}\|$.
  Then $\rho(\pi(\tau))$ is a connected interval of length $\phi_t \leq 5\pi/3$ (that is, the visual angle of $\pi(\tau)$ from $(0,0)$ is at most $5\pi/3$), with equality holding exactly when $k = 7$ and $N_1, \ldots, N_6$ are full-sized and tangent to $B_t.$
  By taking the preimage under $\rho,$ we may identify $S^1$ with the rays in $\mathbb{R}^2$ originating at $(0,0),$ endowing this collection with the resulting topology.
  Let $\mathscr{R}'$ be the ray originating at $(0,0)$ that projects to the midpoint of $S^1 \setminus \rho(\pi(\tau)).$
  That is, $\mathscr{R}'$ is the bisecting ray of the ``wedge'' of rays from $(0,0)$ disjoint from $\fram(\eta).$

  We now show there are enough rays to guarantee the properties we require.
  First, either $\phi_t=5\pi/3$ or $\phi_t < 5\pi / 3.$
  If $\phi_t=5\pi/3,$ then $k = 7$ and $N_1$ and $N_6$ are full-sized and tangent.
  Let $p = N_1 \cap N_6.$
  Then $\mathscr{R}'$ passes through $\pi(p).$
  Let $\eta' = \eta \setminus N_k.$
  Since $p$ is a non-sequential tangency, and since these were shaved away, $\mathscr{R}'$ is disjoint from $\pi(shaved(\eta')).$
  If $\phi_t < 5\pi/3$ instead, then $\mathscr{R}'$ is disjoint from $\pi(\eta'),$
  See figure \ref{fig:escape_top}.
  In either case, $\mathscr{R}'$ is disjoint from the closed set $\pi(shaved(\eta')).$
  So the open set $S^1 \setminus \rho(\pi(shaved(\eta')))$ has a component $\kappa$ containing $\rho(\mathscr{R}').$
  Next, either $(0,0)$ is in the interior of the convex hull of $\pi(\tau)$ or not.
  If not, then every point on $\mathscr{R}' \setminus \pi(t)$ is further from every point of $\tau$ than $\pi(t).$
  Having every point further from every point of $\tau$ than $\pi(t)$ is an open condition, which we will call ``pointing away from $\tau.$''
  Furthermore, only countably many rays originating at $(0,0)$ pass through $\pi(\mT \cup \mF).$
  Also, for any countable collection $\Xi$ of circles, only countably many rays originating at $(0,0)$ are tangent to some element of $\Xi.$
  Finally, since $t \notin \partial \sK,$ only countably many rays originating at $(0,0)$ are not transverse to $\sK.$
  Therefore, for any countable collection $\Xi$ of circles, let $\sG$ be the set of rays $\mathscr{R}$
  that originate at $(0,0);$
  that miss $\pi(shaved(\eta')),$ $\pi(\tau),$ and $\pi(\mT \cup \mF);$
  that point away from $\tau$ if $\mathscr{R}'$ does;
  that are not tangent to any element of $\Xi;$
  and, finally, that are transverse to $\sK.$
  Then $\sG$ is a set of rays dense in an open cone of rays containing $\mathscr{R}'.$
  
  \begin{figure}[h]
    \begin{center}\begin{overpic}[scale=1.0]{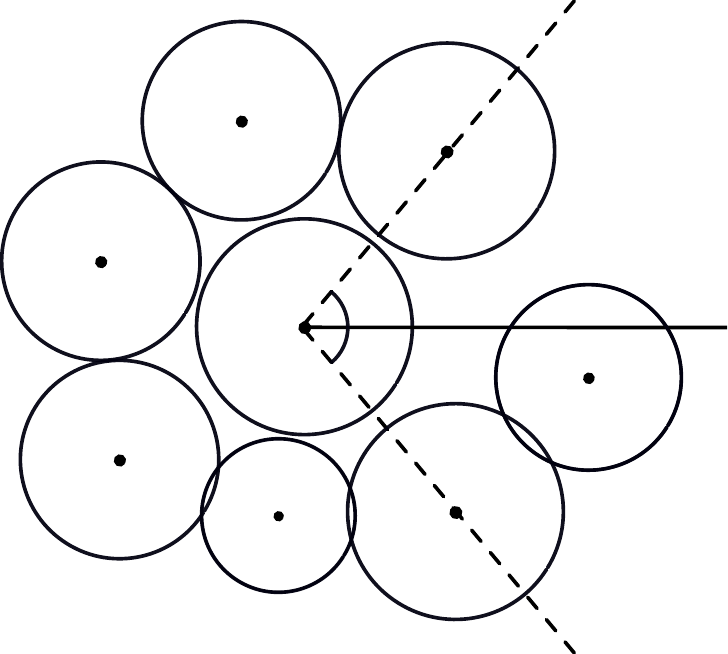}
        \put(48,48){$\psi_t$}
        \put(48,40){$\psi_t$}
        \put(34.5,43.5){$B_t$}
        \put(80,31){$B_s$}
        \put(105,43.5){$\mathscr{R}'$}
        \put(80,85){$\ell^+$}
        \put(80,5){$\ell^-$}
      \end{overpic} 
    \end{center}
    \caption{An escape plane with $N_t$ at infinity and $B_t$ full-sized at $(0,0).$}
    \label{fig:escape_top}
  \end{figure}

  We turn to the proof of item \ref{it:loc_fin}, local finiteness.
  For any ray $\mathscr{R}$ in $\sG,$ let $\hat P_\mathscr{R}$ be $\mathscr{R} \times [0,1] \setminus int(B_t)$.
  Then $\hat P_\mathscr{R}$ is a generic totally geodesic escape plane orthogonal to $N_t$ and $B_t,$ and its closure in $\overline{\mathbb{H}^3}$ is transverse to the closures of the planes determined by $\Xi.$
  Making any such choice of $\mathscr{R}$ for every $t \in \mT' \cap N,$ we get a family $\cF_N = \{\hat P_t\ |\ t \in \mT' \cap N\}.$
  We claim any such family is locally finite.
  To that end, note first that the set $C = \pi(\mT' \cap N)$ of sources of the rays is locally finite, since it is the set of centers of disjointly embedded discs of radius 1.
  Let $U_t$ be the union of $\{P \in \mathbb{R}^2\ |\ d(P,\pi(\tau)) < t\}$ with the convex hull of $\tau.$
  Since $\tau$ is bounded, so is $U_t$.
  Now, if a ray $R$ points away from $\tau$ and intersects $U_t,$ then its source is in $U_t.$
  Since $C$ is locally finite, only finitely many such rays can intersect $U_t.$
  So the set of rays is locally finite as claimed.
  Making any such choice of family $\cF_N$ with $\Xi = \emptyset$ for all the horoballs $N$ in $\eta$ proves item \ref{it:loc_fin}.

  To prove the second conclusion of the lemma, we are slightly more careful with our choice of escape planes.
  First set $\Xi_k = \emptyset.$
  Then in decreasing order, for all $k \geq i \geq 1,$ choose a family $\cF_{N_i}$ not tangent to $\Xi_i$ of escape planes for non-necklace tangencies to $N_i,$ and then define $\Xi_{i-1}$ to be the union of $\Xi_i$ and the boundaries of the new escape planes.

  For any two escape planes $\hat P_s$ and $\hat P_t$ so chosen (possibly associated to different horoballs of $\eta$), every component of $\hat P_s\cap \hat P_t$ is a compact geodesic arc $\hat \alpha$.
  (\textit{A priori}, since these are not convex objects, their intersection could be disconnected.)
  If $\hat\alpha\neq\alpha,$ then the second conclusion of the lemma is immediate. 
  If $\alpha = \hat \alpha,$ then we need to show $\partial \alpha$ does not lie in $B_t\cup B_s.$ 

  \begin{figure}[h]
    \begin{center}
      \begin{minipage}[c]{0.32\textwidth}
        \begin{overpic}[scale=1.0]{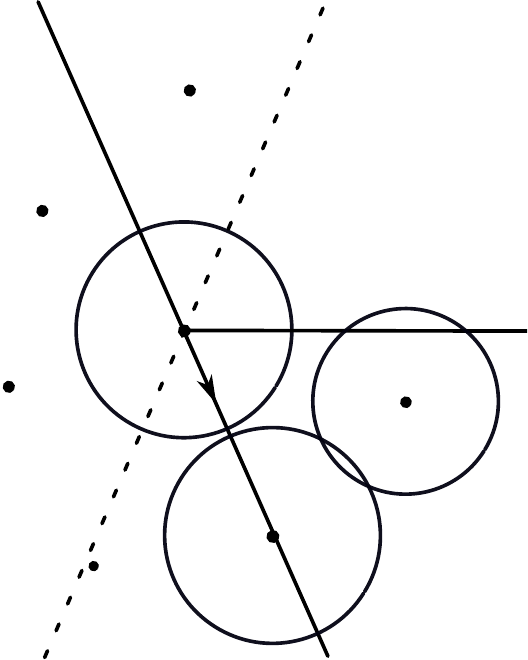}
          \put(19.5,47.5){$B_t$}
          \put(64,37){$B_s$}
          \put(44,18){$N_s$}
          \put(28,97){$V_s$}
          \put(77,52){$P_t$}
          \put(50,90){$\ell_{s+\pi/3}$}
          \put(2,90){$\ell_s$}
        \end{overpic} 
      \end{minipage}
      \hspace*{9ex}
      \begin{minipage}[c]{0.49\textwidth}
        \begin{overpic}[scale=1.0]{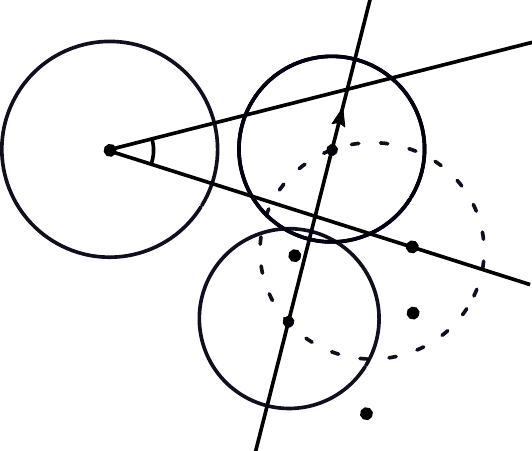}
          \put(15,50){$B_s$}
          \put(32.1,54.5){$\psi_s$}
          \put(63,49){$B_t$}
          \put(44,18){$N_t$}
          \put(85.5,18){$\ell_{s+\pi/3}$}
          \put(73,85){$\ell_s$}
          \put(96,82){$P_s$}
        \end{overpic} 
      \end{minipage}
    \end{center}
    \caption{A escape plane with $N_t$ at infinity and $B_t$ full-sized at $(0,0).$}
    \label{fig:escape_top_planes}
  \end{figure}
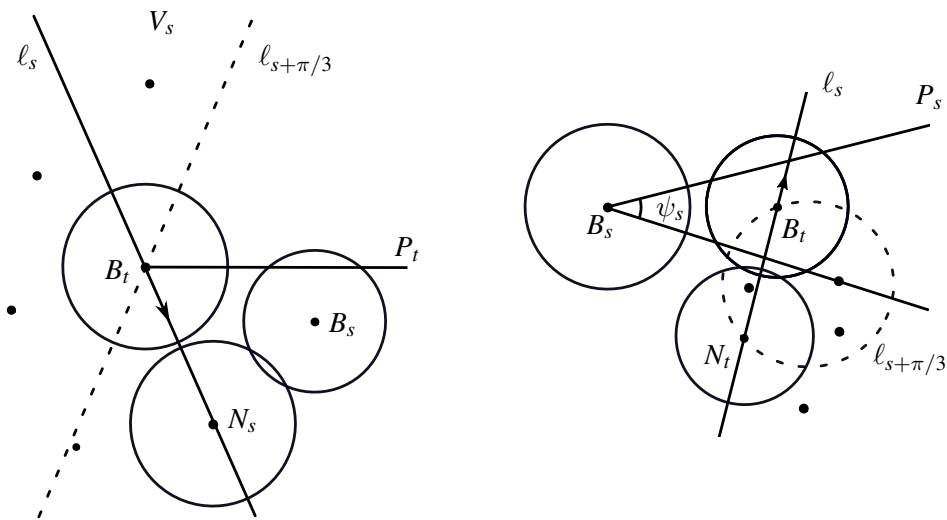

  Suppose for a contradiction that $\partial \alpha\subset B_t\cup B_s.$
  Then $P_t$ must intersect $B_s.$
  We claim that under this assumption, $P_s \cap B_t = \emptyset,$ a contradiction by symmetry.
  Now, $P_t$ intersects $B_s$ either in a disc as in Figure \ref{fig:escape_top} or in a single point.
  Consider first the perspective letting $N_t$ be the unit horoball at infinity, $B_t$ full-sized at the origin, and $P_t$ the positive $x$-axis.
  By reflection, if necessary, we may also assume $N_s$ lies below the $x$-axis, and this determines the perspective.
  Let $\ell_s$ be the oriented line from the center of $B_t$ to the center of $N_s$ and let $\ell_{s + \pi/3}$ be the line obtained by rotating $\ell_s$ by $\pi/3$ in the clockwise direction around the center of $B_t;$ see Figure \ref{fig:escape_top_planes}.
  Recall that $\phi_t$ is, from this perspective, the angle measure of the wedge of rays at $(0,0)$ intersecting the projection $\pi(\fram(\eta))$ of the frame.
  When $\phi_t = \pi,$ the arrangement of $B_t,$ $B_s,$ and $N_s$ is rigid---they are full sized and are part of the hexagonal horoball packing. 
  In this case, there is a half space containing $\gamma(B_s,N_s)$ that avoids all other points of $\partial_\infty \eta,$ so it must contain the escape plane $P_s.$ 
  This half space is disjoint from $B_t,$ so $P_s \cap B_t = \emptyset$ as claimed.
  If instead $\phi_t > \pi,$ then the $\pi/3$ wedge $V_s$ formed by $\ell_s$ and $\ell_{s+\pi/3}$ above the $x$-axis must contain a center at infinity of the necklace $\eta.$
  It follows that at least 3 such centers must lie to the right of $\ell_{s + \pi/3}.$ 
  Change perspective now as before, but with $s$ instead of $t,$ so that $N_s$ is the unit height horoball centered at infinity, $B_s$ $\ell_{s + \pi/3}$ becomes a circle bounding a disc. 
  By Lemma \ref{lem:vis_angle} in the Appendix, the visual angle of this disc from the center of $B_s$ is at most $\pi/3.$ 
  In particular, $\phi_s \leq \pi/3 + 2 \pi/3 = \pi.$ 
  It follows that $\psi_s \geq \pi/2.$ When $\psi_s > \pi/2,$ this contradicts the fact that $N_t$ and $B_t$ are tangent.
  Thus, $P_s$ cannot intersect $B_t.$
  The case of $\psi_s = \pi/2$ is identical to the case above where $\phi_t = \pi.$
\end{proof}

\subsubsection{Unblocked implies unlinked}

We now prove the main proposition of this subsection.

\begin{proof}[Proof of Prop. \ref{prop:unblocked_to_unlinked}.]
  Suppose $\eta$ is a minimal $k$-necklace with $k\le 7.$
  Construct a transverse hull $\sK$ and escape planes $\{P_t\}$ as above.
  Now, $\eta$ has an unknotting disc by Theorem \ref{thm:8unknot}.
  Thus $shaved(\eta)$ admits an unknotting disc $D$ whose boundary is the core $C(\eta)$ away from small neighborhoods of the false tangencies, and at most one arc per side of such a tangency.
  Now, there are arbitrarily small isotopies of proper embeddings $(D, \partial D) \to (\bH^3 - \eta, \partial \eta)$
  making $D$ transverse to $\sK$ and to all $\{P_t\}.$
  Since $C(\eta) \subset int(\sK),$ we may choose an isotopy so small, and in the appropriate direction, that also $\partial D \subset \sK.$
  In particular, $\partial D \cap \partial \sK = \emptyset.$
  We will prove the proposition by a base case and reduction to the base case.
  It is a proof by induction on a notion of complexity that we leave implicit.

  The base case is as follows.
  Suppose that $D \cap \partial \sK = \emptyset$ and that for all $t \in \cT',$ $D \cap P_t = \emptyset.$
  Then in particular, $D \subset \sK.$
  By construction (i.e.~by Lemma \ref{lem:universal_hull_radius}) $\sK \cap \cT'' = \emptyset,$ so $D \cap \cT'' = \emptyset.$
  Moreover, since $\cT' \subset \bigcup P_t,$ $D \cap \cT' = \emptyset.$
  Therefore, the interior of $D$ lies in the exterior of $\cT' \cup \cT''.$
  In particular, it lies in the exterior of $shaved(\eta).$
  So $\eta$ is unlinked, as desired.

  Now for the reductive steps.
  Suppose that $D \cap \partial \sK$ is not empty.
  We will now change $D,$ fixing $\partial D$ pointwise, to be disjoint from $\partial \sK.$
  For convenience, let $\sC = \closure{\bH^3 \setminus \sK}.$
  Since $\partial D \cap \partial \sK = \emptyset,$ by transversality $D \cap \partial \sK$ is a multicurve in the interior of $D.$
  We will proceed below by induction on its number of components.
  Before that induction, though, we prove the following claim.
  Suppose $\alpha$ is an innermost such curve.
  \begin{claim}
    $\alpha$ bounds a disc in $\hat{\partial \sK}$ disjoint from the vertices $\partialinfty \hat{\sK}.$ 
  \end{claim}

  \textbf{Proof of Claim.} 
  Since $\alpha$ is innermost, it bounds a disc $DD$ in $D$ with interior disjoint from $\partial \sK.$
  Also, $\alpha$ bounds a disc $KD$ in $\partial \hat{\sK}.$
  If $KD$ contains no vertices of $\hat{\sK},$ we are done.
  Suppose instead $KD$ contains some vertices of $\hat{\sK}.$
  Let $V$ be the vertices contained by $\Delta,$ and $W$ the other vertices.

  Suppose for a contradiction that $W$ is nonempty.
  Necessarily $DD \subset \sK$ or $DD \subset \sC$ since $DD \cap \partial \sK = \alpha.$
  Now, $\sC$ deformation retracts to $\partial \sK = \partial \sC.$
  So $\sC$ is boundary-incompressible (open compressing discs are not allowed here).
  Therefore, properly embedded discs in $\sC$ are trivial.
  In particular, their boundaries are inessential, and hence do not separate vertices of $\hat{\sK}.$
  Since $W$ and $V$ both are nonempty, $\alpha$ does separate vertices.
  Hence $DD$ cannot lie in $\sC.$
  Therefore $DD \subset \sK.$
  However, since $\fram(\eta)$ is connected, some tie $\gamma$ connects $V$ and $W.$
  Since $\alpha$ separates $V$ and $W$ in $\partial \hat{\sK},$ $DD$ separates them in $\hat{\sK}.$
  Thus $DD$ intersects $\gamma$ and hence intersects $\fram(\eta).$
  This contradicts $D$ being a compressing disc for $\eta.$

  Therefore, $W$ must be empty.
  Consequently, if $KD$ contains any vertices of $\hat{\sK},$ then $KD$ must contain all vertices of $\hat{\sK}.$
  Thus the complementary disc, also bounded by $\alpha,$ contains no vertices of $\hat{\sK},$ proving the claim. $\blacksquare$
  
  Now, let $KD$ be such a disc in $\partial \hat{\sK}.$
  Then $KD \subset \partial{\sK},$ being disjoint from the vertices.
  Let $DD$ be an innermost disc for $\alpha$ in $D.$
  Since $\partial KD = \partial DD = \alpha,$ $D' = (D \setminus DD) \cup KD$ is a compressing disc for $\eta.$
  Let $R$ be a regular neighborhood of $KD.$
  If $DD \subset \sC,$ push $D'$ into $\sK$ in $R;$ if instead $DD \subset \sK,$ push $D'$ into $\sC$ in $R.$
  This yields a new compressing disc $D''$ for $\eta$ transverse to $\partial \sK$ with $|\pi_0(D'' \cap \partial {\sK})| = |\pi_0(D \cap \partial {\sK})| - 1.$
  Thus by induction, we may assume $D \subset \sK.$
  Moreover, none of the above changes modified $\partial D$.

  To conclude the proof of Proposition \ref{prop:unblocked_to_unlinked}, it will not suffice just to show that we may change $D$ to be disjoint from the escape planes.
  We must also ensure that we can do so \textit{maintaining} $D \subset \sK.$
  First, we will change $D$ to have boundary disjoint from the escape planes.
  Finally, we will use another innermost disc argument to change $D$ to have no interior intersections with the escape planes.

  We perform the first isotopy one necklace horoball at a time.
  Choose a necklace horoball, and relabel so that it is the last one, $N_k.$
  Conjugate $\bH^3$ so that $N_k$ is the $z\ge 1$ horoball.
  Let $\pi:\bH^3\to \partial N_k$ denote the orthogonal projection.
  Let $\lambda$ denote $\pi(\fram(\eta)),$ a piecewise linear path from $N_1\cap N_k$ to $N_{k-1}\cap N_k.$ 
  Now, the $\hat P_t$'s intersecting $N_k$ are orthogonal to $\partial N_k.$
  Moreover, all $\hat{P}_t$'s are disjoint from $\shaved(\eta).$
  Thus $(\cup \hat P_t)\cap \lambda=\emptyset.$ 
  Note that the convex hull of $\lambda$ in $\partial N_k$ lies in $int(\sK).$
  Choose a regular neighborhood $U$ of this hull in $\sK.$
  Now, the previous part of the proof did not modify $\partial D.$
  So we may still assume the arc $A = \partial D \cap \partial N_k$ lies arbitrarily close to the segment $\overline{P_k Q_k}$ of the core $C(\eta)$ on $\partial N_k,$ where $P_k = N_1\cap N_k$ and $Q_k=N_{k-1}\cap N_k.$
  In particular, $A$ and $\lambda$ both are paths between $P_k$ and $Q_k$ in the disk $U \cap N_k.$
  So $A$ is isotopic in $U \cap N_k$ fixing $P_k$ and $Q_k$ to $\lambda.$
  We may therefore isotope $D$ in $U$ through embeddings $(D, \partial D) \hookrightarrow (\sK, \sK \cap N_k)$ until $\partial D \cap N_k$ is close enough to $\lambda$ to be disjoint from all the $\hat{P_t}.$
  Note that all $N_i \cap N_k$ tangencies for a fixed $k$ lie in $\lambda,$ even if $\eta$ is not minimal, so the isotopy need not pass through such tangencies in its interior.
  After cycling through the indices from $k$ to $1$, we can assume $\partial D\cap \hat P_t=\emptyset$ for all $t,$ and still have $D \subset \sK.$
  By local finiteness of $\hat P_t$ and boundedness of $D,$ we may also assume that $D$ is transverse to each $\hat P_t,$ and that $D\cap \hat P_t\neq \emptyset$ for only finitely many $t.$ 

  Finally, suppose that for some $t \in \cT'$ that $D \cap \hat{P}_t$ is nonempty.
  A component of this intersection is either a circle or an arc.

  Suppose $D\cap \hat P_t$ contains a circle component.
  Let $\delta$ be one innermost in $\hat P_t.$
  Let $E\subset \hat P_t$ denote the disc bounded by $\delta$ in $\hat P_t.$
  Now $\hat{P}_t$ is geodesic and $\sK$ is convex, so $\hat{P}_t \cap \sK$ is convex.
  But $\alpha \subset \hat{P}_t \cap \sK.$
  So $E \subset \hat{P}_t \cap \sK$ and hence $E \subset \sK.$
  In fact, since $\alpha \subset int(\sK),$ $E \subset int(\sK).$
  So we may compress $D$ along $E$ in $int(\sK).$
  Let $D_1$ denote the resulting disc component in $\sK.$
  Now, by Lemma \ref{lem:escape-plane}, for any other tangency $s \in \cT',$ each component of $\hat P_s\cap \hat P_t$ is an arc with an endpoint on $\partial \hat P_t.$
  Thus $E\cap \hat P_s = \emptyset$ if and only if $\delta \cap \hat P_s = \emptyset.$
  Doing the compression sufficiently close to $E,$ we can guarantee that if $s \neq t,$ then $D_1 \cap \hat P_s$ has as many arc components and circle components as $D\cap \hat P_s$.
  Moreover, we may also maintain that $int(D_1) \subset int(\sK)$ and that $int(D_1) \cap shaved(\eta) = \emptyset.$
  Thus by induction on the number of circle components of $D \cap \hat P_t,$ we may ensure that no component of $D \cap \hat P_t$ is a circle, for every $t.$

  To complete the proof, suppose instead that $D\cap \hat P_t$ contains an arc component $\alpha.$ 
  Since $\partial D\cap P_t=\emptyset$ and $int(D) \cap \eta = \emptyset,$ it follows that $\partial \alpha\subset \partial B_t.$ 
  We can assume that $\alpha$ is an innermost such arc in $P_t.$ 
  Let $E$ denote the half disc bounded by $\alpha$ and an arc in $B_t.$
  Since $\hat{P}_t$ is geodesic and $\sK$ is convex, as above we have $E \subset int(\sK).$
  We may isotope $D$ to $D_1$ by boundary compressing $D$ near $E.$ 
  By Lemma \ref{lem:escape-plane}, no arc component of $P_s\cap P_t$ has endpoints in both $B_s$ and $B_t.$
  Thus $E\cap \hat P_s = \emptyset$ if and only if $\alpha \cap \hat P_s = \emptyset.$
  Again, isotoping sufficiently close to $E,$ we can guarantee that for $s\neq t,$  $D_1 \cap \hat P_{s}$ has as many components as $D_1 \cap \hat P_s\neq \emptyset.$
  So by induction on the number of arc components, we can isotope $D$ such that no component of $D \cap \hat P_t$ is an arc, for every $t.$
  Thus we can isotope $D$ to be disjoint from every $\hat P_t$ as desired.
\end{proof}

\subsection{Globally minimal $\le 7$-necklaces in geometric systems are unblocked} 

In this subsection, we focus on the geometric constraints of horoball necklaces. We will prove:

\begin{proposition}\label{prp:min7unblocked}
  If $\eta$ is a globally minimal $\le 7$-necklace in an oriented geometric horoball system $(\mH, \mT),$ then $\eta$ is unknotted, unblocked, and unlinked.
\end{proposition}

The proof will require a lot of technical geometric control of $\leq 7$-necklaces. Our goal is to prove that $\leq 7$-necklaces cannot wind around a ``blocking'' tie $\gamma_b$ in a geometric horoball system. Note, this will require a good deal of work as this is false in the non-orientable case even for $6$-necklaces, see \cite{AdamsKnudson:2013}.

We start with a few basic facts about horoball geometry.

\begin{lemma}[Horoball distance]\label{lem:hd} 
  Let $B_1, B_2$ be two horoballs with disjoint interiors in the upper-half-space model with $b_i = \partialinfty B_i \in \bR^2$ and of Euclidean heights $h_i.$ 
  Then the hyperbolic distance $d_\bH(B_1,B_2)$ between $B_1, B_2$ is given by $$d_\bH(B_1,B_2) = \log\left(\frac{d_\mathbb{E} (b_1, b_2)^2}{h_1 h_2}\right).$$
\end{lemma}

\begin{proof}
  Consider the point $b_2' = b_1 - (b_2 - b_1)$ and let $\gamma$ be the geodesic between $b_2, b_2'.$ 
  Note that $\gamma$ has Euclidean radius $d_\mathbb{E} (b_1, b_2).$ 
  The highest point $P$ of $\gamma$ lies directly above (or below) the highest point of $B_1$ in the upper-half-space model. 
  In particular, we have the distance $d_\bH(P, B_1) = \log \left( d_\mathbb{E} (b_1, b_2)/h_1 \right).$
  Let $R_\gamma$ be $180^\circ$ rotation around $\gamma$.
  In the plane through $b_1,$ $b_2,$ and $\infty,$ $R_\gamma$ acts as reflection in $\gamma,$ or in Euclidean terms, inversion in $\gamma$'s associated circle.
  Thus under $R_\gamma,$ $B_1$ maps to a horoball at infinity of Euclidean height $d_\mathbb{E} (b_1, b_2)^2/h_1.$ 
  Since $B_1, B_2$ had disjoint interiors, it follows that $h_2 \leq d_\mathbb{E} (b_1, b_2)^2/h_1$ and
  \[d_\bH(B_1,B_2) = \log\left(\frac{d_\mathbb{E} (b_1, b_2)^2}{h_1 h_2}\right).\]
\end{proof}

\begin{lemma}[$\pi/3$-angle]\label{lem:visual} 
  Let $B_1, B_2$ be two horoballs with disjoint interiors in the upper-half-space model with $b_i = \partialinfty B_i \in \bR^2$ and of Euclidean heights $h_i.$ 
  Then the hyperbolic distance $d_\bH(B_1,B_2)$ between $B_1, B_2$ is given by $$d_\bH(B_1,B_2) = \log\left(\frac{d_\mathbb{E} (b_1, b_2)^2}{h_1 h_2}\right).$$
\end{lemma}

\begin{definition}
  Let $\gamma$ be a geodesic in $\bH^3$ with endpoints $b_1, b_2,$ let $\eta$ be a $k$-necklace with $\fram(\eta) \cap \gamma = \emptyset,$ and let $\rho_2 \in Isom(\bH^3)$ map $b_2$ to $\infty.$ 
  The \emph{winding count of $\eta$ around $\gamma$} is the absolute value of the winding number of $\pi(\rho_2(\fram(\eta)))$ around $\rho_2(b_1).$
\end{definition}

\begin{remark}
  This definition is independent of the orientation of $\eta,$ exchanging endpoints of $\gamma,$ and the choice of $\rho_2.$
  Also, this definition is equivalent to the winding count around $b_1$ (resp. $b_2$) of the piecewise circular arc given by circles through $b_2$ (resp. $b_1$) and consecutive points of $\eta$ at infinity.
\end{remark}

\begin{lemma}[Two eyes]\label{lem:two_eyes}
  Let $B_1, B_2$ be two disjoint full-sized horoballs and $\eta$ a $k$-necklace of at most full-sized horoballs. 
  Assume that $\{B_j , N_i\}$ have disjoint interiors for all $i,j$ and that $\eta$ has positive winding count around both $\gamma(B_1,H_\infty)$ and $\gamma(B_2,H_\infty).$ 
  Then, $k > 7.$
\end{lemma}

\begin{proof}
  We will use a visual angle argument to understand the positions of $\partialinfty N_i.$ 
  Assume that $k \leq 7.$ 
  We can translate and rotate so that $(0,0) = b_1 = \partialinfty B_1$ and $(d,0) = b_2 = \partialinfty B_2$ in $\bR^2,$ so that $\pi(\text{int}(B_j))$ are disjoint open discs of radius $1/2,$ and therefore $d \geq 1.$ 
  Our hypothesis on winding count implies that every ray in $\bR^2$ from $b_j$ must intersect some $\pi(\gamma_i).$ 
  In addition, since all horoballs are disjoint and at most full size, the visual angles of $\pi(\gamma_i)$ from $b_j$ are at most $\pi/3.$
  
\begin{figure}[h]
  \begin{center}\begin{overpic}[scale=.3]{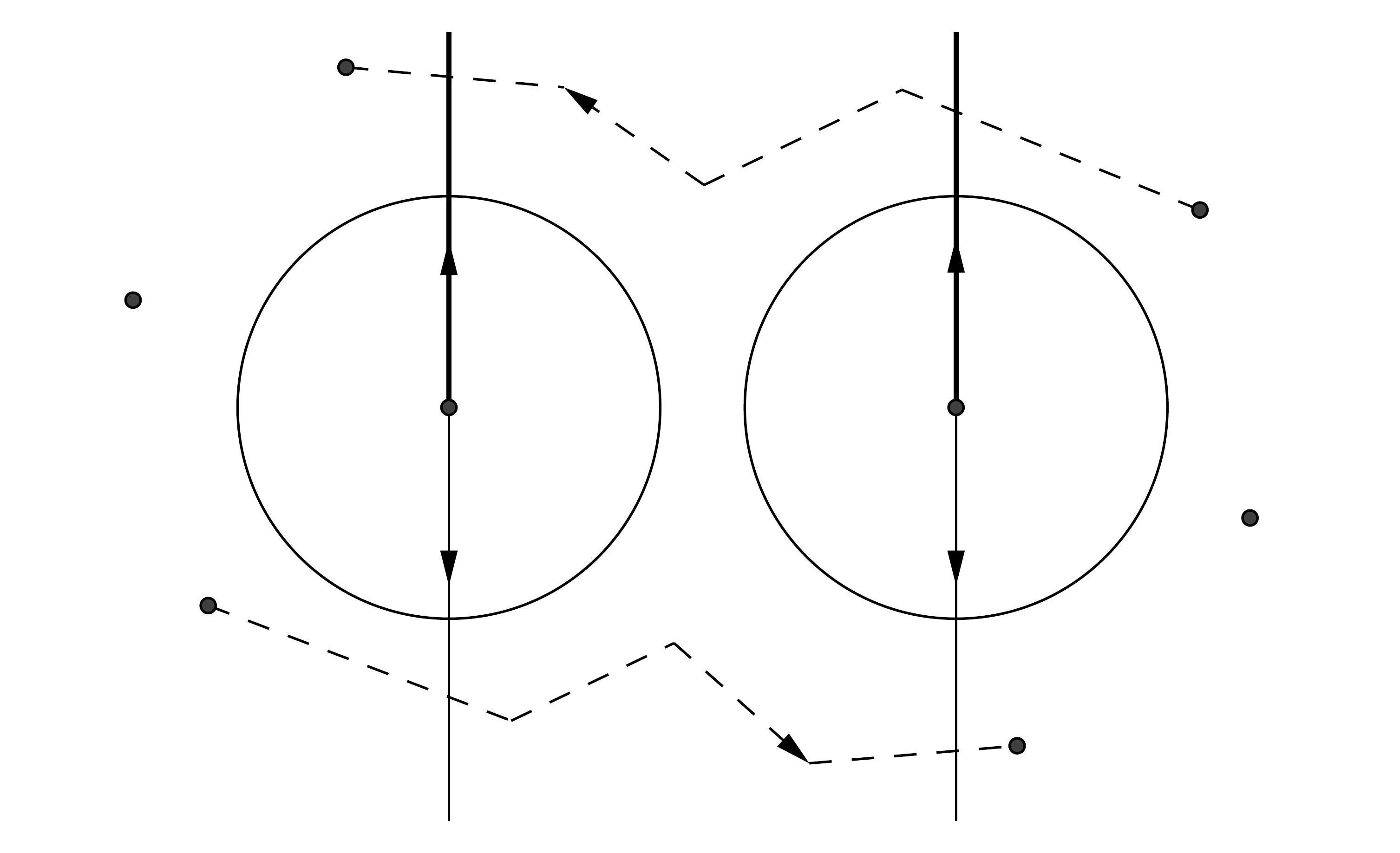}
      \put(33,30){$b_1$}
      \put(72,30){$b_2$}
      \put(-5,55){$D_1$}
      \put(95,55){$D_2$}
      \put(8,17){$n_1^L$}
      \put(17,57){$n_1^F$}
      \put(87,43){$n_2^L$}
      \put(76,5){$n_2^F$}
      \put(34,41){$\ell_1^+$}
      \put(25,20){$\ell_1^-$}
      \put(71,41){$\ell_2^+$}
      \put(63,20){$\ell_2^-$}
    \end{overpic}
  \end{center}
  \caption{Trying to wind around two full-sized horoballs.}
  \label{fig:two_eyes}
\end{figure}
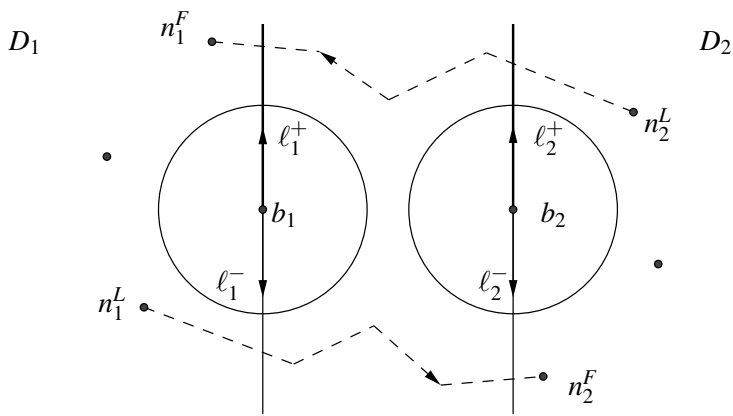

  Consider the rays $\ell_j^\pm$ from $b_j$ perpendicular to the $x$-axis and going in the positive and negative $y$-directions, respectively. 
  We can define the regions $$D_1 = \{ (x,y) \in \bR^2  \mid x < 0 \} \cup \ell_1^+ \text{ and } D_2 = \{ (x,y) \in \bR^2  \mid x > d \} \cup \ell_2^+.$$
  Using visual angle from $b_j,$ we see that each $D_j$ must contain at least two projected edges $\pi(\gamma_i),$ and therefore each $D_j$ contains at least $3$ points of tangency of $\eta.$ 
  Since $\eta$ is cyclically ordered, there are last and first points of tangency $n_1^L \in D_1,$ $n_2^F \in D_2$ going from $D_1$ to $D_2$ and $n_2^L \in D_2,$ $n_1^F \in D_1$ going from $D_2$ to $D_1.$ 
  By construction, there must be a sequence of edges connecting $n_1^L, n_2^F$ and $n_2^L, n_1^F.$ 
  Since the length of $\pi(\gamma_i)$ is at most $1,$  $n_1^L, n_2^F$ (or $n_2^L, n_1^F$) can be connected by only one edge if and only if $d = 1$ and $n_1^L \in \ell_1^+, n_2^F \in \ell_2^+$ (or $n_2^L \in \ell_2^+, n_1^F \in \ell_1^+$). 
  Since $k \leq 7,$ it follows that one of the pairs $n_1^L, n_2^F$ and $n_2^L, n_1^F$ is connected by only one edge.
  If it's $n_1^L, n_2^F,$ then $n_1^L \in \ell_1^-$ and $n_2^F \in \ell_2^-.$
  If it's $n_2^L, n_1^F,$ then $n_1^F \in \ell_1^+$ and $n_2^L \in \ell_2^+.$
  In either case, each $D_j \cup \ell_j^\pm$ must contain at least 3 projected edges, by a similar visual angle argument at $b_j.$
  This contradicts $k \leq 7$ and our proof is complete.
\end{proof}

\begin{lemma}[Three eyes]\label{lem:three_eyes}
  Let $B_1$ be a full-sized ball, $B_2, B_3$ be a pair of tangent horoballs of at most full size and $\eta$ a $k$-necklace of at most full-sized horoballs. 
  Assume that $\{B_j , N_i\}$ all have pairwise disjoint interiors and $\fram(\eta)$ has positive winding count around both $\gamma(B_1,H_\infty)$ and $\gamma(B_2,B_3).$ 
  Then, $k > 7.$
\end{lemma} 

\begin{proof}
  Label $b_j = \partialinfty B_j.$
  Let $d_{i\,j} = d_\bE(b_i,b_j)$ and $h_j$ denote the Euclidean heights of $B_j.$ 
  Without loss of generality, let $d_{1\,2} \leq d_{1\,3}.$
  Consider a convex disc $D_2$ with $b_2, b_3$ on its boundary. 
  We will make the choice of $D_2$ below. 
  Any isometry $\rho$ of $\bH^3$ that sends $b_3 \mapsto \infty$ also sends $D_2$ to a half-plane $\Pi_2.$ 
  Now $\fram(\eta)$ has positive winding count around $\gamma(B_2,B_3).$
  So by visual angle from $\rho(b_2),$ at least two edges $\pi(\rho(\gamma_i))$ lie in $\Pi_2.$
  If neither of these edges intersects $\partial(\Pi_2),$ then $int(\Pi_2)$ has three centers of infinity of $\rho(\eta).$
  If one of them does intersect $\partial(\Pi_2),$ then $\Pi_2$ contains at least three edges.
  In this case, $int(\Pi_2)$ would contain at least two centers of infinity of $\rho(\eta).$
  In both cases, $int(D_2)$ contains at least two centers of infinity of $\eta.$
  
\begin{figure}[h]
  \begin{center}\begin{overpic}[scale=.5]{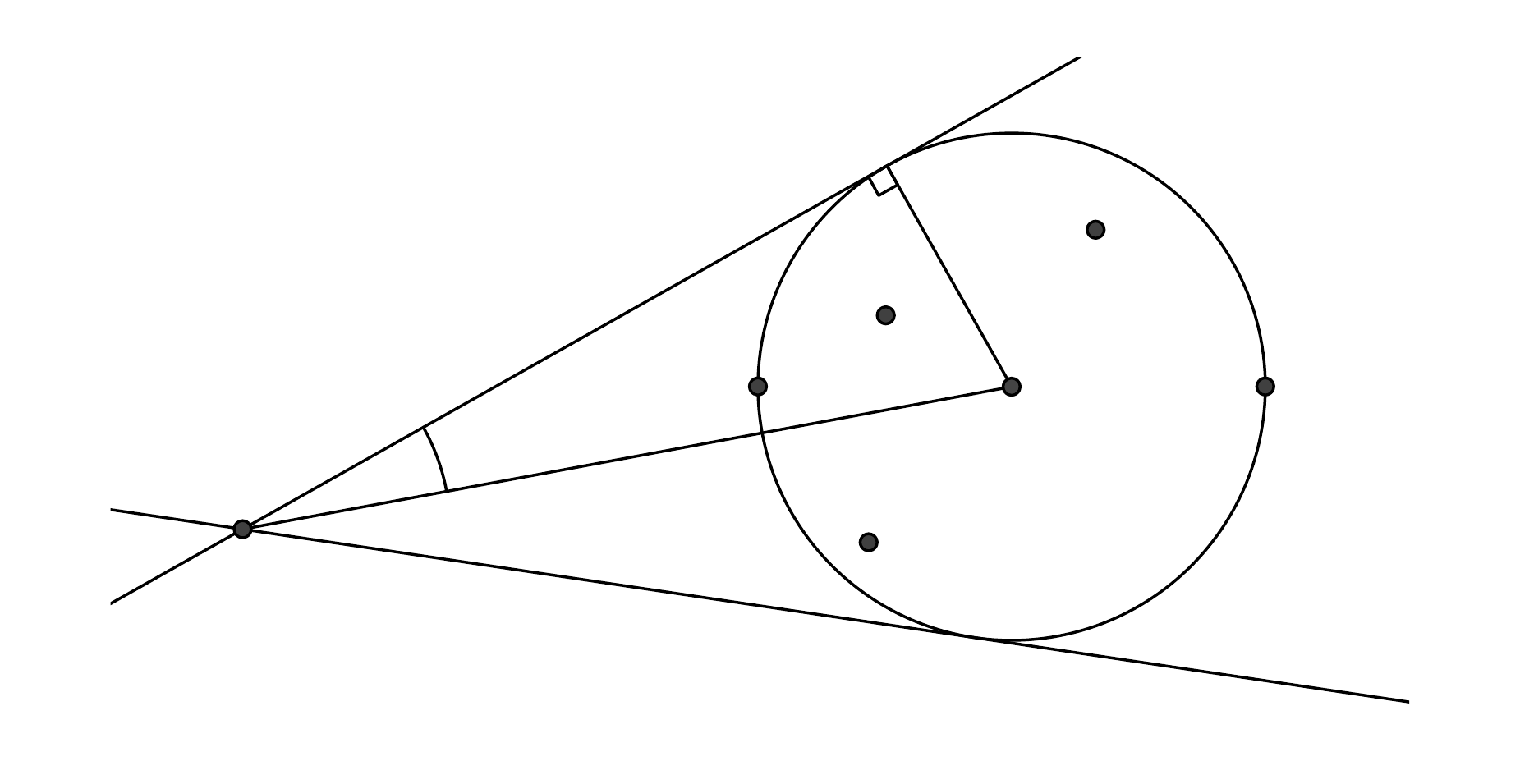}
      \put(12,19){$b_1$}
      \put(45,27){$b_2$}
      \put(86,25){$b_3$}
      \put(72,15){$D_2$}
      \put(30,21){$\alpha$}
      \put(64,34){$r$}
      \put(42,18){$c$}
    \end{overpic}
  \end{center}
  \caption{Diagram for Case 1 of Lemma \ref{lem:three_eyes}.}
  \label{fig:three_eyes}
\end{figure}
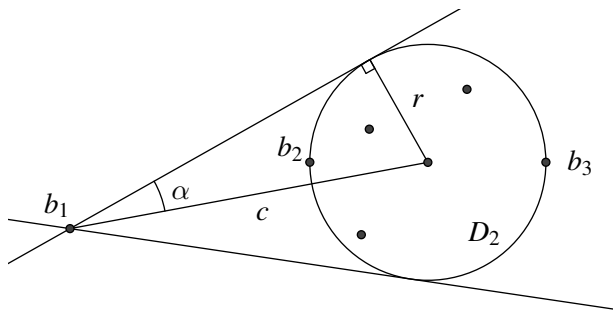

  Assume for a contradiction that the visual angle of $D_2$ from $b_1$ is at most $\pi/3.$ 
  Notice that the visual angle from $b_1$ of the two edges $\pi(\gamma_i)$ contained in $D_2$ is strictly less than $\pi/3$ since at least two centers at infinity of $\eta$ must lie in the interior of $D_2.$ 
  Assuming $k \leq 7,$ there are at most $5$ unaccounted edges $\pi(\gamma_i)$ that need to fit into strictly more than $5 \pi /3.$ 
  Since each edge has visual angle at most $\pi/3$ from $b_1,$ this is a contradiction.

  We will now choose the appropriate $D_2$ whose visual angle from $b_1$ is at most $\pi/3.$ 
  Let $c$ be the distance from $b_1$ to the midpoint of $b_{12}$ and $b_{23}.$

  {\it Case 1:}  $c > d_{12}.$ 
  Let $D_2$ be the disc of diameter $d_{23}$ with $b_2, b_3$ on its boundary. 
  Since $B_1, B_2$ have disjoint interiors, $d_\bH(B_1, B_2) \geq 0.$ 
  As $h_1 = 1,$ Lemma \ref{lem:hd} implies that $d_{12}^2 \geq h_2.$ 
  Similarly, since $B_2, B_3$ are tangent and at most full size, we have that $d_{23}^2 = h_2 h_3 \leq h_2$ and so $d_{23}/d_{12} \leq 1.$ 
  Consider the right triangle in Figure \ref{fig:three_eyes}. 
  By construction, $r = d_{23}/2$ and $c > d_{12}$ and we have $$\alpha = \arcsin\left(\frac{r}{c}\right) < \arcsin\left(\frac{d_{23}}{2 \, d_{12}}\right) \leq \arcsin\left(\frac{1}{2}\right) = \frac{\pi}{6}.$$
  It follows that the visual angle of $D_2$ is at most $2 \alpha < \pi/3.$

  \begin{figure}[h]
    \begin{center}\begin{overpic}[scale=.5]{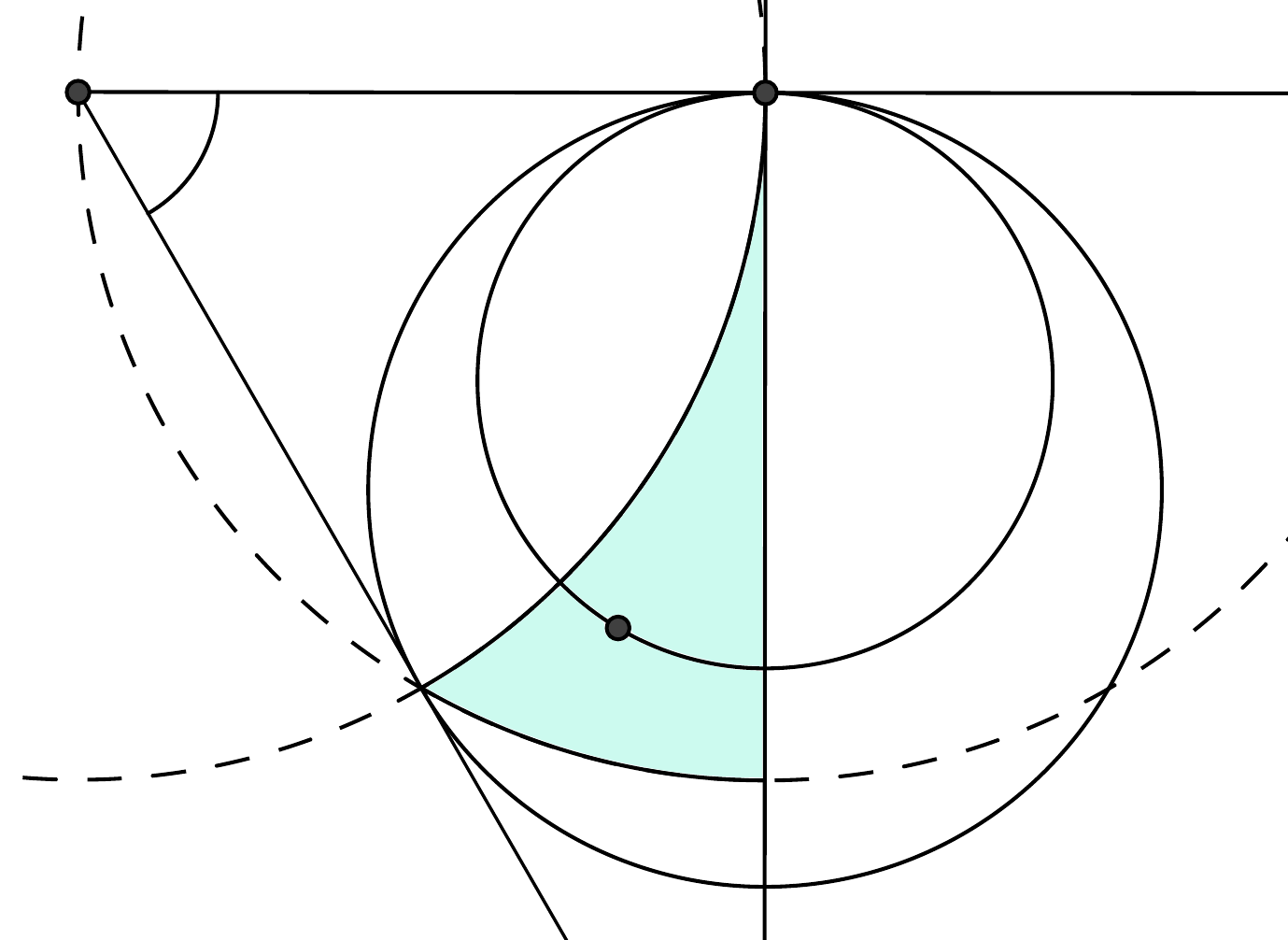}
        \put(0,65){$b_1$}
        \put(47,27){$b_3$}
        \put(54,68){$b_2$}
        \put(68,30){$D_2$}
        \put(91,40){$D'$}
        \put(17,58){$\pi/3$}
      \end{overpic}
    \end{center}
    \caption{Diagram for Case 2 of Lemma \ref{lem:three_eyes}.}
    \label{fig:three_eyes_2}
  \end{figure}
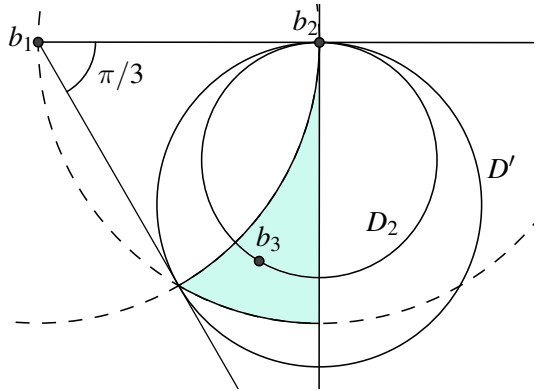

  {\it Case 2:} $c \leq d_{12}.$ 
  Let $D_2$ be the unique disc tangent to the ray from $b_1$ to $b_2$ which contains $b_2,b_3$ on its boundary, see Figure \ref{fig:three_eyes_2}. 
  The dashed circles in Figure \ref{fig:three_eyes_2} have radius $d_{12}.$ 
  The condition that $c \leq d_{12}$ and the fact that $d_{23} \leq d_{12},$ imply that $b_3$ lies in the shaded region denoted in Figure \ref{fig:three_eyes_2}. 
  Note that the region is somewhat larger than it needs to be. 
  Let $D'$ be the maximal disc tangent at $b_2$ to the ray from $b_1$ to $b_2$ with clockwise visual angle of $\pi/3.$ 
  By construction, for all $b_3$ in the shaded region, $D_2 \subset D'$ and therefore has visual angle at most $\pi/3.$
\end{proof}

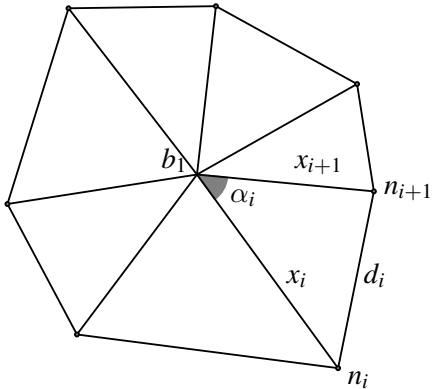
\begin{figure}[h]
  \begin{center}\begin{overpic}[scale=.7]{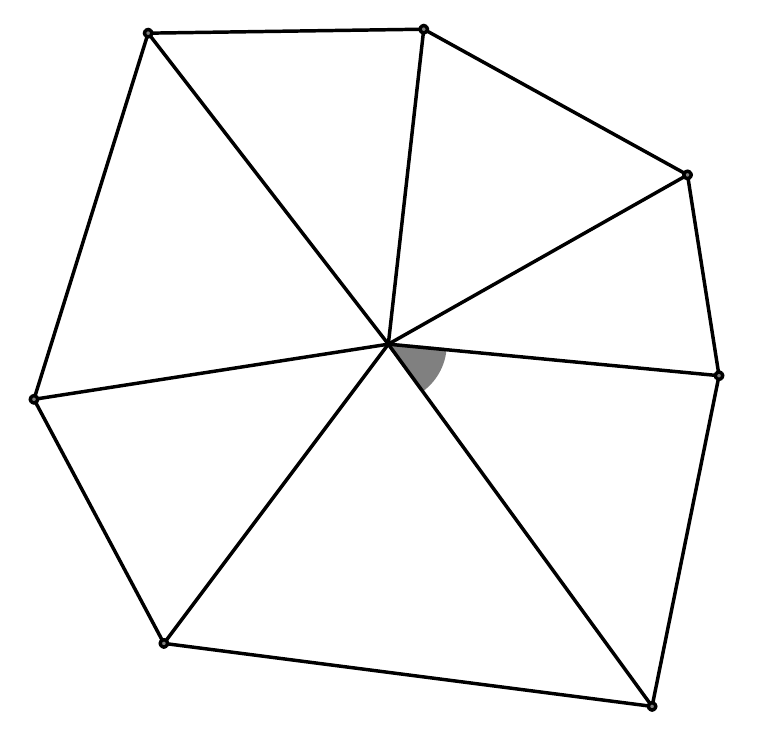}
      \put(42,53.4){$b_1$}
      \put(88,0){$n_i$}
      \put(97,47){$n_{i+1}$}
      \put(92,25){$d_i$}
      \put(59,45){$\alpha_i$}
      \put(73,25){$x_i$}
      \put(75,53.5){$x_{i+1}$}
    \end{overpic}
  \end{center}
  \caption{Diagram for Lemma \ref{lem:neckl_geom}.}
  \label{fig:angles}
\end{figure}

\begin{lemma}[Basic blocking properties]\label{lem:neckl_geom}
  Let $\eta$ be a globally minimal $\le 7$-necklace in a geometric horoball system $(\mH, \mT)$ and assume that $\eta$ is blocked by $B_1 \cap B_2 \in \mT'',$ where $B_2 = H_\infty$ of height 1.

  Define the following as in Figure \ref{fig:angles}.
  Let $n_i = \partial_\infty N_i \in \mathbb{C},$ $d_i = \ell_\mathbb{E}(\pi(\gamma_i)),$ $x_i = |n_i - b_1|,$ $\alpha_i = \angle \left( n_i - b_1, n_{i+1} - b_1\right),$ and $h_i = \text{height}_\mathbb{E}\left(N_i\right).$ Then for all $i$ one has
  \begin{enumerate}
  \item $\al_i \geq 0$
  \item $\pi/3 \leq \al_i + \al_{i+1} \leq 2 \pi / 3$
  \item $2 \pi / 3 \leq \al_i + \al_{i+1} +\al_{i+2} \leq  \pi$
  \item $d_i^2 = h_i \, h_{i+1}$
  \item $h_i \leq 1$ and $d_i \leq 1$
  \item $\sqrt{h_i} \leq x_i$
  \item $x_i \leq 2 \sqrt{h_i}$
  \item $1/4 \leq h_i$
  \end{enumerate}
\end{lemma}

\begin{proof}
  (4), (5), and (6) follow from Lemma \ref{lem:hd} and disjointness. 
  For (1), we use the fact that $\al_0 + \al_1 + \cdots + \al_6 = 2 \pi$ and that $\al_i \leq \pi/3$ by horoball packing. 
  For (2) and (3), we again use $\al_j \leq \pi/3$ for the complementary set of angles.
  
  Results (7) and (8) require a bit of calculus and geometry.
   Without loss of generality, let $i = 1$.
    The law of cosines gives $x_1^2 + x_2^2 - 2 \cos(\al_1) x_1 x_2 = d_1^2 = h_1 h_0$ and similarly for the triangle bounded by $x_0$ and $x_1$.
     Adding the two equations and solving for $x_1$ gives:
  \[2 x_1  = x_0 \cos \al_0 + x_2 \cos \al_1 \pm \sqrt{2 h_1(h_0 + h_2) - 2(x_0^2 + x_2^2) + (x_0 \cos \al_0 + x_2 \cos \al_1)^2}.\]
Since we want an upper bound, we take the $+$ term. Notice that this equation is symmetric in the pairs $(x_0, \al_0)$ and $(x_2, \al_1)$.
 Additionally, it is easy to see that one can always flow towards the diagonal $(x_0, \al_0) = (x_2, \al_1)$ while increasing the value of $x_1$, so there are no critical points on the interior of the realizable parameters.
 Further, this can be done while enforcing  $\pi/3 \leq \al_0 + \al_1$
  Note, flowing to the diagonal may take us through non-realizable parameters, but this doesn't matter as we need an upper bound.
   On the diagonal $\al_0 = \al_1 \geq \pi/6$ and basic calculus shows that $\al_0 = \al_1 = \pi/6$ maximizes the value of $x_1$
   This gives us $2 x_1 = \sqrt{3} x + \sqrt{2 h_1(h_0 + h_2) -x^2 }$ where $x = x_0 = x_2$. 
   Taking $h_0 = h_2 = 1$ further maximizes the value, with the final equation $2 x_1 = \sqrt{3} x + \sqrt{4 h_1- x^2}$ maximized at $x = \sqrt{3h_1}$ and value $x_1 = 2 \sqrt{h_1}$, giving us (7).

For (8), we solve the same sum of equations for $h_1$ instead of $x_1$, giving
\[h_1 = \frac{2x_1^2 + x_0^2 + x_2^2 - 2x_1(x_0 \cos\al_0 + x_2 \cos \al_1)}{h_0 + h_2}.\]
A similars observation show that it is possible to flow to the diagonal $(x_0, \al_0) = (x_2, \al_1)$ while keeping $\pi/3 \leq \al_0 + \al_1$  and decreasing $h_1$.
On the diagonal, $h_1$ is minimized when $\al_0 = \al_1 = \pi/6$.
Notice that there are parameters with $\al_0 = \al_1 = \pi/6$ that a realized by necklaces, so we can restrict to this case.
At such a realizing parameter, the horoballs $N_0$ and $N_2$ are visually $\pi/3$ apart. 
This forces all angles $\al_i$ for $i \neq 0, 1$ to be $\pi/3$, which in turn forces the horoballs $N_i$ for $i \neq 1$ to be full-sized and form the hexagonal packing.
It is then easy to see that the smallest Euclidean height for $N_1$ is $1/4$, giving us (8).  \end{proof}

\begin{lemma}[Blocking edge]\label{lem:has_blocking_edge}
  Let $\eta$ be a globally minimal $7$-necklace  in a geometric horoball system $(\mH, \mT).$ 
  If $\eta$ is blocked by $B_1 \cap B_2 \in \mT''$ and $\zeta = (N_1', \ldots, N_7')$ is a conjugate of $\eta$ containing $B_1$ and $B_2,$ then $\{B_1, B_2\} = \{N_i', N_{i+1}'\}$ for some $i.$\end{lemma} 

\begin{proof}
  The distance in $\zeta$ between $B_1, B_2$ is at most 3. 
  If $B_1 \cap B_2$ is not part of $\zeta,$ then adding this tangency to the shortest path between $B_1, B_2$ makes a $\leq 4$-necklace. 
  This contradicts global minimality of $\eta.$
\end{proof}

\emph{For the remainder of this section, $\eta$ will be a globally minimal $7$-necklace  in a geometric horoball system $(\mH, \mT)$ and will be blocked by $B_1 \cap B_2 \in \mT''.$}

Let $\gamma_b = \gamma(B_1, B_2)$ and for each tie $\gamma_i$ for $\eta,$ let $f_i$ be the element acting on $(\mH, \mT)$ that takes $\gamma_b$ to $\gamma_i.$ 
Set $\eta_i = f_i(\eta)$ and notice that $\eta_i$ is blocked by $\gamma_i$ by construction.

\begin{notation} For two necklaces $\eta$ and $\zeta$ in a horoball system, we use $\eta \Cap \zeta$ to denote the set of shared horoballs.
  This is their intersection as finite sets of horoballs, not their intersection as subspaces of hyperbolic space.
  The latter may include tangency points; the former does not.
\end{notation}

\begin{lemma}[Forced share]\label{lem:forced_share} If $\eta_i$ does not contain $B_1$ or $B_2,$ then $\eta_i \Cap \eta \neq \emptyset.$ \end{lemma} 

\begin{proof}
  Assume, without loss of generality, that $\eta_i$ does not contain $B_2$ and conjugate the picture so that $B_2 = H_\infty.$ 
  Let $\pi_i =  f_i \circ \pi \circ f_i^{-1}$ be the ``projection'' from $n_i$ or $n_{i+1},$ with notation as in Lemma \ref{lem:neckl_geom}. 
  By construction, the winding count of $\pi_i(\fram(\eta_i))$ is nonzero around either $n_i$ or $n_{i+1}.$ 
  As $\eta_i$ is blocked by $\gamma_i,$ $\pi_i(\fram(\eta_i)) \cap \pi(\fram(\eta))$ must contain a component lying entirely in $\pi(\gamma_i).$ 
  Positive winding count implies that $\pi_i(\fram(\eta_i)) \cap \pi(\fram(\eta))$ has at least two connected components. 
  A second component either contains $n_j$ for some $j$ or, for some $k \neq i,$ $\gamma_k$ passes under or over some tie $\kappa$ of $\eta_i.$ 
  However, the latter case violates Lemma \ref{lem:three_eyes} as either $\eta$ is blocked by $\gamma_b$ and $\kappa$ or $\eta_i$ is blocked by $\gamma_i$ and $\gamma_k.$ 
  Thus, $\eta \Cap \eta_i$ is non-empty and contains some horoball of $\eta.$
\end{proof}

\begin{lemma}[No accidental true tangencies]\label{no_acc_true_tangs} $B_k \cap  N_j \notin \mT$ for all $k = 1,2$ and $j = 1, \ldots, 7.$ \end{lemma}

\begin{proof}[Proof.]
  Assume otherwise. Without loss of generality, let $B_k = B_2$ and $N_j = N_1$ with $B_2 \cap N_1 \in \mT$.
  Send $B_2$ to $H_\infty$ and let $f$ acting on $(\mH, \mT)$ take $\gamma(B_1, H_\infty)$ to $\gamma(N_1, H_\infty).$
  Then $f(\eta)$ is a necklace of at most full-sized horoballs that winds around $\gamma(N_1, H_\infty).$
  Projecting to the plane, we obtain piecewise linear curves $\mu = \pi(\fram(\eta))$ and $\mu_f = \pi(\fram(f(\eta)))$, which are embedded by Lemma \ref{lem:neckl_geom}. 
  Each curve then divides the plane into two open regions: the ``inside'' containing $b_1$ or $n_1$ accordingly, and the ``outside''.
  
  Notice that if $\gamma$ is a tie of $\fram(f(\eta))$ with one endpoint inside $\mu$, the other outside $\mu$, and passes over $\fram(\eta)$, then $\eta$ has non-trivial winding count around $\gamma$. 
  However, this would contradict the Three Eyes Lemma \ref{lem:three_eyes} as $\eta$ also has non-trivial winding around $\gamma(B_1, H_\infty)$. 
  By symmetry, the same applies to any tie of $\fram(\eta)$ that has endpoints on either side of $\mu_f$ and goes over $\fram(f(\eta))$.

 Since $\mu$ and $\mu_f$ are both simple, there are three possible isotopy types, which we now consider.

  {\it Case 1:} $\mu_f$ winds around $b_1$.
  This is impossible by the the Two Eyes Lemma \ref{lem:two_eyes}.
  
   {\it Case 2:} $\mu_f$ does not wind around $b_1$ or $\mu_f$ does goes through $b_1$. 
   For this to hold, $\mu_f$ must intersect the half-open line segment $[b_1, n_1)$ between $b_1$ and $n_1$.
   In particular, $\mu_f$ must cross from outside of $\mu$ to the inside, and then again from inside to outside.
   Looking at the ties, we have seen that these crossing cannot arise from $\fram(f(\eta))$ passing over $\fram(\eta)$.
    Every underpass of $\fram(f(\eta))$ under $\fram(\eta)$ must be cancelled out by another underpass, because a tie of $\fram(\eta)$ lying over such an underpass must have endpoints on the same side of $\mu_f$.
    Thus, by parity, we must have at least two crossings via shared horoballs $A_1$ and $A_2$.
    We further assume that a neighborhood of $a_i$ in $\mu_f$ has a point inside $\mu$ for $i = 1,2$.
    This implies that that $A_1$ and $A_2$ cannot be neighbors in both $\eta$ and $f(\eta)$, as otherwise $\pi(\gamma(A_1, A_2))$ is an edge of both $\mu$ and $\mu_f$, so $\mu_f$ would not inside $\mu$.
   In particular, we get that the shortest paths between $A_1$ and $A_2$ in $\eta$ and $f(\eta)$ are distinct.
   Gluing them together would give a $\leq 6$-necklace, contradicting minimality.
   
 We have exhausted all possible avenues, so our proof by contradiction is complete. 
   \end{proof}

\begin{lemma}\label{lem:no_near} $\eta \Cap \eta_i$ cannot contain $N_{i+2}$ or $N_{i-1}$  \end{lemma}

\begin{proof}
  Assume otherwise. 
  Notice that $f_i^{-1}\left(\left\{ N_i, N_{i+1}\right\}\right) = \left\{B_1, B_2 \right\}$ and $f_i^{-1}(\eta_i) = \eta$ by definition. 
  If $N_{i+2}$ or $N_{i-1}$ are in $\eta_i$ then $f_i^{-1}(N_{i+2})$ or $f_i^{-1}(N_{i-1})$ are in $\eta.$ 
  Then, $f_i^{-1}(\gamma_{i+1})$ or $f_i^{-1}(\gamma_{i-1})$ is a tie between $B_k$ and $\eta,$ which contradicts Lemma \ref{no_acc_true_tangs}.
\end{proof}

\begin{lemma}\label{lem:has_both} If $\eta \Cap \eta_i \neq \emptyset$ then $\eta_i$ contains both $B_1$ and $B_2.$\end{lemma}

\begin{proof} 

  Normalize so that $B_2= H_\infty$ as before.
  Looking at the 7-necklace $\eta$ wrapping around $B_1,$ we heuristically see that there isn't very much room for the 7-necklace $\eta_i$ to snake around $N_i$ and $N_{i+1}.$
  By Lemma \ref{lem:no_near}, $\eta$ and $\eta_i$ can only share $N_{i+3}, N_{i - 2},$ or $N_{i + 4}.$
  By symmetry, the cases of $N_{i+3}$ and $N_{i-2}$ are identical.
  We will assume that $\eta \Cap \eta_i$ includes $N_{i+3}.$
  The case of $N_{i+4}$ only requires minor modifications to the arguments below. 

  We claim that $B_1$ and $B_2$ must both be in the necklace $\eta_i.$
  First, we prove that if we assume that neither $B_1$ nor $B_2$ is in $\eta_i,$ then we are led to a contradiction.
  Second, we will prove that if exactly one of $B_1$ or $B_2$ is in $\eta_i,$ then we are also led to a contradiction.

  {\it Case 1:} \noindent{\bf Assume neither $B_1$ nor $B_2$ is in $\eta_i.$}

  Normalize via a M\"obius action $g$ on $(\mH, \mT)$ sending $N_i$ to $H_\infty$ and $N_{i+1}$ to a full-sized horoball.
  We will denote $g(N_{i+3})$ as $C_3$; it is preceded and succeeded in the necklace $g(\eta_i)$ by $C_2$ and $C_4;$ see Figure \ref{fig:shared}.
  Note that we have drawn parts of the two necklaces $g(\eta)$ and $g(\eta_i),$ and these agree at $g(N_{i+3}) = C_3.$
  Note also that $C_2$ or $C_3$ could possibly be $g(N_{i+4}),$ but no other necklace horoball nor $B_1, B_2$.

  The necklace $g(\eta_i)$ consists of horoballs that are less than or equal to full-size, and is blocked by $g(\gamma_i).$
  In particular, $g(\eta_i)$ must encircle the full-sized horoball $g(N_{i+1}).$
  If $g(\eta_i)$ were a 6-necklace, then all the horoballs in it would have to be full-sized and form a hexagonal pattern with $g(N_{i+1}).$
  Because $g(\eta_i)$ is a 7-necklace we have more room to maneuver, but not much more room as seen in Lemma \ref{lem:neckl_geom}. 

  
    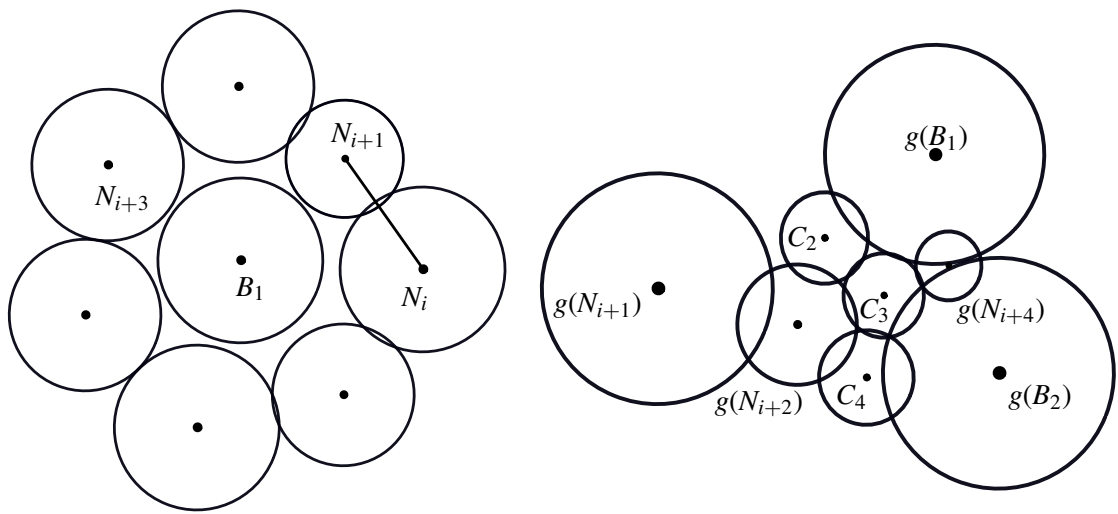
\begin{figure}[h]
    \begin{center}
      \begin{minipage}[c]{0.32\textwidth}
        \begin{overpic}[scale=1.0]{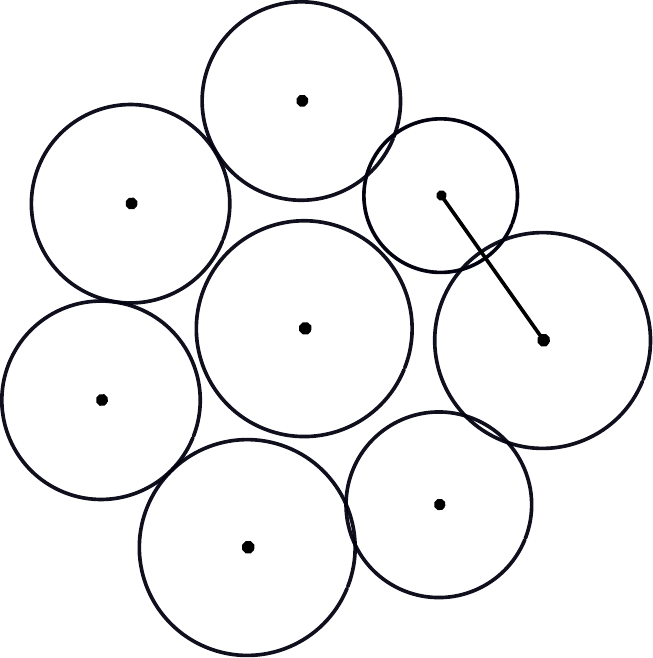}
          \put(45,43){$B_1$}
          \put(17,61){$N_{i+3}$}
          \put(78,41){$N_i$}
          \put(64,73.5){$N_{i+1}$}
        \end{overpic} 
      \end{minipage}
      \hspace*{10ex}
      \begin{minipage}[c]{0.55\textwidth}
        \begin{overpic}[scale=1.4]{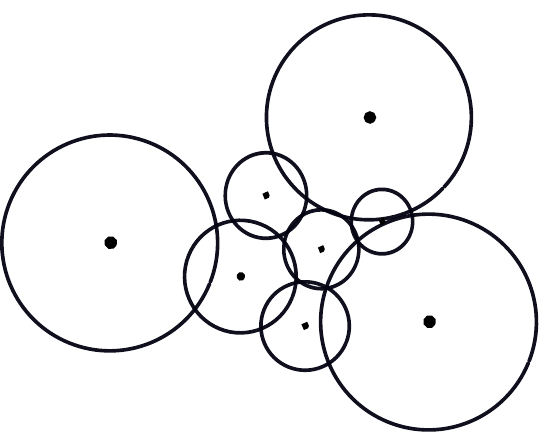}
          \put(2,31){$g(N_{i+1})$}
          \put(63,60){$g(B_1)$}
          \put(81,15){$g(B_2)$}
          \put(30,14){$g(N_{i+2})$}
          \put(72,30){$g(N_{i+4})$}
          \put(42.5,42.5){$C_2$}
          \put(55,30){$C_3$}
          \put(51,15){$C_4$}
        \end{overpic} 
      \end{minipage}
    \end{center}
    \caption{(Left) $Z$ horoball and view from $B_2$. (Right) View from $g(N_i) = H_\infty$ in Case 1.}
    \label{fig:shared}
  \end{figure}
  Our main focus will be on the five horoballs $C_2,$ $C_3,$ $C_4,$ $g(B_1),$ $g(B_2).$
  The horoball centers run around $C_3$ in a pattern either of the form $g(B_\ast), g(B_\ast), C_\ast, C_\ast$ (``BBCC'') or of the form $g(B_\ast), C_\ast, g(B_\ast), C_\ast$ (``BCBC'').
  (Degeneracies of collinearity can be construed in either of these categories,
  since the inequalities of Lemma \ref{lem:neckl_geom} are not strict.)
  We will first treat the former, more difficult case.

  {\it Case 1A:} $B_1,B_2 \notin \eta_i,$ $BBCC$ pattern.
  
  We will analyze four visual angles based at $C_3$: $\angle (C_2,\ C_3,\ C_4)$; $\angle (C_4,\ C_3,\ g(B_2))$; $\angle (g(B_2),\ C_3,\ g(B_1))$; and $\angle (g(B_1),\ C_3,\ C_2).$
  Of course, these angles can also be interpreted as dihedral angles at the geodesic running from $C_3$ to $\infty.$
  The four angles should add up to $2\pi$, but because of various constraints we will show that the calculated sum is in excess of $2\pi$, thereby obtaining a contradiction.

  {\it Claim I:} $\angle (g(B_2),\ C_3,\ g(B_1)) \geq 2\pi/3.$
  
  First note that this is the dihedral angle of the edge from $C_3$ to $\infty$ in the ideal tetrahedron formed by the centers of $g(B_2),\ C_3,\ g(B_1)$ and $\infty.$
  Consider the isometric tetrahedron formed by the centers of $B_2,\ N_{i+3},\ B_1,\ N_i$ (that is, apply $g^{-1}$ to the previous ideal tetrahedron).
  The dihedral angle at the edge from $C_3$ to infinity in the first tetrahedron corresponds to the dihedral angle at the edge from $N_{i+3}$ to $N_i$ in the second tetrahedron, but this dihedral angle is equal to the dihedral angle at the edge from $B_1$ to $B_2$ (note that the center of $B_2$ is infinity).
  The dihedral angle at $\gamma(B_1, B_2)$ is equal to the visual angle from the center of $N_i$ to the center of $N_{i+3}$ measured at the center of $B_1.$
  This visual angle must be at least $2\pi/3$ by Lemma \ref{lem:neckl_geom}. 

  This proves Claim I.

  {\it Claim II:} $\angle (C_2,\ C_3,\ C_4) \geq 2\pi/3.$

  First, we contend that the Euclidean height of $C_3$ is between $1/4$ and $1/3.$
  The $ \le 1/3$ constraint is equivalent to showing $d_\bH(N_i, N_{i+3}) \geq \log 3$.
  To prove this, consider the picture with $B_2$ at infinity and refer to Lemma \ref{lem:neckl_geom}.
  Since $\exp(d_\bH(N_i, N_{i+1})) = d_\bE(n_i, n_{i+3})^2/(h_i h_{i+3})$, we just need to show $d_\bE(n_i, n_{i+3})^2 \geq 3 h_i h_{i+3}$.
  The law of cosines gives 
  \[d_\bE(n_i, n_{i+3})^2 = x_i^2 + x_{i+3}^2 - 2 x_i x_{i+3} \cos \beta \geq x_i^2 + x_{i+3}^2 + x_i x_{i+3} \geq h_i  + h_{i+3} + \sqrt{h_i h_{i+3}} \geq 3 h_i h_{i+3},\] 
  where $\beta \geq 2 \pi/3$ is the visual angle measured at $B_1$ and we use the fact that $\sqrt{h_i} \leq x_i$ for all $i$. 
  The last inequality follows from basic calculus. 
  Thus, the height of $C_3$ is at most $1/3$.

      The lower bound of $1/4$ also comes from Lemma \ref{lem:neckl_geom} (8), where a 7-chain of at most full-sized horoballs surrounding a full-sized horoball (here $N_{i+1}$) must have all horoballs of Euclidean height at least $1/4$.

  Using these height bounds and the fact that $g(\eta_i)$ winds around $\gamma(g(N_i), g(N_{i+1}))$, we now show that the visual angle $\angle (C_2,\ C_3,\ C_4)$ is at least $2\pi/3$.
  \begin{figure}[h]
  \begin{center}\begin{overpic}[scale=1.1]{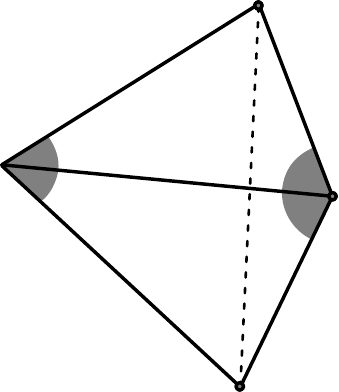}
      \put(18,47){$\alpha_1$}
      \put(18,61){$\alpha_2$}
      \put(-55,56){$g(N_{i+1})$ center}
      \put(65,0){$C_4$ center, height $h_1$}
      \put(90,47){$C_3$ center,  height $h$}
      \put(70,96){$C_2$ center, height $h_2$}
      \put(67,38){$\beta_1$}
      \put(67,62){$\beta_2$}
      \put(77,23){$d_1$}
      \put(77,75){$d_2$}
      \put(26,20){$y_1$}
      \put(26,84){$y_2$}
    \end{overpic}
  \end{center}
  \caption{Diagram for Case II of Lemma \ref{lem:has_both}.}
  \label{fig:case2}
\end{figure}
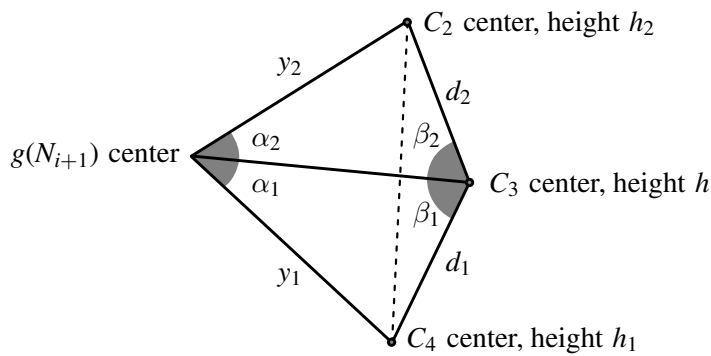
  The argument is similar to the arguments proving Lemma \ref{lem:neckl_geom}.  With labels as in Figure \ref{fig:case2}, the law of cosines gives 
  \[ \cos(\beta_1 + \beta_2) = \frac{d_1^2 - y_1^2 + d_2^2 - y_2^2 + 2 \cos(\al_1 + \al_2) y_1y_2}{2 d_1 d_2}\]
 One can then see that the maximum value of $\cos(\beta_1 + \beta_2)$ must be on the diagonal $(y_1, \al_1) = (y_2, \al_2)$ as a quick gradient analysis shows that we can always flow towards this diagonal while increasing $\cos(\beta_1 + \beta_2)$. Since necklace horoballs are tangent, we know $d_i^2 = h h_i$. With $y = y_i$ and $\al = \al_i$, we have
 \[\cos(\beta_1 + \beta_2)  \leq \frac{d_1^2 + d_2^2 - 4 \sin(\al) y^2}{2 d_1 d_2} = \frac{h h_1 + h h_2 - 4 \sin(\al) y^2}{2 h \sqrt{h_1 h_2}} = \frac{h_1 + h_2}{2\sqrt{h_1 h_2}} + \frac{2}{h} \frac{y}{\sqrt{h_1}} \frac{y}{\sqrt{h_2}} \sin^2(\al)\]
Since we only need to consider realizable parameters on the diagonal, Lemma \ref{lem:neckl_geom} gives that $y/\sqrt{h_i} \geq 1$ and $\al = (\al_1 + \al_2)/2 \geq \pi/6$, so we must have \[\cos(\beta_1 + \beta_2) \leq \frac{h_1 + h_2}{2\sqrt{h_1 h_2}} - \frac{2}{h} \frac{1}{4} \leq 1 - \frac{3}{2} = - \frac{1}{2}\]
 where we use the fact that $1/4 \leq h \leq 1/3$ and a straightforward upper bound on the first term as $1/4 \leq h_i \leq 1$. Thus, the visual angle $\angle (C_2,\ C_3,\ C_4) =  \beta_1 + \beta_2$ is at least $2\pi/3$. 
 
 This proves Claim II.
  
  If we can show that the other two angles $\angle(C_4,\ C_3,\ g(B_2))$ and $\angle(g(B_1),\ C_3,\ C_2)$ are each greater than $\pi/3$, then we will have our contradiction.
  To get control of these angles we first need lower bounds for the sizes of $C_2$ and $C_4.$
  To minimize $C_2$ we must first maximize $C_3,$ so we assume that the height of $C_3$ is $1/3.$
  We start with a warm-up example.
  Assume $C_2$ and $C_4$ are of the same height, $h^2.$
  Then the distance between the centers of $C_2$ and $C_3$ is $h/(\sqrt{3}),$ as is the $C_3$ to $C_4$ distance.
  The sum of these two distances must be greater than or equal to the distance between the centers of $C_2$ and $C_4,$ which, as we saw above, is at least 1.
  So the smallest allowable height for $C_2 = C_4$ satisfies $2h/\sqrt{3} = 1,$ and we see that the smallest height for $C_2 = C_4$ is $3/4.$
  We have finished the warm-up example.

  If we let one of the two horoballs, say $C_4,$ increase then the smallest height for the other horoball decreases.
  The worst case is when $C_4$ has height 1, in which case the distance between the centers of $C_3$ and $C_4$ is $1/\sqrt{3},$ and the sum of the distances between the centers of $C_2, C_3$ and $C_3, C_4$ is $(h + 1)/\sqrt{3}.$
  We compare that with the distance between the centers of $C_2$ and $C_4.$
  We can assume the center of $C_2$ is $(-1/2,\sqrt{h^2 - 1/4})$ and the center of $C_4$ is $(1/2,\sqrt{3}/2).$
  We can then calculate that the minimum height is $0.62127.$\footnote{
  One might wonder whether relaxing the condition that the horoball $C_2$ abuts the horoball $g(N_{i+1})$ in our setup might allow for a smaller $C_2$?
  (Similarly for $C_4.$)
  However, a straightforward calculation shows that rolling $C_2$ along $C_1$ does not allow for a decrease in the size of $C_2.$
  }
  With these height bounds established, we make the following claim.

  {\it Claim III:} $\angle (C_4,\ C_3,\ g(B_2)) > \pi/3.$

  We argue by contradiction and start by assuming that the angle is less than $\pi/3$.
  We know that $C_4$ abuts $C_3,$ and although $C_4$ and $g(B_2)$ might not abut, the smallest angle at $C_3$ will occur when $C_4$ abuts $g(B_2),$ so we will make this assumption. Note, we have freedom to rotate $g(B_2)$ around $C_3$ to make this happen since it is not a necklace horoball and since the height of $C_4$ is at least $1/4$ by Lemma \ref{lem:neckl_geom} (i.e under the rotation we cannot not miss $C_4$ because it is big enough).

  We begin by getting some control over the height of $g(B_2)$; this is the same as the height of $N_i$ in the original picture (both depend on the hyperbolic distance from $B_2$ to $N_i$).
  Consider the triangle formed by the centers of the 3 horoballs $N_{i+3},\ B_1,\ N_i$ in the original horoball diagram.
  It will be convenient to use some notation from [GMM]: We will work with the so-called Euclidean lengths between horoballs; given two horoballs separated by hyperbolic distance $o(n),$ the associated Euclidean length is $\exp(o(n)/2).$
  We will often associate to a single horoball $X$ its Euclidean distance denoted $e_X$ from $H_\infty$; we can compute that the Euclidean height of this horoball is then $1/e_X^2.$
  As computed before, if $e_p, e_q$ are Euclidean lengths from two horoballs to $H_\infty,$ and $e_r$ is the Euclidean length between those two horoballs, then the Euclidean distance between the centers of these two horoballs is ${e_r \over e_p e_q}$ as given by Lemma \ref{lem:hd}.

  In our set up, let $e_p$ and $e_q$ be the Euclidean lengths associated to the horoballs $N_i$ and $N_{i+3}$ respectively.
  Note that $e_p$ is the Euclidean length from horoball $N_i$ to horoball $B_2,$ which is also the Euclidean length from $g(N_i)$--which is $H_\infty$--to $g(B_2).$
  Let $e_r$ be the Euclidean length from horoball $N_i$ to horoball $N_{i+3},$ and $d$ the Euclidean distance between their centers.
  Note that the Euclidean length of the (full-sized) horoball $B_1$ is 1.

  Consider the triangle formed by the centers of $N_i,\ B_1,\ N_{i+3}$ with edges of length $y_1, y_2$ and $d$, which was defined above.
  Since $B_1$ has height $1$, we know that $y_1 \geq 1/e_p$ and $y_2 \geq 1/e_q$.
  Let $\alpha$ be the angle at the center of $B_1.$
  By the Euclidean Law of Cosines, we have $d^2 = y_1^2 + y_2^2 - 2\cos(\alpha)y_1y_2.$
  As observed above, $\alpha$ is between $2\pi/3$ and $\pi$.
  Hence, $1 \leq -2\cos(\alpha) \leq 2$ and $d^2 \ge 1/e_p^2 + 1/e_q^2 + 1/(e_p e_q).$
  So, $$(e_r/(e_p e_q))^2 \ge 1/e_p^2 + 1/e_q^2 + 1/(e_p e_q)$$
  $$e_r^2 \ge e_q^2 + e_p^2 + e_p e_q$$
  $$0  \ge e_p^2 + (e_q) e_p  + (e_q^2 - e_r^2)$$

  So, we have a quadratic inequality in $e_p$ and the solution is
  $$
  {-e_q - \sqrt{4e_r^2 - 3 e_q^2} \over 2} \le e_p \le {-e_q
    + \sqrt{4e_r^2 - 3 e_q^2} \over 2}
  $$

  The broadest range of values is when $e_r$ is maximized and $e_q$ is minimized.
  Because $e_r$ is the Euclidean length from $N_i$ to $N_{i+3},$ i.e., the Euclidean length from $g(N_i)$ to $C_3,$ we see that $\sqrt{3} \le e_r \le 2$ (corresponding to the height of the horoball $C_3$ being in the range from $1/3$ to $1/4$).
  So the maximum value for $e_r$ is 2.
  The minimum value for $e_q$ is 1 and this occurs when $N_{i+3}$ is a full-sized horoball.
  Plugging into the formula for $e_p$ we see that the range of possible values for $e_p$ is from 1 to $(-1 + \sqrt{13})/2 = 1.302775\dots.$
  This is a crude bound and we can improve it if we have better control over $e_q.$

  $e_q$ is the Euclidean length from $N_{i+3}$ to $B_2,$ which is equal to the Euclidean length from $C_3$ to $g(B_2).$
  The Euclidean distance $t$ between the centers of the horoballs $g(B_2)$ and $C_3$ is then $t = e_q/(e_p e_{C_3}).$
  We will analyze $t$ values (so that we can control $e_q$ values) by studying the angle $\angle(C_4,\ C_3,\ g(B_2))$; we want to eliminate $t$ values that result in an angle (at $C_3$) of more than $\pi/3$.
  $C_3$ and $C_4$ abut, and although $C_4$ and $g(B_2)$ do not necessarily abut, we can swing $C_4$ towards $g(B_2)$ until they do abut and in so doing, decrease the angle at $C_3.$
  As such, we will assume that $C_4$ and $g(B_2)$ abut.

  Thus we have a Euclidean triangle formed by the centers of $C_4,\ C_3,\ g(B_2).$
  We should denote the associated Euclidean lengths by $e_{C_4}, e_{C_3}, e_{g(B_2)},$ but in the interest of simplifying some formulas we will abuse notation and denote them respectively as $C_4,\ C_3,\ B_2.$
  Note that in this notation $B_2 = e_p.$
  Now, the edges of the triangle are $t,\ 1/(B_2 C_4),\ 1/(C_3 C_4)$; note that $1/(B_2 C_4)$ is the edge opposite the angle at $C_3$ which we will call $\gamma.$
  Using the Euclidean Law of Cosines, we have
  \begin{align*} 
    1/(B_2 C_4)^2 &= t^2 + 1/(C_3 C_4)^2 - (2t/(C_3 C_4)) \cos(\gamma) \\
    0 &=  t^2 +  (-2\cos(\gamma)/(C_3 C_4)) t + 1/(C_3 C_4)^2 - 1/(B_2
    C_4)^2
  \end{align*}
  Hence,
  $$
  t = \frac{1}{2}
  \left(2\cos(\gamma)/(C_3 C_4) \pm \sqrt{(2\cos(\gamma)/(C_3
    C_4))^2 - 4 ((1/(C_3 C_4)^2 - (1/(B_2 C_4)^2)}\right)
  $$
  $$
  t = \frac{1}{C_3 C_4 B_2}
  \left(B_2 \cos(\gamma) \pm \sqrt{B_2^2 \cos^2(\gamma) - B_2^2
    + C_3^2}\right)
  $$
  So,
  $$
  e_q = C_3 e_p t = C_3 B_2 t = \frac{1}{C_4}
  \left(B_2 \cos(\gamma) \pm \sqrt{B_2^2 \cos^2(\gamma) - B_2^2 + C_3^2}\right)
  $$
    We want to resolve the choice of $\pm$ above. Recall that we have the bounds $1 \leq C_2, C_3 \leq 2$, since the horoballs have heights at least $1/4$, $1 \leq B_2 = e_p \leq 1.302775\dots$ and $0 \leq \gamma \leq \pi/3$, by above. Since $e_q \geq 1$ by definition, it is easy to see that over these parameters the must choose $+$ in the $\pm$ above.
  
  What's the smallest allowable $e_q$ (where the angle $\gamma$ is between 0 and $\pi/3$)?
  Previously, we only needed the trivial bound $e_q \ge 1,$ but here we can do better.
  Now, first, we take the smallest $\cos(\gamma),$ i.e. $1/2,$ to obtain
  $$
  e_q = \frac{1}{2C_4} \left(B_2 + \sqrt{B_2^2 - 4B_2^2 + 4C_3^2}\right)
  = \frac{1}{2C_4}
  \left(B_2 + \sqrt{-3B_2^2 + 4C_3^2}\right)
  $$
  Next, we take the smallest $C_3,$ which is $\sqrt{3}$ to obtain
  $$
  e_q =  \frac{1}{2C_4} \left(B_2 + \sqrt{12-3B_2^2}\right)
  $$
  The minimum occurs at the largest allowable value of $B_2 = e_p,$ which is $1.302775\dots.$
  Combining this with our above upper bound on $C_4 = \sqrt{1/0.62127} = 1.2687,$ we get that $e_q \ge 1.54928.$

  We now have better control over $e_q$ and we can go back up to our $e_p$ control calculation and improve it.
  \begin{align*} 
    e_p &\le \left(-e_q + \sqrt{4e_r^2 - 3 e_q^2}\right)/2 \\
    &\le -1.5492 + \sqrt{16 - 3 (1.5492)^2})/2\cr
    &\le 0.708532
  \end{align*} 

  This is a contradiction, and we conclude that our original assumption that the angle at $C_3$ is $\pi/3$ or less is incorrect.

  This proves Claim III.

  {\it Claim IV:} $\angle (g(B_1),\ C_3,\ C_2) > \pi/3$.

  Taking our original set-up we can invert in the totally geodesic plane equidistant from $B_1$ and $B_2.$
  We then redo the analysis with $B_1$ replacing $B_2,$ and $C_2$ replacing $C_4.$
  We see that the angle at $C_3$ must be greater than $\pi/3$.

  This proves Claim IV, and thus concludes the proof of Case 1A.

  {\it Case 1B:} $B_1,B_2 \notin \eta_i,$ $BCBC$ pattern.

  We operated under the assumption that the cyclic positioning of the horoballs around $C_3$ was (clockwise) $B_1, B_2, C_4, C_2.$
  We need to rule out $B_1, C_4, B_2, C_2$ (the other case follows similarly).
  In particular, we need to rule out that $C_4$ lies in the angle $B_1, C_3, B_2.$
  We argue by contradiction and assume that $C_4$ is within that angle.
  We construct the triangle with centers at $B_1, B_2, C_4$ and see that its angle at $C_4$ is maximized when the Euclidean distance between the centers of $B_1$ and $B_2$ is 1.
  Denote the distance between the centers of $B_1$ and $C_4$ as $a$ and the distance between the centers of $B_2$ and $C_4$ as $b.$
  Then $a^2 + b^2 \ge (\sqrt{0.62127})^2 + (\sqrt{0.62127})^2 > 1.$
  Now, using the Euclidean Law of Cosines, we see that the angle at $C_4$ must be less than $\pi/2$.
  Consequently, the angle $B_1, C_3, B_2$ must be less than $\pi/2$ as well, which is a contradiction.

  This proves Case 1B, and finishes the proof of Case 1.

  {\it Case 2:} \noindent{\bf Assume $B_1$ or $B_2,$ but not both, in $\eta_i$}.

  The previous argument is still valid as long as $g(B_1)$ is not $C_2$ (similarly for $g(B_2)$ and $C_4$).
  But, if $g(B_1)$ were $C_2$ then there would be a tangency in $\cT$ between $g(B_1)$ and $C_3,$ hence $B_1$ and $N_{i+3}.$
  But the latter is not allowed by Lemma \ref{no_acc_true_tangs}, above.

  This proves Case 2.

  Therefore, $B_1, B_2 \in \eta_i.$
\end{proof}

We can now combine all the pieces to prove part of the main result of of this section.

\begin{proposition}\label{prp:min7unblocked}
  If $\eta$ is a globally minimal $\le 7$-necklace in an oriented geometric horoball system $(\mH, \mT),$ then $\eta$ is unknotted, unblocked, and unlinked.
\end{proposition}

\begin{proof}
  By Theorem \ref{thm:8unknot}, we know that $\eta$ is unknotted.
  Assume, by way of contradiction that $\eta$ is blocked by some $\gamma_b = \gamma(B_1,B_2)$ for $b \in \mT.$ 
  We can further assume that $b \in \mT''$ since $\eta$ is minimal and we can make the unknotting disk disjoint form $\mT'$ ties by Lemma \ref{lem:escape-plane}.
  Thus, the $N_i$ are at most full size. 
  It was shown in \cite[Proposition 2.8]{AdamsKnudson:2013} that all $\leq 6$-necklaces in an oriented geometric horoball system are unblocked.
  Note, however, \cite{AdamsKnudson:2013} gives an example of a non-orientable 3-manifold with a blocked 6-necklace.

  We now turn to the case where $k = 7.$

  By Lemmas \ref{lem:forced_share} and \ref{lem:has_both}, $\eta_i$ contains both $B_1$ and $B_2 = H_\infty$ for all $i = 1,\ldots, 7.$ 
  By Lemma \ref{lem:has_blocking_edge}, each $\eta_i$ contains the tie $\gamma_b.$ 
  In particular, this implies that $\eta_i$ must be one of the {\it seven} conjugates of $\eta$ that pass though $\gamma_b.$ 
  Further, $\eta_i \neq \eta_j$ for $i \neq j$ as $f_i \neq f_j.$ 
  Thus $\{\eta_i\}$ is exactly the set of seven conjugates of $\eta$ passing through $\gamma_b.$ 

  Let $\zeta_i = f_i^{-1}(\eta),$ which are again all distinct. 
  Both $\eta_i$ and $\zeta_i$ contain the tie $\gamma_b$ and are conjugates of $\eta.$ 
  Thus, $\eta_i = \zeta_j$ for some pair $(i,j).$ 
  Since $i, j = 1,\ldots,7,$ by parity we have that $\eta_k = \zeta_k$ for some $k.$ 
  Then $f_k( \eta ) = f_k^{-1} ( \eta)$ and $f_k^2( \eta ) = \eta.$ 
  This is a contradiction as $f_k^2$ cannot have finite order.

  Therefore, $\eta$ must be unblocked.
  By Proposition \ref{prop:unblocked_to_unlinked}, $\eta$ is therefore unlinked.
\end{proof}

\subsection{Full Necklace Structures}\label{ssec:fullness}

We will need an additional condition on our compressing disks. This condition is called ``fullness'' and will allow us to enumerate manifolds combinatorially in the next section. 

Throughout this section, we use the notation of \cite{GMM:2011}.

\begin{definition} Let $M$ be a compact 3-manifold with $\partial M$ a union of tori and $B\subset \partial M$ is either a torus, a disc, or one or two annuli.
  A \emph{general-based necklace-$n$ structure} $(M,B,\Delta)$ is a handle structure on $M$ of the form $B\times I$ with a single 1-handle and a single valence-n 2-handle attached to $B\times 1$, where we require that if $I_1, I_2\subset B\times 1$ are the islands, then the core of each bridge is an essential arc in $B\times 1\smallsetminus \inte(I_1 \cup I_2).$ 
\end{definition}

\begin{definition}
  A \emph{necklace-$n$ structure} is a general based necklace-$n$ structure where $B$ is a torus.
\end{definition}

\begin{definition}
  We say that $\Delta$ is \emph{full} if every lake of $\Delta$ is a disc.
  \end{definition}

  Following the convention of \cite{GMM:2011}, a \emph{compact hyperbolic 3-manifold} $(N,T),$ is a manifold $N$ whose interior supports a complete hyperbolic structure of finite volume, where $T$ is a component of $ \partial N,$ i.e.~$T$ identifies a preferred cusp.

\begin{definition}
  Let $N$ be a compact hyperbolic 3-manifold.
  An \emph{internal necklace-$n$ structure on $N$} consists of a non-elementary embedding $f:M\to N,$ where $(M,B,\Delta)$ is a necklace-$n$ structure and each component of $\partial M$ is either boundary parallel in $N$ or bounds a solid torus complementary to $M.$
\end{definition}

\begin{remarks}\label{rem:necklace}
  \begin{enumerate}[(i)]
  \item $f$ as above is $\pi_1$-surjective.
  \item  By adding 1-handles we can subdivide the 2-handle into valence-3 2-handles.
    Thus a necklace-n structure gives rise to a Mom-$(n-2)$ structure.
    On the other hand a Mom-2 structure gives rise to a necklace$\leq 5$ structure with the same underlying manifold.
  \item  Since an internal Mom-$n$ structure on a hyperbolic 3-manifold must have $n\ge 2$ it follows that an internal necklace structure must have $n\ge 4.$ 
  \end{enumerate}
\end{remarks}

The following is the necklace analogue of Lemma 1.13 \cite{GMM:2011}.

\begin{lemma}\label{lem:not_full_to_full}
  A non-elementary embedding of the necklace$n$ manifold $M$ into the compact hyperbolic 3-manifold $N$ will fail to give an internal necklacen structure on N if and only if some component of $\partial M$ maps to a convolutube.
  In that case, a reimbedding of $M,$ supported in a neighborhood of the convolutubes gives rise to an internal necklace$n$ structure on $N.$
  Components of $\partial M$ that are boundary parallel in $N$ are fixed under the reimbedding.\qed
\end{lemma}

\begin{proposition}\label{prp:7_to_full}
  Let $(N, T)$ be a 1-cusped hyperbolic 3-manifold with the internal necklace-$k$ structure $(M,T,\Delta).$ 
  If  $\Delta$ is not full, then there exists a full  necklace-$k'$ structure $(M',B',\Delta')$ on $N$ with $k'<k.$ 
  Furthermore, if $k\le 7$ or $k$ is odd and $k'=k-1,$ then this structure can be chosen so that $B'=T.$ 
\end{proposition}

\begin{remark}
  Note that the key conclusion here is fullness. 
  It is an open question if the structure can be made hyperbolic. 
  In this paper, we will be shown this to be true for $k\le 7.$ 
\end{remark}

\begin{proof}
  To start with, if  $\Delta$ is not full, then each non-disc lake is an annulus.
  Indeed, a lake with smaller Euler characteristic can only be attained if all bridges are parallel (i.e. a punctured torus or a punctured annulus). In that case, we see that $M$ contains a boundary connect sum of $B\times I$ and a lens space $\rsetminus \inte(N(B^3)),$ which is impossible
  Since $\pi_1(N)$ is torsion free it follows that the inclusion map $B\to N$ induces a $\pi_1$-surjection, which is a contradiction. 

  Through repeated application of Lemmas \ref{lem:step1} and \ref{lem:step2} below we will produce internal necklace structures of smaller and smaller necklace number, hence will eventually obtain the full necklace-$k'$ structure $(M',B',\Delta').$ 
  Repeat applications may be necessary since the result at a given step may have annular lakes. 
  If $|\partial M'|=2$ and $B'\neq T,$ then by replacing $\Delta$ by its dual handle structure we can assume that $B'=T.$ 
  If $k$ is odd and $k'=k-1,$ then we shall see that only Step 1 will apply, it will be invoked exactly once and $|\partial M'|=2.$ 
  Thus, if $k\le 7$ we can either arrange for $B'=T$ or $k'\le 5$ and $|\partial M'|=3.$

  Assume $k'\le 5$ and $|\partial M'|=3.$ We now show that either $M'$ is hyperbolic or $N$ has a full internal necklace$\le 5$ structure $(M'', B'', \Delta'')$ with $M''$ hyperbolic and $|\partial M''|=2$ in which case we are finished as before. 
  By Remark \ref{rem:necklace}, it follows that $\Delta'$ subdivides to a full Mom-$\le 3$ structure on $N.$ 
  By Theorem 4.1 \cite{GMM:2011} and its proof we can assume that either $M'$ is hyperbolic or $N$ has an internal Mom-2 structure whose underlying manifold is hyperbolic. 
  In the latter case, using Remark \ref{rem:necklace}, it follows that there exists a full internal necklace-$\le 5$ structure $(M'', B'', \Delta'')$ on $N$ such that $M''$ is hyperbolic and is the underlying manifold to a Mom-2 structure. 
  By \cite{GMM:2011} $|\partial M''|=2,$ so we can arrange that $B'=T.$

  Finally, if $M'$ is hyperbolic, then since $|\partial M'|=3,$ it follows from  \cite{GMM:2011}  that $M' = s776,$  also known as the \emph{magic manifold}. 
  Since $\Diff(s776)$ acts transitively on its ends, we can arrange that $B'=T.$ 
\end{proof}

\begin{remark} \label{annuli}
  In what follows we make repeated use of the following standard topological facts about hyperbolic 3-manifolds. 
  Suppose that $N$ is a compact hyperbolic 3-manifold and $M = N\rsetminus \inte(N(L))$ where $L$ is a link with components $L_1, \cdots, L_n.$ 
  Suppose that $A\subset M$ is a properly embedded annulus whose boundary components are essential in $\partial M.$ 
  \begin{enumerate}
  \item If $\partial A\subset T,$ where $T$ is a component of $\partial N,$ then $A$ is boundary parallel in $N.$  

  \item If $A$ connects $T$ to $\partial N(L_i),$ then $|A\cap D_i|=1,$ where $D_i$ is a meridian of $N(L_i).$  

  \item If $A$ intersects the distinct components $N(L_i)$ and $N(L_j),$ then $|A\cap D_m|=1$ for some $m\in \{i,j\}.$  

  \item If $\partial A\subset N(L_i),$ then $A$ separates $M.$
  \end{enumerate}
\end{remark}

\begin{proof}
  This is obvious possibly except for the case that $A$ hits two tubes. 
  In that case the tubes and $N(A)$ give a \SFS\ $X$ with two singular fibers of index $p,q.$ 
  If $\partial X$ is incompressible in $N,$ then this gives a $\pi_1$ contradiction. 
  If $N$ compresses, then there are three cases. 
  In each case $N$ is a non-trivial connect sum, which is a contradiction. 
  If the boundary of the compressing disc is parallel to a fiber, then $N$ is a non-trivial connect sum with two lens spaces. 
  If the boundary hits once, then it is a connect sum with one lens space.
  If it hits more than once, then $N$ is a connect sum with a \SFS\ with three singular fibers.
\end{proof}

\begin{lemma}[Step 1]\label{lem:step1}  If $\Delta$ has a single annular lake and is necklace$n,$ then N has an internal necklace$<n$ structure.  \end{lemma}

\begin{proof}
  First observe that  $\partial M$ has two or three components where the latter occurs if and only if the attaching curve of the 2-handle is separating. 
  Thus, if $k$ is odd, then $|\partial M|=2.$ 
  Let $A$ be an annulus in $B\times I$ that cuts open the lake. 
  By Remark \ref{annuli}, the resulting manifold $M_1=M\rsetminus \inte(N(A))$ is isotopic to $N$ or $N\rsetminus \inte N(K)$ depending if $|\partial M|=2$ or $3.$ 
  $M_1$ has a general based necklace structure $(M_1, B_1, \Delta_1),$ where $B_1$ is an annulus. 
  Note that $\Delta_1$ has two islands $I_1, I_2$ and $k$ bridges. 
  Also $x$ parallel bridges connect $I_1$ to $I_1,$ $x$ parallel bridges connect $I_2$ to $I_2,$ and the bridges that connect $I_1$ to $I_2$ fall into two parallel families containing $a$ and $b$ bridges respectively. 
  Assume that $a\ge b.$ 
  Note that  $\Delta$ has a single annular lake if and only if $a>0.$ 
  Also that $k=2x+a+b,$ hence $\Delta$ has a single annular lake if $k$ is odd.

  Now $M_1 =B_1\times I\cup \eta\cup \sigma,$ where $\eta$ is the 1-handle and $\sigma$ is the 2-handle. 
  Let $\sigma'$ be the core of the 2-handle  and $J=B_1\times I\cup \eta.$ 
  Let $E$ be a properly embedded disc in $B_1\times I$ such that $|E\cap \partial \sigma'|=2x+b.$ 
  Let  $\alpha\subset \inte(J)$ be a simple closed curve disjoint from $E,$ that goes through $\eta$ once and then parallels part of $B_1$ along the $a$-bridges. 
  Similar to the operation in Lemma 4.6 \cite{GMM:2011} we can hollow out $N(\alpha)$ from $M_1$ to obtain a new internal necklace-$(2x+b)$ structure $(M_2, B_2,\Delta_2)$ on $N$ where $M_2=M_1\rsetminus \inte(N(\alpha)), B_2=\partial N(\alpha), N(E)$ becomes the 1-handle and $\sigma$ continues as the 2-handle.
\end{proof}

\begin{lemma}[Step 2]\label{lem:step2} If  $\Delta$ has two annular lakes and is necklace$m,$  then $m=2n$ and $N$ has an internal necklace$\le n$-structure.\end{lemma}

\begin{proof}
  Let $M_1$ be obtained by splitting $B\times I$ along vertical annuli which intersect $B\times 1$ in the cores of the annular lakes. 
  Using Remark \ref{annuli} this gives rise to a general based internal necklace structure $(M_1, B_1, \Delta_1)$ where $B_1$ consists of two annuli and $M_1$ is obtained by starting with $ B_1\times I,$ attaching a 1-handle $\eta$ to the components of $B_1\times 1$ and then attaching a valence-$k$ 2-handle $\sigma.$ 
  With notation as in Lemma \ref{lem:step1}, we have $a=b=0$ and $2x=k.$ 
  Let $C$ and $C'$ be  the components of $B_1$ and let $E$ be a compressing disc for $C'\times I$ disjoint from $\eta.$ 
  Let $(M_2, B_2, \Delta_2)$ be the necklace-$k$ 3-manifold constructed as follows. 
  First let $M_2 = M_1.$ 
  Next observe that $B_1\times I \cup \eta$ is isotopic to $C\times I \cup \eta'$ where $\eta'$ is a 1-handle with cocore $E.$ 
  Here $\eta \cup (C'\times I\rsetminus N(E))$ have been absorbed into $C\times I.$ 
  With respect to this structure, $\sigma$ is now of valence $x=k/2.$ 
  The $\Delta_2$ necklace structure is based on the annulus $B_2.$ 
  The proof of Step 1 now produces an internal necklace-$\le x$ structure $(M_3, B_3, \Delta_3)$ on $N.$
\end{proof}


\section{Enumerating internal necklace structures}

We are almost ready to prove Theorem \ref{thm:main}.
So far, by Sections \ref{sec:param} and \ref{sec:necklaces}, we have shown that if $N$ is a complete hyperbolic 3-manifold with at least one cusp of maximal cusp volume 2.62, then $N$ has a full, internal necklace-$k$ structure with $k \le 7$---that is, a nonelementary embedding of a full handle structure based on a torus, with one 1-handle and one 2-handle of valence at most 7.
In the present section, we enumerate full necklace-$n$ structures ("Problem Bead"), and classify their possible nonelementary embeddings ("Problem NE") to show that our original cusped hyperbolic 3-manifold $N$ is gotten by Dehn filling of at least one of the 16 cusped hyperbolic 3-manifolds listed in Table \ref{table:ancestral}.
This is all put together in the proof of Theorem \ref{thm:main} in subsection \ref{ssec:pf_main}.
We start off by codifying some common notions in the literature of hyperbolic 3-manifolds.

In the theory of J\o{}rgensen and Thurston, one gets around the infinity of certain classes of hyperbolic 3-manifolds by showing these classes admit finite sets of ancestors where every element of the given class is a Dehn filling of one of these ancestral manifolds.

\begin{definition}\label{def:ancestral}
  Suppose $\cS$ is a set of 3-manifolds.
  An \emph{ancestral set for $\cS$} is a set $\cA$ of 3-manifolds such that for every $s \in \cS$, there exists $A \in \cA$ and a list $c$ of Dehn-filling coefficients for the cusps of $A$ such that $s$ is homeomorphic to $A[c]$.
\end{definition}

Every set of 3-manifolds is, trivially, an ancestral set for itself.
Usually, however, the objective is to obtain a finite ancestral set for an infinite class of 3-manifolds.
For instance,

\begin{theorem}[J\o{}rgensen-Thurston]
  For all $V > 0$ the class $VOL(V)$ of hyperbolic 3-manifolds with volume at most $V$ admits a finite ancestral set.
\end{theorem}

Another example of this is the theorem of Gabai, Meyerhoff, and Milley.

\begin{theorem}
  The set $\{m125, m129, \cdots, s959\}$ is an ancestral set for the class of 1-cusped orientable hyperbolic 3-manifolds with volume at most 2.848.
\end{theorem}

Our main result, Theorem \ref{thm:main}, likewise gives a finite ancestral set for the class of orientable hyperbolic 3-manifolds with a maximal cusp of volume at most \volumeBound.
Seeing as how we have proven that every such 3-manifold admits a nonelementary embedding from a necklace manifold of at most 7 beads, it behooves us to solve the following computational problems.


\begin{definition}\label{def:le}
  A \emph{generalized link exterior} is the interior of a compact connected orientable 3-manifold\footnote{Some authors require that the given closed 3-manifold be irreducible. We do not make this requirement.} whose boundary is a disjoint union of at least one torus.
\end{definition}

\begin{definition}\label{def:ne}
  Suppose $M$ is a compact 3-manifold.
  $NE^+(M)$ is the class of all finite-volume hyperbolic manifolds into which $M$ admits a nonelementary embedding.
\end{definition}

To make the problem of finding $NE^+(M)$ computationally accessible and algorithmic, we replace $M$ with an oriented finite triangulation\footnote{This in not an ideal triangulation, but a triangulation of the compact manifold $M$ with boundary.}.

\begin{problem}[NE]
  Given an oriented finite triangulation $\sT$ of $M$, find an algorithm to compute a finite ancestral set for $NE^+(M)$ consisting of hyperbolic 3-manifolds.
\end{problem}

\begin{problem}[Bead]
  Given a natural number $n$, to enumerate the necklace manifolds with $n$ beads.
\end{problem}

The reason for the restriction in problem NE to hyperbolic 3-manifolds is that they admit good algorithms for determining hyperbolicity of Dehn fillings.

The main results of this section are solutions to problems NE and Bead.
The solution to problem NE relies upon algorithms in normal surface theory and a reformulation of results in \cite{GMM:2011}.
The solution to problem Bead is a simple exercise in recursion.

\subsection{Problem NE}

Before beginning the solution to problem NE, we note in passing the following fact: if $A$ and $B$ are ancestral sets, respectively, for sets $U$ and $V$ of 3-manifolds, and $W \subset U \cup V$, then $A \cup B$ is an ancestral set for $W$.


\subsubsection{Nonelementary embeddings}

We first show that Definition \ref{def:ancestral} gives just the set of hyperbolic Dehn fillings of $M$, though \textit{prima facie} it could give a proper superset of these fillings.

\begin{lemma}\label{lem:hypDehn}
  If $M$ is a generalized link exterior, then $NE^+(M)$ is the class of cusped hyperbolic Dehn fillings of $M$.
\end{lemma}

One could prove this as an elementary consequence of Proposition 1.7 of \cite{GMM:2011}.
The idea there is to embed convolutubes (tori that only compress towards the generalized link exterior boundary) in disjoint balls in $N$.
But we provide a ``hands-on'' proof in the Appendix \ref{proof:hypDehn}.

\begin{definition}
  A \emph{fault} in a generalized link exterior $M$ is a properly embedded surface with nonnegative Euler characteristic that is one of the following:
  \begin{itemize}
  \item nonorientable;
  \item a sphere not bounding a ball (an \emph{essential sphere});
  \item a disc not separating a ball from $M$ (a \emph{compressing disc});
  \item an incompressible torus not parallel to $\partial M$ (an \emph{essential torus}); or
  \item an incompressible, $\partial$-incompressible annulus (an \emph{essential annulus}).
  \end{itemize}

  A generalized link exterior is \emph{irreducible} when it admits no essential sphere; is \emph{boundary-irreducible} when it also admits no compressing disc; is \emph{geometrically atoroidal} when it also admits no essential torus; and, finally, is \emph{anannular} when it also admits no essential annulus. 
\end{definition}

The following lemma, which relies to a large extent on Lemma 1.13 of \cite{GMM:2011}, enables the computation of a finite ancestral set for $NE^+(M)$ for any given a triangulation $\cT$ of $M$, ultimately by recursion on the natural decomposition of $M$ along spheres and tori.

\begin{lemma}\label{lem:ne}
  Let $M$ be a generalized link exterior.
  \begin{enumerate}
  \item If $M$ admits a nonseparating fault, then $NE^+(M) = \emptyset$.\label{it:nonsep}
  \item If $M = L \connectsum R$, then $NE^+(M) \subset NE^+(L) \cup NE^+(R)$.\label{it:reduce}
  \item If $M = L \cup_\phi R$, where $\phi: T_L \to T_R$ is a homeomorphism between two torus boundary components $T_L$ and $T_R$ of $L$ and $R$, then $NE^+(M) \subset NE^+(L) \cup NE^+(R)$.\label{it:torus}
  \item If $M$ is irreducible, boundary-irreducible, and is geometrically atoroidal, but has an annulus separating it into two solid tori, then $NE^+(M) = \emptyset$.\label{it:2d2s1}
  \end{enumerate}
\end{lemma}

We note that if $M$ is irreducible, boundary-irreducible, and geometrically atoroidal, and does not admit a nonseparating fault, but has an essential annulus, then that annulus cuts $M$ into two solid tori.
It is markedly easier to show an annulus cuts $M$ into two solid tori than to show further that that annulus is essential, so the former is the condition we use in the algorithm for computing $NE^+(M)$.

\begin{proof}\hfill
  \begin{enumerate}
  \item
    If $i: M \hookrightarrow N$ were a nonelementary embedding, and $\Sigma$ were a nonseparating fault in $M$, then $i(\Sigma) \subset N$ would be a nonseparating fault in $N$, contrary to hyperbolicity of $N$.
    
  \item 
    Suppose $M = L \connectsum R$.
    Let $S \subset M$ be a reducing sphere with the surgery of $M$ along $S$ being $L \sqcup R$, i.e.\,\,such that $M \setminus S \homeo L^\ast \sqcup R^\ast$, $L^\ast$ and $R^\ast$ being $L$ and $R$ punctured.
    Suppose further that $i: M \hookrightarrow N$ is a nonelementary embedding of $M$ into an orientable hyperbolic 3-manifold $N$.
    Then $i(S)$ is a sphere in $N$.
    Thus, since $N$ is hyperbolic, $N \setminus S = N^\ast \cup D^3$.
    Then exactly one of $L^\ast$ and $R^\ast$ has image in $D^3$.
    Say it's $L^\ast$.
    Then $\pi_1(M) = \pi_1(L^\ast) \freeproduct \pi_1(R^\ast)$ and
    \[ i_\ast(\pi_1(M))
    \isomorphic i_\ast(\pi_1(L^\ast)) \freeproduct i_\ast(\pi_1(R^\ast))
    \isomorphic i_\ast(\pi_1(R^\ast)). \]
    Hence $i|_{R^\ast}$ is a nonelementary embedding.
    In this case, capping $i(R^\ast)$ off with the $D^3$ above yields a nonelementary embedding on $R$ into $N$.
    Similarly, if $R^\ast$ has image in $D^3$, then $i$ can be modified to a nonelementary embedding of $L$ into $N$.
    Hence, $NE^+(M) \subset NE^+(L) \cup NE^+(R).$
    
  \item
    Suppose $i: M \hookrightarrow N$ is a nonelementary embedding.
    Then by item 1, the identified torus $T = T_L/\phi = T_R/\phi$ splits $M$ into two components.
    By item (iv) of Lemma 1.10 of \cite{GMM:2011}, on one of those components, $i$ restricts to a nonelementary embedding.
    
  \item 
    Suppose $M$ is irreducible, $\partial$-irreducible, and geometrically atoroidal, but admits an annulus splitting it into two solid tori.
    Then $\pi_1(M) \isomorphic \bZ \ast_{\bZ} \bZ$.
    Such a group does not admit a nonelementary map into $\PSL_2\bC$.
    Therefore $M$ does not admit a nonelementary embedding into an orientable hyperbolic 3-manifold.
  \end{enumerate}
\end{proof}

\begin{theorem}\label{thm:problemNE}
  Algorithm \ref{alg:probNE} is a solution to problem NE for nontrivial generalized link exteriors.
\end{theorem}
\newcommand{\prob}{\mbox{\textsc{prob}}}
\begin{algorithm}
  \caption{Problem NE}\label{alg:probNE}
  \begin{algorithmic}[0]
    \Procedure{probNE}{$\cT$}
    \State Assume $\cT$ is a generalized link exterior.
    \If{$\pi_1(\cT)$ already has a common axis presentation}
    \State \Return $\emptyset$
    \ElsIf{$\cT$ admits a strict angle structure}
    \State \Return $\{\cT\}$
    \EndIf
    \State Let $\cF$ be the set of fundamental normal surfaces in $\cT.$
    \If{there is a nonseparating fault in $\cF$}
    \State \Return $\emptyset$
    \ElsIf{there is a reducing sphere $S$ in $\cF$}
    \State Calculate $\cT_L, \cT_R$ such that $\cT = \cT_L \connectsum \cT_R.$
    \State \Return $\prob\NE(\cT_L) \cup \prob\NE(\cT_R)$
    \ElsIf{there is an essential torus $T$ in $\cF$}
    \State Calculate $\cT_L, \cT_R$ such that $\cT - T = \cT_L \sqcup \cT_R.$
    \State \Return $\prob\NE(\cT_L) \cup\prob\NE(\cT_R)$
    \ElsIf{there is an annulus $A$ in $\cF$ splitting $\cT$ into two solid tori}
    \State \Return $\emptyset$
    \Else
    \State \Return $\{\cT\}$
    \EndIf
    \EndProcedure
  \end{algorithmic}
\end{algorithm}

The initial two sanity checks are discussed below in $\S$\ref{sec:sanity}. 
We turn now to the bulk of the algorithm.

\begin{proof}
  Suppose $M$ is a nontrivial generalized link exterior.

  If $M$ admits a nonseparating fault, then by Lemma \ref{lem:ne}.\ref{it:nonsep}, $NE^+(M) = \emptyset$, so the algorithm returns a correct ancestral set.
  
  If $M$ admits a reducing sphere $S$, then, letting $L\connectsum R$ be the connect-sum decomposition of $M$ along $S$, Lemma \ref{lem:ne}.\ref{it:reduce} shows that $NE^+(M) \subset NE^+(L) \cup NE^+(R)$.
  By induction on the size of the prime and JSJ decompositions of $M$, we may let $U = \bigcup_{Y \in NE^+(L)} NE^+(Y)$ and $V = \bigcup_{Y \in NE^+(R)} NE^+(Y)$ be the ancestral sets of $NE^+(L)$ and $NE^+(R)$, respectively. Thus $U \cup V$ is an ancestral set for $NE^+(M)$, so the algorithm returns a correct ancestral set.

  If $M$ is a solid torus, then $\pi_1(M) \isomorphic \bZ$, and so $M$ admits no nonelementary embeddings.
  Thus the algorithm returns a correct ancestral set.

  Otherwise, $M$ is irreducible and $\partial$-irreducible.
  If in this case $M$ admits a separating essential torus $T$, then, letting $L \sqcup R \homeo M - T$ be the components of the exterior of $T$ in $M$, Lemma \ref{lem:ne}.\ref{it:torus} shows that $NE^+(M) \subset NE^+(L) \cup NE^+(R)$.
  As with the case of a reducing sphere, we recursively compute the ancestral sets $U$ and $V$ of $NE^+(L)$ and $NE^+(R)$, respectively. It follows that $U \cup V$ is an ancestral set for $NE^+(M)$, so that the algorithm returns a correct ancestral set.

  Otherwise, $M$ is irreducible, $\partial$-irreducible, geometrically atoroidal, and admits no nonseparating faults.
  Thus, since $M$ is a nontrivial generalized link exterior, either $M$ is Seifert-fibered over a disc with two exceptional fibers, or $M$ is hyperbolic.
  In the former case, $M$ admits an essential annulus $A$ separating it into two solid tori, and so by Lemma \ref{lem:ne}.\ref{it:2d2s1}, $NE^+(M) = \emptyset$, so the algorithm returns a correct ancestral set in this case, again by induction on the complexity of the decomposition of $M$.

  Otherwise, and finally, if $M$ admits no faults in $\cF,$ then by Lemma \ref{prop:Matveev} in section \ref{sssec:faults}, $M$ admits no faults at all, and therefore is hyperbolic by Thurston's Haken hyperbolization theorem.
  So by Lemma \ref{lem:hypDehn}, $\{M\}$ is an ancestral set for $NE^+(M)$.
  Thus the algorithm returns a correct ancestral set in this case.
\end{proof}

Python listings implementing this code using \Regina\ are available under the name \filename{problemNE.py} at \cite{low-cusp-volume}.

\subsubsection{Fault finding}\label{sssec:faults}

To turn the above lemmas into implementable code requires algorithms for various tasks, all to do with finding faults.
To determine the existence of a fault and, if one exists, to find one, we may appeal to the theory of normal surfaces as implemented in \Regina, which we now briefly review.

\begin{definition}
  Fix a triangulation $\cT$ of a 3-manifold.
  A \emph{normal isotopy} is an isotopy through isomorphisms of $\cT$.
  A \emph{normal disc} in a tetrahedron $\tau$ is a properly embedded disc either separating one vertex of $\tau$ from the others, or separating two pairs of vertices.\footnote{That is, regarding $\tau$ as an affine simplex, a normal disc is a properly embedded disc normally isotopic to an affine disc transverse to $\tau$.} 
  A \emph{normal surface} is a properly embedded surface transverse to $\cT$ that is a disjoint union of normal discs.
\end{definition}

The following theorem is the theoretical backbone of our use of normal surface theory.
It is a corollary of well-known results (cf. \cite{Matveev:2007}).

\begin{proposition}\label{prop:Matveev}
  Suppose $\cT$ is a triangulation whose underlying space is a 3-manifold $M$.
  There is a finite set $\mathcal{F}$ of normal surfaces in $\cT$, called \emph{fundamental normal surfaces}, that is computable from $T$, such that
  \begin{itemize}
  \item if $M$ has an embedded projective plane, then there is one in $\mathcal{F}$;
  \item else, if $M$ has an essential $S^2$, then there is one in $\mathcal{F}$;
  \item else, if $M$ has a compressing $D^2$, then there is one in $\mathcal{F}$;
  \item else, if $M$ has an embedded Klein bottle, then there is one in $\mathcal{F}$;
  \item else, if $M$ has an essential torus, then there is one in $\mathcal{F}$;
  \item else, if $M$ has an embedded M\"{o}bius band, then there is one in $\mathcal{F}$;
  \item else, if $M$ has an essential annulus, then there is one in $\mathcal{F}$.
  \end{itemize}
  In particular, $M$ admits a fault if and only if it admits a fault in $\mathcal{F}$.
\end{proposition}

\begin{proof}
  This is essentially a concatenation of Theorems 3.3.30, 4.1.12, 4.1.13, 4.1.36, 6.4.7, and 6.4.8 of \cite{Matveev:2007}.
\end{proof}

\begin{remark}
  In fact our implementation only enumerates \emph{vertex normal} surfaces.
  This smaller collection of surfaces sufficed for our needs.
\end{remark}

\begin{proposition}
  Algorithm \ref{alg:t2xi} determines whether or not a given triangulation is homeomorphic to $T^2 \times I$.
\end{proposition}

\begin{algorithm}
  \caption{Homeomorphism to $T^2 \times I$}\label{alg:t2xi}
  \begin{algorithmic}[0]
    \Procedure{is $T^2\times I$}{$\cT$}
    \If{$\cT$ has different homology from $T^2\times I$}
    \State \Return False
    \EndIf
    \State Simplify $\cT$ so that each boundary component has a one-vertex triangulation.
    \State Pick a boundary component $k$; it has three edges.
    \For{edge $e$ of $k$}
    \State Let $\cT_e$ be gotten from $\cT$ by closing-the-book along $e$.
    \If{$\cT_e$ is not a solid torus}
    \State \Return False
    \EndIf
    \EndFor
    \State \Return true
    \EndProcedure
  \end{algorithmic}
\end{algorithm}

\begin{proof}
  We give a brief outline. For more details, see \cite{Haraway:2020}.
  Closing-the-book along an edge in a one-vertex triangulation of the torus accomplishes a Dehn filling along the flip of the edge.
  Dehn filling $T^2\times I$ always yields a solid torus.
  Hence, if closing-the-book doe not yield a solid torus, then certainly $\cT$ is not $T^2 \times I$.
  Suppose, on the other hand, that each filling of $\cT$ along a flip $\alpha_e$ of an edge of the boundary triangulation yields a solid torus.
  Then $\cT$ admits three Dehn slopes distance 3 apart that fill to solid tori.
  By \cite{Gabai:1989} and \cite{Gabai:1990}, it is not difficult to show that $\cT$ must be homeomorphic to $T^2 \times I$.
\end{proof}

\begin{remark}
  The above algorithm, though very short, relies crucially upon solid torus recognition.
  The lowest known upper bounds on worst-case running time for solid torus recognition are exponential in the number of tetrahedra.
  So the same is true of the above algorithm.
\end{remark}

Algorithms \ref{alg:nonSep}, \ref{alg:essS2}, \ref{alg:essT2}, and \ref{alg:solidTorusA2} therefore accomplish the tasks required from Algorithm \ref{alg:probNE}, assuming the ambient triangulation in which $S$ lies is connected, and, in Algorithm \ref{alg:essT2}, assuming that triangulation is irreducible.
In Algorithm \ref{alg:probNE}, this assumption is satisfied at the point where Algorithm \ref{alg:essT2} is needed.

\begin{algorithm}
  \caption{Is a normal surface a non-separating fault?}\label{alg:nonSep}
  \begin{algorithmic}[0]
    \Procedure{isNonSeparatingFault}{S}
    \If{$\chi(S) < 0$}
    \State \Return \textbf{False}
    \EndIf
    \State Let $\cT$ be the triangulation in which $S$ lies.
    \State Let $\cT'$ be $\cT$ cut along $S$.
    \If{$\cT'$ is connected}
    \State \Return \textbf{True}
    \Else
    \State \Return \textbf{False}
    \EndIf
    \EndProcedure
  \end{algorithmic}
\end{algorithm}

\begin{algorithm}
  \caption{Is a normal surface an essential sphere?}\label{alg:essS2}
  \begin{algorithmic}[0]
    \Procedure{isEssentialSphere}{S}
    \If{$S$ is not a sphere}
    \State \Return \textbf{False}
    \EndIf
    \State Let $\cT$ be the triangulation in which $S$ lies.
    \State Let $\cT'$ be $\cT$ cut along $S$.
    \If{$\cT'$ is connected}
    \State \Return \textbf{True}
    \EndIf
    \State $\cT'$ has two components; call them $L$ and $R$.
    \If{$L$ or $R$ is a ball}
    \State \Return \textbf{False}
    \Else
    \State \Return \textbf{True}
    \EndIf
    \EndProcedure
  \end{algorithmic}
\end{algorithm}

\begin{algorithm}
  \caption{Is a normal surface an essential torus?}\label{alg:essT2}
  \begin{algorithmic}[0]
    \State \textbf{precondition} The ambient triangulation is irreducible.
    \Procedure{isEssentialTorus}{S}
    \If{$S$ is not a torus}
    \State \Return \textbf{False}
    \EndIf
    \State Let $\cT$ be the triangulation in which $S$ lies.
    \State Let $\cT'$ be $\cT$ cut along $S$.
    \If{$\cT'$ is connected}
    \State \Return \textbf{True}
    \EndIf
    \State $\cT'$ has two components; call them $L$ and $R$.
    \If{$L$ or $R$ is $\partial$-compressible or $T^2\times I$}
    \State \Return \textbf{False}
    \Else
    \State \Return \textbf{True}
    \EndIf
    \EndProcedure
  \end{algorithmic}
\end{algorithm}

\begin{algorithm}
  \caption{Is a normal surface an annulus separating $\cT$ into two solid tori?}\label{alg:solidTorusA2}
  \begin{algorithmic}[0]
    \Procedure{isSolidTorusAnnulus}{S}
    \If{$S$ is not an annulus}
    \State \Return \textbf{False}
    \EndIf
    \State Let $\cT$ be the triangulation in which $S$ lies.
    \State Let $\cT'$ be $\cT$ cut along $S$.
    \If{$\cT'$ is connected}
    \State \Return \textbf{False}
    \EndIf
    \State $\cT'$ has two components; call them $L$ and $R$.
    \If{$L$ and $R$ are both solid tori}
    \State \Return \textbf{True}
    \Else
    \State \Return \textbf{False}
    \EndIf
    \EndProcedure
  \end{algorithmic}
\end{algorithm}

Python listings implementing the above algorithms using \Regina\ are available under the names \filename{faultFinding.py} and \filename{problemNE.py} at \cite{low-cusp-volume}.

\subsection{Sanity checks}\label{sec:sanity}
The previous section technically completes our solution to problem NE.
However, a few additional considerations speed up our computations.
We turn to these now.


First, all necklace gluings are orientable.
Likewise, they are all connected.
If we were testing arbitrary MOM gluings, then we would add a connectedness test.
One other requirement, though, is that the links of all the vertices must be tori.
Many necklace gluings are not generalized link exteriors, and hence are of no concern to us in our enumeration.
So, we must implement a test for being a generalized link exterior.

Another fast, extremely useful check is a fundamental group check.
\Regina\  has routines for calculating and simplifying presentations of fundamental groups of triangulations (and ideal triangulations).
A significant plurality of necklace gluings of bead number at most 7 admit fundamental group presentations of one of the following forms, and these presentations admit no nonelementary embeddings into $\PSL_2\bC$.

\begin{definition}
  Suppose $P = \langle a,b\ |\ \ldots\, \rangle$ is a two-generator finite presentation of some group.
  A \emph{common axis commutator} is a relation of the form $[a^p,b^q]$ for some $p,q \in \Z$.
  A \emph{common axis equation} is a relation of the form $a^p b^q$ for some $p,q \in \Z$.
  We call these common axis \emph{relations}.
\end{definition}

\begin{lemma}
  Presentations with common axis relators do not present fundamental groups of 3-manifolds admitting nonelementary embeddings into hyperbolic 3-manifolds.
\end{lemma}

\begin{proof}
  Suppose such a presentation $P$ presents $G < \PSL_2\bC$.

  If $P$ has a common axis commutator or equation, then $a^p$ and $b^q$ commute.
  Since $G < \PSL_2\bC$, $a$ and $b$ must commute as well, since commuting elements have the same axis, and $a^p,b^q$ have the same axes as $a,b$, respectively.
  So $G$ is a quotient of $\Z\oplus\Z$.
\end{proof}

Python listings implementing these sanity checks using \Regina\ are available under the name \filename{sanity.py} at \cite{low-cusp-volume}.

Yet another sanity check one can do on a triangulation is to try to find a strict angle structure on it.

\begin{definition}
  Suppose $Q_\tau$ is the set of all normal isotopy types of quadrilaterals in a tetrahedron $\tau$, and that $Q_T = \bigcup_{\tau \in T} Q_\tau$.
  A \emph{strict angle structure} is a function $\theta: Q_T \to (0,\pi)$ such that
  \[\sum_{q \in Q_\tau} \theta(q) = \pi\]
  for all $\tau \in T$, and such that
  \[\sum_{q\ opposite\ e} \theta(q) = 2 \pi\]
  for all edges $e$ of $T$.
\end{definition}

\begin{theorem}[\cite{Lackenby:2000},\cite{FuterGueritaud:2009}]
  If an ideal triangulation of a generalized link exterior $M$ admits a strict angle structure, then $M$ is hyperbolic. 
\end{theorem}

\Regina\ already provides a routine that can usually find such structures in practice.
Note that this result implies that any oriented finite triangulation of $M$ has no faults.

The final sanity check one may perform on a triangulation is to determine if it lies in a census of known manifolds.
\Regina\ provides such censuses, and our computations in the next section rely overwhelmingly on the pre-computed properties of census manifolds.
We describe this reliance in more detail there.

\subsection{Problem Bead}\label{subsec:bead}
We now enumerate all necklace manifolds with a given number of beads, solving problem Bead.
The representations of the manifolds are the same as in \cite{GMM:2011}.
The actual enumeration is also done similarly, but with an appeal to isomorphism signatures instead of to symmetries of dipyramids.

We review the representations briefly.
A full handle structure without 3-handles deformation retracts to a ``core'' 2-skeleton of the handle structure.
This 2-skeleton is a spine for the 3-manifold.
Taking the dual of this complex gives one an ideal cellulation $C$ of the 3-manifold by \emph{dipyramids}, an $n$-dipyramid being the suspension of an $n$-gon.
The face-pairing maps preserve the distinction between the polygons' vertices and the suspension points.
We call such a gluing of dipyramids a \emph{MOM gluing}.
It is important to note that this construction required the fullness condition.

We triangulate an $n$-dipyramid with $n$ tetrahedra by identifying $n$ tetrahedra ``around an edge,'' which edge runs between the suspension points of the dipyramid.
The suspension base is not a subcomplex of this triangulation.
A MOM gluing thus yields a natural ideal triangulation of the corresponding handle structure.
We call such triangulations and face-pairings \emph{split}, as they are precisely those triangulations that admit splitting surfaces, normal surfaces with no triangles and one quadrilateral in every tetrahedron.
A full necklace manifold on $n$ beads is, among other things, a full handle structure with one 1-handle of valence $n$.
The split triangulation of such a handle structure is a split face-pairing of a single $n$-dipyramid.
For any pair of faces, there is only one face-pairing between them preserving the distinction between $n$-gon vertices and suspension points and reversing the orientation on the faces.
A full necklace manifold on $n$ beads is thus determined merely by which pairs of faces are glued.
This puts these structures in bijection with partitions of $D_n = \{i: 0\leq i < 2 n\}$ into pairs, or \emph{perfect matchings} on $D_n.$
These are our representations of these manifolds.

Lines 1309--1401 of \texttt{enum\textunderscore{}gluings.cpp} and the function \texttt{IncrementMatching} starting at line 297 of \texttt{enum\textunderscore{}utils.cpp} from \cite{GMM:2011} give an explicit, ``unrolled'' iteration over perfect matchings.
Algorithm \ref{alg:necklEnum} iterates over these perfect matchings recursively instead.
We found the latter more transparent and easier to implement, though its call stack presumably makes it less space-efficient.

\renewcommand{\Return}{\textbf{return}}

\begin{algorithm}
  \caption{Do $f(\mu)$ for all perfect matchings $\mu$ on $l \cup m$ with subpartition $p$ that match the first element of $l$ to something else in $l$.}\label{alg:necklEnum}
  \begin{algorithmic}[0]
    \Procedure{pfm}{$p,l,m,f$}
    \If{$l = nil$}
    \If{$m = \emptyset$} \State do $f(p)$
    \EndIf
    \ElsIf{$tail(l) \neq nil$} 
    \State let $x_0,\ x_1,\ l_2 = head(l),\ head(tail(l)),\ tail(tail(l))$
    \State do \textsc{pfm}$(cons((x_0,x_1),p), l_2 \cup m, \emptyset, f)$
    \State do \textsc{pfm}$(p, cons(x_0, l_2), \{x_1\} \cup m, f)$
    \EndIf
    \EndProcedure
  \end{algorithmic}
\end{algorithm}

Notice that our ultimate goal in enumerating these gluings is to get triangulations of the underlying 3-manifolds.
Therefore, it would suffice to enumerate the isomorphism signatures of all such triangulations as defined in \cite{Burton2011}.
We thank Neil Hoffman for this suggestion.
This enables a simpler and more efficient enumeration than was done in \cite{GMM:2011}.
We maintain a hash table \texttt{sigs} of the isomorphism signatures seen so far, and then for every perfect matching, construct the associated triangulation and put its isomorphism signature into \texttt{sigs} if it isn't there already.
That is, define $f(\mu)$ to be the operation of putting the isomorphism signature of $\mu$'s gluing into \texttt{sigs} if it isn't there already, and run $\mbox{\textsc{pfm}}(nil,[0,\ldots,2 n - 1],\emptyset,f)$.
This solves problem Bead.


\begin{remark}
  There are 242 cusped necklace manifolds with 4 to 7 beads.
  Among these, 35 admit a strict angle structure;
  63 admit a common axis group presentation;
  and the remaining 144 admit a nonseparating fault.
  Census checks play a small role in this calculation.
  The opposite was true for the calculation in section \ref{sec:applications}.
\end{remark}

\subsection{An ancestral set for manifolds of low cusp volume}\label{ssec:pf_main}

We conclude this section with a proof of the main technical result of the paper.

\addtocounter{table}{-6} 
\main*
\addtocounter{table}{6}

\begin{proof}

  Suppose $N$ is a complete nonelementary hyperbolic 3-manifold with at least one maximal cusp of volume at most \volumeBound.
  Since $N$ is complete and hyperbolic, we may identify $N$ with $\mathbb{H}^3 / \Gamma$ for some discrete, torsion-free, nonelementary $\Gamma < PSL_2\mathbb{C}.$
  Since $N$ has a maximal cusp of volume at most \volumeBound, there is a bicuspid triple $(P,S,L) \in \mathscr{P}$ whose associated bicuspid marking's group $G$ is conjugate into $\Gamma.$
  Without loss of generality, we may thus assume $G < \Gamma.$
  Let $m, n, g$ be the generating set for $G$ associated with $(P,S,L)$.
  Because $(P,S,L) \in \mathscr{P},$ by Theorem \ref{thm:param} there is a reduced non-identity word $w$ in $m,n,g$ with $g$-length at most 7 such that $w$ is trivial in $G,$ and hence trivial in $\Gamma.$

  Let $\mathscr{H}$ be the associated horoball system upstairs, the preimage of the chosen maximal cusp $\kappa$ of $N$ under the universal cover $\mathbb{H}^3 \to \mathbb{H}^3 / \Gamma = N$.
  Letting $\ell$ be the $g$-length of $w$, by Lemma \ref{lem:words_to_necklaces} there is an $\ell$-necklace in $\mathscr{H}$.
  Letting $\eta$ be a minimal necklace in $\mathscr{H}$, $\eta$ must be a $k$-necklace with $k \leq \ell \leq 7$.
  By Proposition \ref{prp:min7unblocked}, $\eta$ is unknotted, unblocked, and unlinked.
  Moreover, by Lemma \ref{lem:looptheorem}, we may assume $\eta$ is simple.
  Hence, there is a compressing disk for $\eta$ which descends to an embedded disk $D$ in $N$.
  In fact, in $N$ we have a handlebody $M$ based on the maximal cusp neighborhood (actually, based on $T \times [0,1]$ where $T$ is the torus boundary of the maximal cusp neighborhood), with one-handle $\sigma$ associated to ${\cal O}(1)$ and 2-handle $\tau$ gotten by thickening our embedded disk $D$.
  By construction, $M$ is non-elementarily embedded in $N$.
  Lemma \ref{lem:not_full_to_full} and Proposition \ref{prp:7_to_full} tell us that we can assume this nonelementary embedding produces an internal necklace-$n$ structure on $N$ with $n \le 7$, and that this structure is full.

  The solution of problem Bead in section \ref{subsec:bead} enumerates all full necklace structures with given bead number.
  Let $FNS_{\leq 7}$ be this set of triangulations so enumerated.
  Then $N \in NE^+(M')$ for some element $M' \in FNS_{\leq 7}$.
  The solution of problem NE in Theorem \ref{thm:problemNE} computes a finite ancestral set $A(M')$ for $NE^+(M')$ for all $M' \in FNS_{\leq 7}.$
  So by Lemma \ref{lem:hypDehn} $N$ is a Dehn filling of some element $M \in \bigcup_{M' \in FNS_{\leq 7}} A(M')$.
  This union has the 28 elements given in Table \ref{tbl:initAncestral}, which are all hyperbolic as one would expect from Theorem \ref{thm:problemNE}.
  Since they are hyperbolic, with SnapPy one may easily check Table \ref{tbl:dehnReduction}, which asserts that the associated manifolds from Table \ref{tbl:initAncestral} are Dehn fillings of the magic manifold s776.
  Thus, $N$ is a Dehn filling of some element of Table \ref{table:ancestral}.
\end{proof}

\addtocounter{table}{-5} 
\begin{table}[h]
  \begin{center}
    \begin{tabular}{|c|c|c|c|c|c|c|}\hline
      \textbf{m125} & \textbf{m129} & \textbf{m202} & \textbf{m203} & \textbf{m292} & \textbf{m295} & \textbf{m328} \\ \hline
      \textbf{m359} & \textbf{m367} & \textbf{s441} & \textbf{s443} & s596 & s647 & s774 \\ \hline
      \textbf{s776} & s780 & s782 & s785 & \textbf{v1060} & v2124 & v2355 \\ \hline
      v2533 & v2644 & v2731 & v3108 & v3127 & v3211 & v3376 \\ \hline
    \end{tabular}
    \caption{Initial ancestral set with Dehn fillings of s776 in \textbf{bold}.}\label{tbl:initAncestral}
  \end{center}
\end{table}
\addtocounter{table}{5}

\begin{table}[h]
  \begin{center}
    \begin{tabular}{|c|c|c|}\hline
      m$125: (2,1)$ & m$129: (3, -1)$ & m$202: (1,2)$ \\ \hline
      m$203: (3,-2)$ & m$292: (3,1)$ & m$295: (4,-1)$ \\ \hline
      m$328: (2,-3)$ & m$359: (4,-3)$ & m$367: (5,-2)$\\ \hline
      s$441: (4,1)$ & s$443: (5,-1)$ & v$1060: (-6, 1)$\\\hline
    \end{tabular}
    \caption{Dehn fillings of s776 in the shortest-next-shortest basis of any cusp.}\label{tbl:dehnReduction}
  \end{center}
\end{table}

The Python listing \texttt{provegordon.py} available at \cite{low-cusp-volume} performs the above calculations in the course of proving Theorem \ref{thm:gordon}.

\begin{remark}
  The paper \cite{GMM:2011} showed that every MOM-$\leq 4$ manifold that embedded nonelementarily into a hyperbolic 3-manifold was itself hyperbolic.
  Our calculations showed likewise that every necklace-$\leq 7$ manifold that embeds nonelementarily into a hyperbolic 3-manifold is itself hyperbolic.
  Both show this by cutting along surfaces.
  However, \cite{GMM:2011} performed the cutting in the text of the paper, showing that a calculation could skip over non-hyperbolic MOM structures.
  We did not skip over them, cutting along the surfaces during the calculation.
\end{remark}


\section{Applications}\label{sec:applications}
\subsection{Gordon's conjecture}

In this section we apply our previous work, together with some more computations, to prove the following conjecture of Cameron Gordon from 1995.
To state the theorem, let $e(M)$ denote the number of exceptional Dehn fillings of $M$.
\gordon*

Our proof of this theorem depends upon the following important result, known as the 6-Theorem:

\begin{theorem*}[\cite{Agol:2000}, \cite{Lackenby:2000}]\label{thm:6}
 Suppose $N$ is a hyperbolic generalized link exterior and $s$ is a slope on $\partial N$. If $\ell(s) > 6$ on the maximal cusp torus, then $N(s)$ is hyperbolic.
\end{theorem*}

We note their original conclusion was that $N$ is irreducible and has infinite word-hyperbolic fundamental group, a condition they called \emph{hyperbolike}.
However, by Perelman's affirmative resolution of Thurston's Geometrization Conjecture, this implies that $N$ is in fact hyperbolic.  

\begin{lemma}[\cite{Agol:2010}]\label{lem:agolComment}
  A 1-cusped orientable hyperbolic 3-manifold $N$ with more than 8 exceptional Dehn fillings admits a cusp of maximal area at most 36/7.
\end{lemma}

\begin{proof}
  Suppose the set $S$ of $N$'s exceptional Dehn fillings has more than 8 members.
  Then by Lemma 8.2 of \cite{Agol:2000}, the pairwise intersection numbers of elements of $S$ are not all bounded above by 6.
  So some two distinct exceptional Dehn fillings $s, s'$ must have $\Delta(s,s') \geq 7$.
  Now,
  \[\Delta(s,s') = \frac{\ell(s)\ell(s')\sin(\theta)}{A},\]
  where $\ell$ is geodesic length and $\theta$ is the angle between $s$ and $s'$ on the maximal cusp torus, and $A$ is its area.
  By the 6-Theorem, the lengths of $s$ and $s'$ are at most 6.
  Therefore, $A \leq 36/7$.
\end{proof}

\begin{proof}[Proof of Thm. \ref{thm:gordon}.]
  Just for this proof, let us call 1-cusped orientable hyperbolic 3-manifolds admitting at least 9 exceptional Dehn fillings \emph{superexceptional}.
  Suppose $M$ is a superexceptional manifold.
  By Lemma \ref{lem:agolComment} below, $M$ has a cusp of volume at most $18/7 = 2.\overline{571428}$, which is less than $2.62$.
  Thus, by Theorem \ref{thm:main}, $M$ is a Dehn filling either of s776 or of one of the 15 other manifolds detailed in Table \ref{table:ancestral}.
  The Dehn fillings of s776 were classified in detail in \cite{MartelliPetronio:2006}; in particular, their Corollary A.6 shows that the only superexceptional Dehn filling of s776 is $m004$.
  So either $M$ is m004, or $M$ is a Dehn filling of a manifold in Table \ref{table:ancestral}.
  An simple special case of the method of isolated slopes used in \cite{MartelliPetronio:2006} can be implemented as discussed in the following subsections.
  Running this method on each manifold of Table \ref{table:ancestral} only ever yields m004.
  Therefore, m004 is the only superexceptional manifold.
  That is, m004 is the only 1-cusped hyperbolic 3-manifold admitting at least 9 exceptional Dehn fillings.
\end{proof}

As we remarked in the introduction, Crawford \cite{Crawford} carried out an analysis of these Dehn filling families using an alternate approach.

\subsubsection{Determining short slopes}\label{ssec:shortSlopes}
The methods of \cite{MartelliPetronio:2006} depended critically on being able to enumerate all slopes on a maximal cusp with lengths at most a given bound $B$, which, fixing $B$, we will call \emph{short slopes}.
(In \cite{MartelliPetronio:2006} and \cite{MartelliPetronioRoukema:2014} $B = 2 \pi;$ for us, $B = 6.$)
The software from the more recent paper \cite{MartelliPetronioRoukema:2014} follows the methods of \cite{MartelliPetronio:2006}.
Instead of enumerating just the short slopes, this software in general enumerates a set containing those slopes, but that is sufficient for our needs.
Moreover, for manifolds of small complexity like those in Table \ref{table:ancestral}, this larger set includes few extraneous long slopes, if any at all.
It is also now possible to rigorously enumerate small slopes in SnapPy when imported as a module in Sage; this is the approach we take.

\subsubsection{Heuristics for determining hyperbolicity of closed 3-manifolds}\label{ssec:heuristics}
We now turn to the question of determining the hyperbolicity of closed 3-manifolds.
With Regina we can at least determine whether or not a closed 3-manifold admits any faults, using the algorithms from section \ref{sssec:faults}.
If a closed, orientable 3-manifold admits a fault, then it is not hyperbolic.
If, on the other hand, it admits no fault, then all we can conclude is that it is geometric.
If it is not hyperbolic, then the 3-manifold is small-Seifert-fibered.
There are algorithms to recognized small-Seifert-fibered manifolds, given by Tao Li in \cite{Li:2006} and Hyam Rubinstein in \cite{Rubinstein:2004}.
Unfortunately their algorithms are not yet feasible to implement in code.
Fortunately, Regina is able to recognize some triangulations of many such 3-manifolds.
In particular, it recognizes all the nonhyperbolic faultless closed triangulations we encountered as small-Seifert-fibered spaces, after some randomizations and simplifications.

Thus we turn to figuring out how one can verify the hyperbolicity of a closed 3-manifold.
In many cases SnapPy in SAGE can already do so, when the given triangulation is a geometric triangulation---that is, when the triangulation supports a solution to Thurston's gluing equations each tetrahedron of which is positively oriented.
Our methods then amounted to coaxing out various different triangulations of the same 3-manifold, until finding one that SnapPy can prove geometric.
Previous methods incorporated the taking of covers.
But we found this did not help much.
The techniques we found most useful for proving hyperbolicity were partial filling and dual curve drilling.

We first remind the reader briefly of how SnapPy works with closed 3-manifolds.
At the core of SnapPy is a modification of Newton's method, applied to Thurston's gluing equations for an ideal triangulation.
If the triangulation is not ideal in the usual sense, then the solution space for the equations might not be 0-dimensional, regardless of the conditions one places on the cusps.
SnapPy's Newton's method is not designed for this situation.
Instead, one represents a closed 3-manifold as a Dehn filling of a generalized link exterior.
With an ideal triangulation of the generalized link exterior, one gets a 0-dimensional solution space after imposing the completeness equations as usual.
But also, one gets a 0-dimensional solution space by imposing, instead of the completeness conditions, a modified condition on the cusp depending on the choice of Dehn filling coefficient.
If SnapPy's Newton's method settles on a putative solution to these modified equations, and SAGE verifies this solution, and if, of course, the tetrahedra are positively oriented, then the incompleteness in the structure amounts to a missing geodesic.
Filling in this geodesic recovers the original closed 3-manifold, and proves it admits a hyperbolic structure.

In the course of our work, we are given a generalized link exterior $M$ and a Dehn filling coefficient $c$ on its boundary, and we wish to determine whether or not $M(c)$ is hyperbolic.
One has several options at this point.
The most obvious option is to attach to Thurston's gluing equations the cusp conditions corresponding to $c$, and hope SnapPy finds a geometric solution to those equations.
If this does not work, one has the following two other options.
If $M$ is multiply-cusped, then one may pick a cusp $\tau$, take the slope $s$ of $c$ on $\tau$, construct an ideal triangulation of $M(s)$, and consider the generalized link exterior $M(s)$ with the Dehn filling coefficient $c - s$.
$M(s)(c-s)$ is homeomorphic to $M(c)$.
But SnapPy very well may be able to find a positive solution for the equations coming from $M(s)(c-s)$ and not on those coming from $M(c)$.
This was a useful technique, which decreased the number of cusps.
The opposite technique, that of drilling out curves instead of filling them in, also proved useful.
SnapPy can drill out a curve in the dual 1-skeleton of a triangulation, putting an ideal triangulation on its complement, and remembering which slope on the new boundary component is the meridian $\mu$.
(We thank Neil Hoffman for pointing out this feature to us.)
Again, SnapPy might be able to find a positive solution for the equations on $(M-\gamma)(c+\mu)$ and not on $M(c)$.
These heuristics sufficed for us to verify the hyperbolicity of the hyperbolic closed 3-manifolds we encountered.

\subsubsection{Resolving Gordon's conjecture}

To resolve Gordon's conjecture, we now only require an accounting of exceptional slopes.
Suppose $M$ is a hyperbolic 3-manifold with 2 cusps.
We wish to determine the superexceptional Dehn fillings of $M$.
To that end, recall that for us a slope $s$ is \emph{short} when it is of length at most 6 on a maximal torus boundary of its associated cusp.
We may first classify the short slopes of $M$ on each cusp into hyperbolic and exceptional.
It is routine to accomplish such a classification.
To enumerate short slopes on $\tau$, just have SnapPy determine an area form estimate on $\tau$, then use the methods of section \ref{ssec:shortSlopes}.
To determine which among them are hyperbolic and which are exceptional, use the methods of sections \ref{sssec:faults} and \ref{ssec:heuristics}.
For each cusp $\tau$, let $hyp(\tau)$ be the hyperbolic short slopes on $\tau$ and $exc(\tau)$ the exceptional slopes, and let $short(\tau)$ be all the short slopes.

Pick a cusp $\tau$. 
For each hyperbolic short slope $s$ on $\tau$, we may determine the exceptional fillings of $M(s)$ as above.
This directly decides whether or not $M(s)$ is superexceptional.
So our problem comes down to determining whether or not there are any superexceptional long fillings on $\tau$.
There are infinitely many such slopes, so a simple enumeration is out of the question.
We use the following lemma to treat these slopes.
Here, we let $hyp(s)$ and $exc(s)$ be the short hyperbolic and exceptional slopes on the unique cusp of $M(s)$.

\begin{lemma}
  Suppose $M$ is an orientable hyperbolic 2-cusped 3-manifold.
  Suppose $\tau$ is a cusp of $M$ and $s$ is a long slope on $\tau$.
  Let $\tau'$ be the other cusp of $M$ .
  Then
  \begin{equation}
    |exc(s)| \leq |exc(\tau')| + |\{s'\in hyp(\tau')\ |\ s \in exc(s')\}|.
  \end{equation}
\end{lemma}

\begin{proof}
Suppose that $s$ is a long slope on $\tau$.
Then by the 6-Theorem, $M(s)$ is hyperbolic.
Suppose $s'$ is an exceptional filling of $M(s)$.
We may interpret this as the exceptional filling $M(s,s')$ of $M$.
Since this is an exceptional filling and $s$ is long, by the 6-Theorem, $s'$ must be short on its cusp $\tau'$ in $M$.
So $s' \in short(\tau')$.
This already gives us the bound $|exc(s)| \leq |short(\tau')| = |hyp(\tau')| + |exc(\tau')|$.
However, we may do even better.
Suppose $s' \in hyp(\tau')$ and $M(s,s')$ is exceptional.
Then $s$ is an exceptional slope on $M(s')$.
So $s \in exc(s')$.
We thereby get the better bound
\[
|exc(s)| \leq |exc(\tau')| + |\{s'\in hyp(\tau')\ |\ s \in exc(s')\}|.
\]
\end{proof}

Python listings implementing the above bound are available in \texttt{gordon.py} at \cite{low-cusp-volume}.


\subsection{Results on hyperbolic volume}

The arguments in this subsection largely follow those in Section 9 of \cite{GMM:2009}.
As in \cite{GMM:2009} we will use the following theorem of Futer, Kalfagianni, and Purcell:

\begin{theorem*}[\cite{FuterKalfagianniPurcell:2008}, 1.1]
  Let $M$ be a complete, finite-volume hyperbolic manifold with cusps.
  Suppose $C_1, \ldots, C_k$ are disjoint horoball neighborhoods of some subset of the cusps.
  Let $s_1, \ldots, s_k$ be slopes on $\partial C_1, \ldots, \partial C_k,$ each with length greater than $2 \pi.$
  Denote the minimal slope length by $\ell_{\min}$.
  If $M(s_1, \ldots, s_k)$ satisfies the geometrization conjecture, then it is a hyperbolic manifold, and
  \[\vol(M(s_1, \ldots, s_k))\geq\left(1 - \left(\frac{2\pi}{\ell_{\min}}\right)^2 \right)^{3/2}\vol(M).\]
\end{theorem*}
  
\begin{corollary}\label{cor:307}
  Every one-cusped orientable hyperbolic three-manifold of volume at most 3.07
  is a Dehn filling of a manifold in Table \ref{table:ancestral}.
\end{corollary}

\begin{proof}
  Suppose $N$ is a 1-cusped orientable hyperbolic 3-manifold.
  Let $V_C(N)$ be the volume of its maximal cusp.
  As in \cite{Meyerhoff:1986}, by horoball packing results of B\"{o}r\"{o}czky,
  $V_C(N) \leq \sqrt{3} \cdot vol(N) /(2 \mathcal{V}),$
  where $\mathcal{V}$ is the volume of a regular ideal hyperbolic tetrahedron.
  If $vol(N) < 3.07 < 2.62 \cdot 2\mathcal{V} / \sqrt{3},$ then $V_C(N) < 2.62.$
  Thus by Theorem \ref{thm:main} $N$ is a Dehn filling of a manifold in Table \ref{table:ancestral}.
\end{proof}


\addtocounter{table}{-8} 
\onecusp*
\addtocounter{table}{8} 


\begin{proof}
  By Corollary \ref{cor:307} we just need to enumerate all the one-cusped fillings of elements of Table \ref{table:ancestral} with volume at most 3.07.
  Suppose $M(\textbf{s})$ is a Dehn filling of a manifold $M$ in Table \ref{table:ancestral}.
  Suppose the volume of the Dehn filling is at most 3.07.
  Then by \cite[1.1]{FuterKalfagianniPurcell:2008} above,
  \[
  \ell_{\min}(\textbf{s}) \leq 2\pi\cdot (1 - (3.07 / \vol(M))^{2/3})^{-1/2} = FKP(M).
  \]

  \newcommand{\magic}{\mbox{\texttt{s776}}}
  Let $\textbf{s}$ be a Dehn filling coefficient of \magic{}
  with $\ell_{\min}(\textbf{s}) \leq FKP(\magic).$
  Let $s_i$ be the slope of $\textbf{s}$ with minimum length.
  Then $\magic(\textbf{s})$ is also a Dehn filling of $\magic(s_i).$
  So a one-cusped Dehn filling of \magic{} with volume at most 3.07
  is also a Dehn filling of a two-cusped Dehn filling of \magic{}
  along a slope of length at most $FKP(\magic).$
  There are finitely many slopes on \magic{} of length at most $FKP(\magic).$
  The two-cusped fillings of \magic{} along such slopes
  that are not in Table \ref{table:ancestral} are as follows:

  \begin{table}[h]
    \begin{center}
      \begin{tabular}{|c|c|c|c|c|c|c|c|c|c|}\hline
        \texttt{m125} & \texttt{m129} & \texttt{m202} & \texttt{m203} & \texttt{m292} & \texttt{m295} & \texttt{m328} & \texttt{m329} & \texttt{m357} & \texttt{m359}\\\hline
        \texttt{m366} & \texttt{m367} & \texttt{m388} & \texttt{m391} & \texttt{s441} & \texttt{s443} & \texttt{s503} & \texttt{s506} & \texttt{v1060} & \texttt{v1061}\\\hline
      \end{tabular}
    \end{center}
    \caption{Two-cusped fillings of \texttt{s776} along slopes of length less than $FKP(\magic).$}
    \label{table:s776_fillings}
  \end{table}

  For each $M$ in Tables \ref{table:ancestral} and \ref{table:s776_fillings}, for each cusp of $M,$
  we have SnapPy in Sage determine a maximal horoball neighborhood $C$ of the cusp,
  then enumerate the slopes on $\partial C$ of length at most $FKP(M).$
  The resulting one-cusped manifolds with volume less than 3.07 are in Table \ref{table:onecusp307}.
\end{proof}

\volume*


\begin{proof}
As in \cite{GMM:2009} we use the following lemma of Agol, Culler, and Shalen.

  \begin{lemma}[{\cite[3.1]{AgolCullerShalen:2006}}]\label{lem:ACS}
    Suppose that $M$ is a closed, orientable hyperbolic 3-manifold,
    and that $C$ is a \emph{shortest} geodesic in $M$ such that $\mbox{tuberad}(C) \geq (\log 3)/2.$
    Set $N = \mbox{drill}_C(M).$
    Then $\vol(N) < 3.02 \cdot \vol(M).$
  \end{lemma}
  
  The proof of the lemma above relies on \cite[Prop. 10.1]{AgolStormThurstonDunfield:2007} and \cite[Cor. 4.4]{Przeworski:2006}.
  The proofs of these results only assume the given geodesic is embedded, not shortest.
  The same is true of the proof of this lemma.
  So the word ``shortest'' above can be omitted.
  Suppose then that $N$ is a closed orientable hyperbolic 3-manifold of volume at most $3.07/3.02.$
  Then either $N$ is a Dehn filling of a one-cusped hyperbolic 3-manifold of volume at most 3.07,
  or else no embedded geodesic in $N$ admits a tube of radius at least $(\log 3) / 2.$
  For instance, no embedded geodesic in the manifold Vol3 admits such a tube.
  But in fact, it is the unique orientable hyperbolic 3-manifold with this property \cite[Cor. 2]{GabaiTrnkova:2015}.
  Thus either $N$ is a Dehn filling of a one-cusped hyperbolic 3-manifold of volume at most 3.07,
  or else $N$ is Vol3.
  By Theorem \ref{thm:onecusp} either $N$ is a Dehn filling of an element of Table \ref{table:onecusp307} or $N$ is Vol3.
  In \texttt{volume\textunderscore{}bounds.py} we use SnapPy, the lemma of Futer, Kalfagianni, and Purcell, and our hyperbolicity tests to enumerate the hyperbolic Dehn fillings of elements of Table \ref{table:onecusp307} with volume at most $3.07/3.02.$
  The only such fillings are \texttt{m003(2,1)} \texttt{m004(5,1)}, and Vol3.
\end{proof}

\newpage

\appendix


\section{Group isomorphisms}

Here we list the group isomorphisms necessary for completing the proof of Theorem \ref{thm:fig8}. As these are isomorphisms between presentations, they can easily be verified by hand.

\begin{lemma}\label{isos} For each pair for relators $r_1$ and $r_2$ in Table \ref{table:relm003}, there is an isomorphism between $\Gamma(r_1, r_2) = \la m,n, g \mid mnMN, r_1, r_2 \ra$ and $\pi_1(m003)$. Similarly, we have isomorphisms for each pair in Table \ref{table:relm004} and $\pi_1(m004)$.\end{lemma}

\begin{proof} From SnapPy, we have $\pi_1(m003) = \la a,b \mid abAAbabbb \ra$. Using the \texttt{quotpic} package by Holt and Rees \cite{holt_rees_1999}, we list the isomorphisms $\phi$ in the Table \ref{table:isomm003}.
\begin{table}[h]
    \begin{center}
    \begin{tabular}{| l | l | l | l | l | l |}\hline
    Group & $\phi(a)$ & $\phi(b)$ &  $\phi^{-1}(m)$ & $\phi^{-1}(n)$ & $\phi^{-1}(g)$ \\\hline
$\Gamma(GNgMgNG,mGnGmGmnGmnG)$ & $G$ & $mGn $ & $BBABA $ & $aaBA$ & $A$ \\ \hline
$\Gamma(MnGmGMngg,MgggMgNg)$ & $gMgN$ & $gMg $ & $BABBA$ & $Ab $ & $BA$ \\ \hline
$\Gamma(mGGmgMNg,mmnGGmGmGG)$ & $G$ & $Gm $ & $Ab$ & $BAbabA $ & $A$ \\ \hline
$\Gamma(mgNgmGG,mGmGGmnGG)$ & $G$ & $Gm $ & $Ab$ & $BABBA $ & $A$ \\ \hline
$\Gamma(mnGGmngMg,NGmGGGmG)$ & $MGmGN$ & $GGm $ & $ABBAB$ & $aB $ & $ABB$ \\ \hline
$\Gamma(mnGNggNG,GGmnGmGmnG)$ & $Gnm$ & $Gm $ & $aBABab $ & $Ba$ & $aBABa$ \\ \hline
$\Gamma(nGNmggNmG,NgMgNggg)$ & $gNgM$ & $gNg $ & $Ab $ & $BABBA$ & $BA$ \\ \hline
$\Gamma(ngMgnGG,mNGmGGGmG)$ & $mGM$ & $NmG $ & $ABABB$ & $AAba$ & $BABAbab$ \\ \hline
    \end{tabular}
        \end{center}
      \caption{Isomorphisms with $\pi_1(m003) = \la a,b \mid abAAbabbb \ra$.}
    \label{table:isomm003}
 \end{table}

Similarly, $\pi_1(m004) = \la a,b \mid aaabABBAb \ra$. We list the isomorphisms $\psi$ in the Table \ref{table:isomm004}.

\begin{table}[h]
    \begin{center}
    \begin{tabular}{| l | l | l | l | l | l |}\hline
    Group & $\psi(a)$ & $\psi(b)$ &  $\psi^{-1}(m)$ & $\psi^{-1}(n)$ & $\psi^{-1}(g)$ \\\hline
$\Gamma(MNgmGMGmg,MgMGmgNgmG)$ & $gN$ & $nGnGM $ & $BAA$ & $ABabaBAb $ & $BabaBAb$ \\ \hline
$\Gamma(MnGmgMgmG,gmGMgMGmg)$ & $g$ & $mG $ & $ba$ & $baBAbABa $ & $a$ \\ \hline
$\Gamma(MnGmgMgmG,mgmGMgNgMG)$ & $Ng$ & $mGnGn $ & $baa$ & $bABabaBA $ & $bABabaB$ \\ \hline
$\Gamma(mGMnGmgMg,MGMgmGnGmg)$ & $Gn$ & $NgNgM $ & $BAA$ & $BabABAba $ & $BabABAb$ \\ \hline
$\Gamma(mnGMgmgMG,MgmGGmgMG)$ & $G$ & $mg $ & $ba$ & $AbaBabAB $ & $A$ \\ \hline
    \end{tabular}
        \end{center}
      \caption{Isomorphisms with $\pi_1(m004) =  \la a,b \mid aaabABBAb \ra$.}
    \label{table:isomm004}
 \end{table}

\end{proof}


\section{Technical necklace theory}

The following lemma is used in the proof of Lemma \ref{lem:escape-plane}. As the details are a bit technical, we chose to leave it out of the main exposition and place it in the appendix.

\begin{lemma}\label{lem:vis_angle}
  Let $B_0$ be a full-sized horoball centered at $(0,0)$ and let $B_1, B_2$ be a tangent pair of at most full-sized horoballs. 
  Let $\mathcal{D}$ be the disk though $b_1 = \partial_\infty B_1$ and $b_2 = \partial_\infty B_2$ that makes an angle of $\pi/3$ with the oriented ray from $b_2$ though $b_1$. 
  If all the horoballs are disjoint, then the visual angle of $\mathcal{D}$ from $(0,0)$ is at most $\pi/3$.
\end{lemma}

\begin{figure}[h]
  \begin{center}
    \begin{overpic}[scale=1.5]{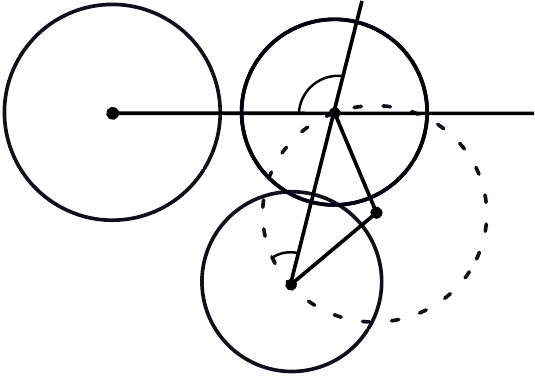}
      \put(12,48){$B_0$}
      \put(51,57){$\pi - \theta$}
      \put(69,55){$B_1$}
      \put(47,12){$B_2$}
      \put(51.3,26){${\frac{\pi}{3}}$}
      \put(75,22){$D$}
    \end{overpic}
  \end{center}
  \caption{.}
  \label{fig:escape_top_lemma}
\end{figure}

\begin{proof}
  Let $A$, $B$, and $C$ be, respectively, the centers of $B_0$, $B_1$, and $B_2$.
  Let $h_A$, $h_B$, and $h_C$ likewise be their heights.
  Then $h_A = 1$ since $A$ is full-sized, and $h_B, h_C \leq 1$ since $B, C$ are at most full-sized.
  Likewise, by Lemma \ref{lem:hd}, $AB^2 \geq h_A   h_B = h_B$ and $AC^2 \geq  h_A   h_C = h_C$ since the interiors of $B_0$, $B_1$, and $B_2$ are all disjoint.
  Moreover, $BC^2 = h_B   h_C$ since $B_1$ and $B_2$ are tangent.
  From these inequalities, we can conclude $BC \leq AB\cdot AC$ and $BC \leq 1$.
  Conversely, if $BC \leq AB \cdot AC$ and $BC \leq 1$, then there is some choice of heights $h_B$ and $h_C$ yielding a collection of horoballs as given above.
  So the conditions $BC \leq AB \cdot AC$ and $BC \leq 1$ together are equivalent to  the antecedent of the lemma.

  On the other hand, we can also rewrite the consequent.
  Let $r$ be the radius of $\mathcal{D}$, $D$ its center, and $d = AD$.
  Then the visual angle of $\mathcal{D}$ at $A$ is at most $\gamma$ if and only if $\sin(\gamma/2) \leq r/d$.
  So $2   r \leq d$ is equivalent to the consequent.
  Thus the lemma is equivalent to the following:
  if $BC \leq AB \cdot AC$ and $BC \leq 1$, then $2 r \leq d$.

  We now want to rewrite this in terms of $r$ and $d$ and one other parameter.
  First, by construction of $\mathcal{D}$, $BC^2 = 3 r^2$.
  So $BC \leq 1$ is equivalent to $3 r^2 \leq 1$.
  
  Next, consider all the possible configurations with a given pair of $d$ and $r$.
  These are parametrized by the angle $\angle ADB$.
  However, we find it more convenient to use instead the parameter $\phi = \angle ADB - \pi/3$, which measures the deviation from the symmetric position where $AB = AC$ and $A$ and $D$ lie on opposite sides of the supporting line of $\overline{BC}$.
  By the law of cosines,
  \begin{align*}
    AB^2
    &= d^2 + r^2 - 2 d r \cos(\pi/3 + \phi)\\
    &= x - y (\cos(\phi)/2 - \sqrt{3} \sin(\phi)/2)\mbox{ and}\\
    AC^2
    &= d^2 + r^2 - 2 d r \cos(\pi/3 - \phi)\\
    &= x - y (\cos(\phi)/2 + \sqrt{3}\sin(\phi)/2),
  \end{align*}
  letting $x = d^2 + r^2$ and $y = 2 d r.$
  Thus
  \[
  AB^2 \cdot AC^2 = y^2 \cos^2(\phi) - x y \cos(\phi) + x^2 - 3 y^2 / 4.
  \]
  Hence the other part of the antecedent we may rewrite as
  \begin{equation}\label{ineq:tech_lem_A}
    3 r^2 \leq 4 d^2 r^2 \cos^2(\phi) - 2 \cos(\phi) (d r^3 + d^3 r) + d^4 - d^2 r^2 + r^4.
    \end{equation}
  Finally for this setup, we note first that, by assumption, $\phi \in [-\pi/3, \pi/3]$, for otherwise $A$ and $D$ would lie on the same side of the supporting line of $\overline{BC}$, contrary to how we drew $\mathcal{D}$;
  and second, that, of course, we assume $d > 0$ and $r > 0$.
  Thus the lemma is equivalent to the following statement:
  \begin{quotation}
    Suppose $d,r > 0$ and $\phi \in [-\pi/3,\pi/3]$.
    Suppose $3 r^2 \leq 1$.
    Finally, suppose \ref{ineq:tech_lem_A} holds.
    Then $d \geq 2 r$.
  \end{quotation}

  To show this, let us first bound the extreme values of the right-hand side of \ref{ineq:tech_lem_A} for $\phi \in [-\pi/3, \pi/3]$.
  To that end, we enumerate the critical points in $[-\pi/3, \pi/3]$.
  Using $x$ and $y$ as above,
  \[
  \frac{d}{d\phi}\left(AB^2 \cdot AC^2\right)
  =
  y \sin(\phi) (x - 2 y \cos(\phi))
  \]
  First, $y = 2 d r = 0$ is impossible.
  Next, $\sin(\phi) = 0$ when $\phi = 0$, i.e.~when $\cos(\phi) = 1$, since $\phi \in [-\pi/3, \pi/3]$.
  The last critical point condition yields $\cos(\phi) = x/(2  y)$.
  Finally, the edge cases $\phi = \pm \pi/3$ both yield $\cos(\phi) = 1/2$.
  Thus the extreme values of $AB^2 \cdot AC^2$ for $\phi \in [-\pi/3, \pi/3]$ are in the set
  \begin{align*}
    S
    &= \left\{ (y - 2 x)^2/4,\quad (x - y) (y + 2 x)/2,\quad 3 (x - y) (x + y)/4\right\}\\
    &= \left\{ (d^2 - d r + r^2)^2,\quad (d-r)^2 (d^2 + d r + r^2),\quad 3 (d - r)^2 (d + r)^2 / 4\right\}.
  \end{align*}
  Therefore, under our assumptions, $3 r^2 \leq s$ for some $s \in S$.
  These are all quartic forms.
  Dividing them all by $r^4$ and setting $z = d/r$ yields the set
  \[ T = \{ (z^2 - z + 1)^2,\quad (z-1)^2 (z^2 + z + 1),\quad 3 (z-1)^2 (z+1)^2 / 4 \}. \]
  Assuming $0 < r \leq \sqrt{1/3}$, we have $3 r^2/r^4 = 3 / r^2 \geq 9$.
  So it will suffice to show that if any element of $T$ is at least 9, then $z \geq 2$.
  We finish by case analysis.
  
  Suppose $(z^2 - z + 1)^2 \geq 9$.
  Since $z^2 - z + 1 \leq -3$ is impossible, we have $z^2 - z + 1 \geq 3$.
  Now, $z^2 - z - 2 = (z - 2) (z + 1)$.
  Since $z > 0$, $z^2 - z + 1 \geq 3$ implies $z \geq 2$.

  Suppose $(z - 1)^2 (z^2 + z - 1) \geq 9$.
  Let $f(z) = (z - 1)^2 (z^2 + z - 1)$.
  Then $f'(z) = (z - 1) (4 z^2 + z + 1)$.
  Now, $4 z^2 + z + 1$ is positive for all $z \in \mathbb{R}$.
  So $f$ is decreasing on $(-\infty,1]$ and increasing on $[1,\infty)$.
  Calculating, $f(0) = 1$, $f(1) = 0$, and $f(2) = 7$.
  Hence if $z \geq 0$ and $f(z) \geq 9$, then $z > 2$.

  Finally, suppose $3 (z - 1)^2 (z + 1)^2 / 4 \geq 9$.
  Then $(z^2 - 1)^2 \geq 12$.
  Since $z^2 - 1 \leq -\sqrt{12}$ is impossible, $z^2 - 1 \geq \sqrt{12}$.
  Since $z > 0$, $z \geq \sqrt{1 + \sqrt{12}} > 2$.

\end{proof}


\section{Convolotubes}\label{proof:hypDehn}

The following is a ``hands-on'' proof of Lemma \ref{lem:hypDehn}. Since this Lemma is a consequence of Proposition 1.7 of \cite{GMM:2011}, we only include the proof here for completeness. Recall that we want to show that if $M$ is a generalized link exterior, then $NE^+(M)$ is the class of cusped hyperbolic Dehn fillings of $M$.

\begin{proof}[Proof of Lemma \ref{lem:hypDehn}]
  Abusing notation, we realize an embedding in $NE^+(M)$ as $M \subset N$ where $N$ is an orientable hyperbolic finite-volume $N$.
  Suppose $T$ is a torus boundary component of $M$.
  It therefore is either $\partial$-parallel in $N$ or compressible in $N,$ since $N$ is atoroidal.
  Suppose $T$ is compressible.
  Let $\Delta$ be a compressing disk for $T$ in $N$.
  Up to isotopy, $\Delta \cap \partial M$ is a multicurve on $\Delta$ with $\partial \Delta$ as a component.
  There may or may not be other components.

  \begin{itemize}
  \item
    Suppose $\Delta \cap \partial M = \partial \Delta$.
    Then either $\Delta \subset M$, or $\Delta$ lies in a component of $N \setminus int(M)$. 
    
    \begin{itemize}
    \item
      Suppose $\Delta \subset M$.
      Then $\Delta$ is a compressing disc for $T$ in $M$.
      Let $U$ be a regular neighborhood of $\Delta \cup T$ in $M$, so that $\partial U = T \sqcup \Sigma$ for some sphere $\Sigma \subset M$.
      Thus $M$ is a (possibly trivial) connect-sum with a solid torus, along $\Sigma$.
      Now, since $N$ is irreducible, $\Sigma$ bounds a ball $B$ in $N$.
      Either $T \subset B$, or not.

      \begin{itemize}
      \item
        If $T \subset B$, then $U \subset B$.
        Hence $M = M' \cup_\Sigma U$, where $M' = M \setminus int(B)$.
        Let $i: M \hookrightarrow N$ be the inclusion of $M$ into $N$.
        We may modify $i$ on $U$ to a different embedding so that $i(U)$ is the complement of an unknot in $B$.
        Thus after so modifying $i$, $T$ bounds a solid torus in $N$ away from $M$, and no other boundary tori of $M$ are modified, thus reducing the number of convolutubes of $i(\partial M)$.
        
      \item
        Otherwise, $U \cup B$ is a solid torus $\cT$ in $N$, and $M \subset \cT$.
        Thus the embedding of $M$ in $N$ is elementary.
      \end{itemize}
      
    \item
      Suppose $\Delta$ lies in a component of $N \setminus M$.
      Now let $U$ be a regular neighborhood of $\Delta \cup T$ in $N \setminus int(M)$, so that $\partial U = T \sqcup \Sigma$ for some sphere $\Sigma \subset N \setminus int(M)$.
      Since $N$ is irreducible, $\Sigma$ bounds a ball $B$ in $N$.
      Either $T \subset B$ or not.
      
      \begin{itemize}
      \item
        Suppose $T \subset B$.
        Then $M \subset B$, for $M$ is connected and lies to that side of $T$ away from $\Sigma$.
        Hence the embedding of $M$ in $N$ is elementary.
        
      \item
        Otherwise, $T$ bounds the solid torus $U\cup B$ in $N$ to the side of $T$ away from $M$.
      \end{itemize}
      
    \end{itemize}

    Therefore, if $\Delta \cap \partial M = \partial \Delta$, then either the embedding is elementary, or it admits a (possibly trivial) redefinition such that $T$ bounds a solid torus in $N$ away from $M$.
    
  \item 
    Suppose instead that $int(\Delta) \cap \partial M \neq \emptyset$.
    Then it has an innermost component $\gamma$.
    It bounds a unique disc $\Delta'$ in $\Delta$ such that $\Delta' \cap \partial M = \gamma$.
    Then $\partial \Delta'$ lies on some other boundary component $T'$ of $M$.
    On this component $T'$, $\partial \Delta'$ is either essential or not.
    
    \begin{itemize}
    \item
      If $\partial \Delta'$ is inessential on $T'$, then it bounds a disc $\delta$ on $T'$.
      Then $\partial \Delta \cup_\gamma \delta$ is a sphere $\Sigma$ in $N$, and thus bounds a ball $B$ in $N$.
      Hence we may isotope $\Delta$, fixing $\partial \Delta$, through $B$, and then slightly pushing off $\delta$, to a new disc $D$ such that $int(D)\cap\partial M=(int(\Delta)\cap\partial(M))\setminus\gamma$, thus reducing the number of components in $int(\Delta)\cap\partial(M)$, and leaving $M$ invariant.
      
    \item
      If $\partial \Delta'$ is essential on $T'$, then $\partial \Delta'$ is a compressing disc for $T'$ with $\Delta' \cap \partial M = \partial \Delta'$.
    \end{itemize}

    Thus by induction on $int(\Delta) \cap \partial M$, there is some compressing disc $\Delta$ for some boundary component of $M$ such that $\Delta \cap \partial M = \partial \Delta$, reducing to the previous case.
  \end{itemize}

  Therefore, if $M \subset N$ is nonelementary, then there is some nonelementary embedding $i: M \hookrightarrow N$ such that $N \setminus i(M)$ consists of boundary collars and solid tori---that is, such that $i$ manifests $N$ as a Dehn filling of $M$.

  Conversely, if $N$ is a Dehn filling of $M$, then the Dehn filling embedding $i: M \hookrightarrow N$ has $i_\ast(\pi_1(M)) = i_\ast(\pi_1(N))$.
  In particular, if $N$ is hyperbolic, then $i_\ast(\pi_1(N))$ is nonabelian, so the embedding is nonelementary.

  Thus the class of hyperbolic 3-manifolds $N$ admitting nonelementary embeddings from $M$ is the same as the class of hyperbolic Dehn fillings of $M$.
\end{proof}


\section{Computational aspects}\label{sec:verify}

In this appendix, we will go over the technical details needed to validate Propositions \ref{thm:gLength7} and \ref{prop:identify}. We organize the discussion around the theoretical aspects and the arithmetic of complex affine 1-jets, followed by a description of the files, code, and data provided.

As described in Section \ref{sec:param}, the proofs are provided as binary trees where each terminal node corresponds a (sub)box of the parameter space $\sB$ with a label that encodes an inequality which holds for all points in that box. This inequality is rigorously verified by the programs \texttt{verify} and \texttt{identify} to hold over the box. Each inequality is encoded either as a boundary condition or a word. In the following sections, we prove the validity of how each inequality is constructed and the validity of the arithmetic used by \texttt{verify}.

Recall that the main parameter space objects are the parent box $\sB$, the compact subset $\sP$ that is guaranteed to contain a bicuspid triple for every cusped hyperbolic 3-manifolds with cusp volume $\leq 2.62$, and $\gdspace$, the set of all geometric bicuspid triples in $\sP$.

We have four types of terminal conditions at a (sub)box $\sB_\fb$ of $\sB$. These are:

\begin{itemize}
\item Boundary conditions --- these prove that  $\sB_\fb$ lies outside of $\sP$ by showing that $\sB_\fb$ that the inverse of one of the inequalities in Proposition \ref{prop:compact} holds over the entire box.
\item Killed words ---  a killer word $w$ associated to $\sB_\fb$ is used to prove that $\sB_\fb \subset \cK_w$, which by Corollary \ref{cor:killer} means that $\sB_\fb \cap \gdspace = \emptyset$. This is equivalent to saying that the bicuspid triples in $\sB_\fb$ are either indiscrete or are incorrectly marked, meaning that a correct marking is found elsewhere in $\sP \sB$.
\item Necklace words --- a necklace word $w$ associated to $\sB_\fb$ is used to prove $\sB_\fb \subset \cU_w$, which by Lemma \ref{lem:relator} means that any triple in $\sB_\fb \cap \gdspace$ is a relator of $g$-length at most that of $w$. In \texttt{verify}, we allow this $g$-length to be at most $7$, while in \texttt{identify} we limit this to $3$.
\item Variety words --- a variety word $w$ associated to $\sB_\fb$ is used to prove $\sB_\fb \subset \cV_w$, which by Lemma \ref{lem:var_nbd} means that $w$ is a relator for any bicuspid group associated to a triple in $\sB_\fb \cap \gdspace$. This condition is only used in \texttt{identify}.
\end{itemize}

The inequalities in the definitions of $\cK_w$, $\cU_w$, and $\cV_w$ are straightforward to check after constructing the matrix corresponding to $w$ at points of the box $\sB_\fb$. To prove that these conditions hold over the entire box, we use 1-jets with error and round-off error for computations as is done in \cite{GMT:2003}. If one wanted to use interval arithmetic instead, we expect that further subdivision would be necessary. 

\subsection{1-jets and error control}\label{ssec:affine_jets}

Our computational tool for verifying the above mentioned inequalities are 1-jets. Let \[\sA = \{ (z_0, z_1, z_2) \in \bC^3 : |z_i| \leq 1 \}\] and consider a (holomorphic) function $g: \sA \to \bC.$ A $1$-jet approximation with error $\eps$ of $g$ is a linear map $(z_0, z_1, z_2) \mapsto c + a_0 z_0 + a_1 z_1 + a_0 z_2$ such that \[|g(z_0, z_1, z_2) - c + a_0 z_0 + a_1 z_1 + a_0 z_2) | \leq \eps.\]
Considering all maps approximated by a given 1-jet, one defines the jet-set \[S(c,a_0,a_1,a_2; \eps) = \{ g : \sA \to \bC :  |g(z_0, z_1, z_2) - c + a_0 z_0 + a_1 z_1 + a_0 z_2) | \leq \eps \}.\]
In \cite{GMT:2003}, the authors derive the arithmetic for jets. For example, given jet-sets $S_1$ and $S_2$, they show to to compute the parameters for a jet-set $S_3$ where $f \cdot g \in S_3$ for all $f \in S_1$ and $g \in S_2$. They do this for all basic arithmetic operations $\pm, \times, /,$ and  $\sqrt{\,}$. In code, we call these \texttt{ACJ}s for ``affine complex jets.''

For our computations, we use this exact same arithmetic. Note, to make this arithmetic fast, the parameters of $S(c,a_0,a_1,a_2; \eps)$ are all given as floating-point doubles. Recall that floating-point numbers are a finite subset of the reals represented via a list of bits on a computer. As such, any arithmetic operation between two floating-point numbers must make a choice about which floating-point result to give, as the real value of this result may not be representable as floating-point number. On a machine conforming to the IEEE-754 standard \cite{ieee}, all operations performed with floating points numbers are guaranteed to round in a consistent way to a closest floating-point representative, as long as overflow and underflow have not occurred. The code supplied as part of  \texttt{verify} and \texttt{identify} checks that underflow and overflow does not happen during validation.

Just as in Section 7 of \cite{GMT:2003}, each boxcode corresponds to a (sub)box of our parameter space that has a floating-point center $(c_0, c_1, c_2, c_3, c_4, c_5)$ and a floating-point size $(s_0, s_1, s_2, s_3, s_4, s_5)$. 
Using the IEEE-754 standard, the floating-point size is large enough so that the floating-point version of the box,
i.e. $\{ (x_0, x_1, x_2, x_3, x_4, x_5) \in \bR^6 \sep |x_i - c_i| \leq s_i\}$,
  contains the true box as a proper subset.
Recall that over a box we associate the coordinate functions $L = x_3 + \i x_0$, $S = x_4 + \i x_1$, and $P = x_5 + \i x_2$ for our parameters.
We can replace these coordinate functions with linear maps $g_L, g_L, g_P: \sA \to \bC$ given by:
\[\hspace{-2ex}g_L(z_0, z_1,z_2) = c_3 + \i c_0 + (s_3 + \i s_0)z_0, \quad g_S(z_0, z_1,z_2) = c_4 + \i c_1 + (s_4 + \i s_1)z_1, \quad g_L(z_0, z_1,z_2) = c_5 + \i c_2 + (s_5 + \i s_2)z_2.\]
Notice that, by construction, the linear maps see all the $L, S, P$ values over the given box.
We can now think of $g_L, g_S, g_P$ as living in jet-sets.
For example, $g_P \in S(c_5 + \i c_2, 0, 0, s_5 + \i s_2; 0)$ and similarly for others.
It follows that any evaluations with jet-arithmetic will include the true values over the box up to the error that accumulates.
See \cite{GMT:2003}, Sections 7 and 8 for concrete details.

\subsection{Elimination conditions}

Recall that the computational proofs of our results are encoded by binary trees where each terminal leaf contains a string that encodes the necessary list of inequalities that we must hold true over the box. The code in \texttt{verify.c} and \texttt{identify.c} is responsible for reading the tree starting from the root node, building the box \texttt{ACJ} parameters for each boxcode, and checking each condition. The evaluations of each condition occurs in \texttt{elimination.c}. The code in these files is commented and we believe the reader will be best served by following the comments there as well as the \texttt{README.md}.

We want to stress a few main points. First, the jet arithmetic code is entirely borrowed from \cite{GMT:2003}, aside from a tiny patch of a clear typo in the published code, see the patch file in the scripts folder of \cite{verify-code}. In addition, we make use of the floating-point rounding estimates from \cite{GMT:2003} in a few places in our code. In particular, we apply this error control to evaluate boundary conditions such as cusp area and bounds on real and imaginary parts of parameters. Since most of these quantities are monotonic in the parameters, we simply construct validated over and under estimates at the vertices of our box. These are the \texttt{nearer}, \texttt{further}, and \texttt{greater} points, which are nearer than all box points to $\vec{0}$, further than all box points for $\vec{0}$, and greater than all box points, respectively. This simple constructions allows us to easily check the boundary conditions  of Proposition \ref{prop:compact}. For a detailed explanation of how these points are chosen, see the comments in \texttt{box.c}.

\newpage

\bibliographystyle{alpha}	
\bibliography{cuspVol}

\newcommand{\etalchar}[1]{$^{#1}$}
\begin{thebibliography}{BBP{\etalchar{+}}21}

\bibitem[ACS06]{AgolCullerShalen:2006}
Ian Agol, Marc Culler, and Peter~B. Shalen.
\newblock Dehn surgery, homology and hyperbolic volume.
\newblock {\em Algebr. Geom. Topol.}, 6:2297--2312, 2006.

\bibitem[Ada87]{Adams:1987}
Colin~C. Adams.
\newblock The noncompact hyperbolic {$3$}-manifold of minimal volume.
\newblock {\em Proc. Amer. Math. Soc.}, 100(4):601--606, 1987.

\bibitem[Ago00]{Agol:2000}
Ian Agol.
\newblock Bounds on exceptional {D}ehn filling.
\newblock {\em Geom. Topol.}, 4:431--449, 2000.

\bibitem[Ago02]{Agol:2002}
Ian Agol.
\newblock Volume change under drilling.
\newblock {\em Geom. Topol.}, 6:905--916, 2002.

\bibitem[Ago10]{Agol:2010}
Ian Agol.
\newblock Bounds on exceptional {D}ehn filling {II}.
\newblock {\em Geom. Topol.}, 14(4):1921--1940, 2010.

\bibitem[AK13]{AdamsKnudson:2013}
Colin Adams and Karin Knudson.
\newblock Unknotting tunnels, bracelets and the elder sibling property for
  hyperbolic 3-manifolds.
\newblock {\em J. Aust. Math. Soc.}, 95(1):1--19, 2013.

\bibitem[AST07]{AgolStormThurstonDunfield:2007}
Ian Agol, Peter~A. Storm, and William~P. Thurston.
\newblock Lower bounds on volumes of hyperbolic {H}aken 3-manifolds.
\newblock {\em J. Amer. Math. Soc.}, 20(4):1053--1077, 2007.
\newblock With an appendix by Nathan Dunfield.

\bibitem[BBP{\etalchar{+}}21]{Regina}
Benjamin~A. Burton, Ryan Budney, William Pettersson, et~al.
\newblock Regina: Software for low-dimensional topology.
\newblock {\tt http://\allowbreak regina-normal.\allowbreak github.\allowbreak
  io/}, 1999--2021.

\bibitem[Ber]{heegaard}
John Berge.
\newblock \texttt{heegaard3}.
\newblock Available at \url{https://github.com/3-manifolds/heegaard3},
  (2021-06-17).

\bibitem[Boy02]{Boyer:2002}
Steven Boyer.
\newblock Dehn surgery on knots.
\newblock In {\em Handbook of geometric topology}, pages 165--218.
  North-Holland, Amsterdam, 2002.

\bibitem[Bur11]{Burton2011}
Benjamin~A. Burton.
\newblock The {P}achner graph and the simplification of 3-sphere
  triangulations.
\newblock In {\em Computational geometry ({SCG}'11)}, pages 153--162. ACM, New
  York, 2011.

\bibitem[Cal01]{Calegari:2001}
Danny Calegari.
\newblock Napoleon in isolation.
\newblock {\em Proc. Amer. Math. Soc.}, 129(10):3109--3119, 2001.

\bibitem[CFJR01]{meyername}
Ted Chinburg, Eduardo Friedman, Kerry~N. Jones, and Alan~W. Reid.
\newblock The arithmetic hyperbolic $3$-manifold of smallest volume.
\newblock {\em Annali della Scuola Normale Superiore di Pisa - Classe di
  Scienze}, Ser. 4, 30(1):1--40, 2001.

\bibitem[CM01]{CaoMeyerhoff:2001}
Chun Cao and G.~Robert Meyerhoff.
\newblock The orientable cusped hyperbolic {$3$}-manifolds of minimum volume.
\newblock {\em Invent. Math.}, 146(3):451--478, 2001.

\bibitem[Cra18]{Crawford}
Thomas Crawford.
\newblock {\em A Stronger Gordon Conjecture and an Analysis of Free Bicuspid
  Manifolds with Small Cusps}.
\newblock PhD thesis, Boston~College, 2018.

\bibitem[FG09]{FuterGueritaud:2009}
David Futer and Fran\c{c}ois Gu\'{e}ritaud.
\newblock Angled decompositions of arborescent link complements.
\newblock {\em Proc. Lond. Math. Soc. (3)}, 98(2):325--364, 2009.

\bibitem[FKP08]{FuterKalfagianniPurcell:2008}
David Futer, Efstratia Kalfagianni, and Jessica~S. Purcell.
\newblock Dehn filling, volume, and the {J}ones polynomial.
\newblock {\em J. Differential Geom.}, 78(3):429--464, 2008.

\bibitem[FPS19]{FuterPurcellSchleimer:2019}
David Futer, Jessica~S. Purcell, and Saul Schleimer.
\newblock Effective distance between nested {M}argulis tubes.
\newblock {\em Trans. Amer. Math. Soc.}, 372(6):4211--4237, 2019.

\bibitem[Gab89]{Gabai:1989}
David Gabai.
\newblock Surgery on knots in solid tori.
\newblock {\em Topology}, 28(1):1--6, 1989.

\bibitem[Gab90]{Gabai:1990}
David Gabai.
\newblock {$1$}-bridge braids in solid tori.
\newblock {\em Topology Appl.}, 37(3):221--235, 1990.

\bibitem[Gab01]{Gabai:2001}
David Gabai.
\newblock The {S}male conjecture for hyperbolic 3-manifolds: {${\rm
  Isom}(M^3)\simeq{\rm Diff}(M^3)$}.
\newblock {\em J. Differential Geom.}, 58(1):113--149, 2001.

\bibitem[Gita]{low-cusp-volume}
GitHub.
\newblock low-cusp-volume.
\newblock Available at \url{https://github.com/bobbycyiii/low-cusp-volume}
  (2021-06-17).

\bibitem[Gitb]{verify-code}
GitHub.
\newblock verify-cusp.
\newblock Available at \url{https://github.com/andrew-yarmola/verify-cusp}
  (2021-06-17).

\bibitem[GMM09]{GMM:2009}
David Gabai, Robert Meyerhoff, and Peter Milley.
\newblock Minimum volume cusped hyperbolic three-manifolds.
\newblock {\em J. Amer. Math. Soc.}, 22(4):1157--1215, 2009.

\bibitem[GMM11]{GMM:2011}
David Gabai, Robert Meyerhoff, and Peter Milley.
\newblock Mom technology and volumes of hyperbolic 3-manifolds.
\newblock {\em Comment. Math. Helv.}, 86(1):145--188, 2011.

\bibitem[GMT03]{GMT:2003}
David Gabai, G.~Robert Meyerhoff, and Nathaniel Thurston.
\newblock Homotopy hyperbolic 3-manifolds are hyperbolic.
\newblock {\em Ann. of Math. (2)}, 157(2):335--431, 2003.

\bibitem[Gor98a]{Gordon:1998}
C.~McA. Gordon.
\newblock Boundary slopes of punctured tori in {$3$}-manifolds.
\newblock {\em Trans. Amer. Math. Soc.}, 350(5):1713--1790, 1998.

\bibitem[Gor98b]{Gordon:1995}
C.~McA. Gordon.
\newblock Dehn filling: a survey.
\newblock In {\em Knot theory ({W}arsaw, 1995)}, volume~42 of {\em Banach
  Center Publ.}, pages 129--144. Polish Acad. Sci. Inst. Math., Warsaw, 1998.

\bibitem[GT15]{GabaiTrnkova:2015}
David Gabai and Maria Trnkova.
\newblock Exceptional hyperbolic 3-manifolds.
\newblock {\em Comment. Math. Helv.}, 90(3):703--730, 2015.

\bibitem[GW19]{GuillouxWill:2019}
Antonin Guilloux and Pierre Will.
\newblock On {$\rm SL(3,\Bbb C)$}-representations of the {W}hitehead link
  group.
\newblock {\em Geom. Dedicata}, 202:81--101, 2019.

\bibitem[Har20]{Haraway:2020}
Robert~C. Haraway, III.
\newblock Determining hyperbolicity of compact orientable 3-manifolds with
  torus boundary.
\newblock {\em J. Comput. Geom.}, 11(1):125--136, 2020.

\bibitem[HK05]{HodgsonKerckhoff:2005}
Craig~D. Hodgson and Steven~P. Kerckhoff.
\newblock Universal bounds for hyperbolic {D}ehn surgery.
\newblock {\em Ann. of Math. (2)}, 162(1):367--421, 2005.

\bibitem[HK08]{HodgsonKerckhoff:2008}
Craig~D. Hodgson and Steven~P. Kerckhoff.
\newblock The shape of hyperbolic {D}ehn surgery space.
\newblock {\em Geom. Topol.}, 12(2):1033--1090, 2008.

\bibitem[HR99]{holt_rees_1999}
Derek~F. Holt and Sarah Rees.
\newblock Computing with abelian sections of finitely presented groups.
\newblock {\em Journal of Algebra}, 214(2):714–728, 1999.

\bibitem[iee85]{ieee}
IEEE Standard for Binary Floating-Point Arithmetic.
\newblock {\em ANSI/IEEE Std 754-1985}, pages 1--20, 1985.

\bibitem[Lac00]{Lackenby:2000}
Marc Lackenby.
\newblock Word hyperbolic {D}ehn surgery.
\newblock {\em Invent. Math.}, 140(2):243--282, 2000.

\bibitem[Li06]{Li:2006}
Tao Li.
\newblock An algorithm to find vertical tori in small {S}eifert fiber spaces.
\newblock {\em Comment. Math. Helv.}, 81(4):727--753, 2006.

\bibitem[LM13]{LackenbyMeyerhoff:2013}
Marc Lackenby and Robert Meyerhoff.
\newblock The maximal number of exceptional {D}ehn surgeries.
\newblock {\em Invent. Math.}, 191(2):341--382, 2013.

\bibitem[Mae07]{Maehara:2007}
Hiroshi Maehara.
\newblock On configurations of solid balls in 3-space: chromatic numbers and
  knotted cycles.
\newblock {\em Graphs Combin.}, 23(suppl. 1):307--320, 2007.

\bibitem[Mar06]{Martelli:2006}
Bruno Martelli.
\newblock Complexity of 3-manifolds.
\newblock In {\em Spaces of {K}leinian groups}, volume 329 of {\em London Math.
  Soc. Lecture Note Ser.}, pages 91--120. Cambridge Univ. Press, Cambridge,
  2006.

\bibitem[Mat07]{Matveev:2007}
Sergei Matveev.
\newblock {\em Algorithmic topology and classification of 3-manifolds},
  volume~9 of {\em Algorithms and Computation in Mathematics}.
\newblock Springer, Berlin, second edition, 2007.

\bibitem[Mey85]{Meyerhoff:1985}
Robert Meyerhoff.
\newblock The cusped hyperbolic {$3$}-orbifold of minimum volume.
\newblock {\em Bull. Amer. Math. Soc. (N.S.)}, 13(2):154--156, 1985.

\bibitem[Mey86]{Meyerhoff:1986}
Robert Meyerhoff.
\newblock Sphere-packing and volume in hyperbolic {$3$}-space.
\newblock {\em Comment. Math. Helv.}, 61(2):271--278, 1986.

\bibitem[Mil09]{Milley:2009}
Peter Milley.
\newblock Minimum volume hyperbolic 3-manifolds.
\newblock {\em J. Topol.}, 2(1):181--192, 2009.

\bibitem[MP06]{MartelliPetronio:2006}
Bruno Martelli and Carlo Petronio.
\newblock Dehn filling of the ``magic'' 3-manifold.
\newblock {\em Comm. Anal. Geom.}, 14(5):969--1026, 2006.

\bibitem[MPR14]{MartelliPetronioRoukema:2014}
Bruno Martelli, Carlo Petronio, and Fionntan Roukema.
\newblock Exceptional {D}ehn surgery on the minimally twisted five-chain link.
\newblock {\em Comm. Anal. Geom.}, 22(4):689--735, 2014.

\bibitem[NR93]{NeumannReid:1993}
Walter~D. Neumann and Alan~W. Reid.
\newblock Rigidity of cusps in deformations of hyperbolic {$3$}-orbifolds.
\newblock {\em Math. Ann.}, 295(2):223--237, 1993.

\bibitem[NZ85]{NeumannZagier:1985}
Walter~D. Neumann and Don Zagier.
\newblock Volumes of hyperbolic three-manifolds.
\newblock {\em Topology}, 24(3):307--332, 1985.

\bibitem[Per02]{Perelman:2002}
Grisha Perelman.
\newblock The entropy formula for the ricci flow and its geometric
  applications, 2002.

\bibitem[Per03]{Perelman:2003}
Grisha Perelman.
\newblock Ricci flow with surgery on three-manifolds, 2003.

\bibitem[Prz06]{Przeworski:2006}
Andrew Przeworski.
\newblock A universal upper bound on density of tube packings in hyperbolic
  space.
\newblock {\em J. Differential Geom.}, 72(1):113--127, 2006.

\bibitem[Pur08]{Purcell:2008}
Jessica~S. Purcell.
\newblock Cusp shapes under cone deformation.
\newblock {\em J. Differential Geom.}, 80(3):453--500, 2008.

\bibitem[Rol90]{Rolfsen:1990}
Dale Rolfsen.
\newblock {\em Knots and links}, volume~7 of {\em Mathematics Lecture Series}.
\newblock Publish or Perish, Inc., Houston, TX, 1990.
\newblock Corrected reprint of the 1976 original.

\bibitem[Rub04]{Rubinstein:2004}
J.~Hyam Rubinstein.
\newblock An algorithm to recognise small {S}eifert fiber spaces.
\newblock {\em Turkish J. Math.}, 28(1):75--87, 2004.

\bibitem[Som02]{Soma:2002}
Teruhiko Soma.
\newblock Epimorphism sequences between hyperbolic 3-manifold groups.
\newblock {\em Proc. Amer. Math. Soc.}, 130(4):1221--1223, 2002.

\bibitem[Thu78]{Thurston:1978}
William~P. Thurston.
\newblock Geometry and topology of three-manifolds.
\newblock Lecture notes available at \url{http://library.msri.org/books/gt3m/}
  (2021-06-17), 1978.

\bibitem[Thu97]{Thurston:1997}
William~P. Thurston.
\newblock {\em Three-dimensional geometry and topology. {V}ol. 1}, volume~35 of
  {\em Princeton Mathematical Series}.
\newblock Princeton University Press, Princeton, NJ, 1997.
\newblock Edited by Silvio Levy.

\bibitem[Til04]{Tillmann:2000}
Stephan Tillmann.
\newblock Character varieties of mutative 3-manifolds.
\newblock {\em Algebr. Geom. Topol.}, 4:133--149, 2004.

\bibitem[Wie78]{Wielenberg:1978}
Norbert Wielenberg.
\newblock The structure of certain subgroups of the {P}icard group.
\newblock {\em Math. Proc. Cambridge Philos. Soc.}, 84(3):427--436, 1978.

\end{thebibliography}

\end{document}